\documentclass{article}
\usepackage{
amsmath,
amsfonts,
latexsym, 
amsrefs,
amssymb
}

\usepackage{bm}

\usepackage{tocloft}

\usepackage{cite}

\usepackage{graphicx}
\newcommand{\Lnorm}[2]{\left\| #2 \right\|_{{#1}}}

\DeclareMathOperator{\Prob}{\mathbb{P}}   %probability
\newcommand{\1}{\mathds{1}}
\usepackage{enumerate}
\usepackage{mathrsfs}
\usepackage{wrapfig}

\usepackage{color}

\newcommand{\Id}{{\mathrm{Id}}}

\numberwithin{equation}{section}
\newcommand{\bla}{\bm{\lambda}}
\newcommand{\boeta}{\bm{\eta}}
\newcommand{\boxi}{\bm{\xi}}

\newcommand{\boldu}{\mbox{\boldmath $u$}}

\newcommand{\rd}{{\rm d}}

\newcommand{\bx}{{\bf{x}}}
\newcommand{\by}{{\bf{y}}}
\newcommand{\bu}{{\bf{u}}}
\newcommand{\bv}{{\bf{v}}}
\newcommand{\bw}{{\bf{w}}}

\newcommand{\bh}{{\bf{h}}}

\newcommand{\al}{\alpha}

\newcommand{\be}{\begin{equation}}
\newcommand{\ee}{\end{equation}}

\newcommand{\e}{{\varepsilon}}

\newcommand{\la}{\lambda}

\newcommand{\om}{{\omega}}

\newcommand{\LL}{{\rm L}}
\newcommand{\PP}{{\rm P}}
\newcommand{\QQ}{{\rm Q}}

\newcommand{\cN}{{\cal N}}

\newcommand{\rU}{{\rm U}}

\usepackage{amsmath} %[intlimits]
\usepackage{amssymb}
\usepackage{amsthm}

\def\RR{{\mathbb R}}

\renewcommand{\b}[1]{\bm{\mathrm{#1}}} %bold
\renewcommand{\cal}{\mathcal}

\newcommand{\ii}{\mathrm{i}} %\newcommand{\mi}{\mathrm{i}}

\newcommand{\col}{\mathrel{\mathop:}}

\renewcommand{\epsilon}{\varepsilon}
\renewcommand{\leq}{\leqslant}
\renewcommand{\geq}{\geqslant}

\renewcommand{\le}{\leq}
\renewcommand{\ge}{\geq}

\renewcommand{\P}{\mathbb{P}}
\newcommand{\E}{\mathbb{E}}

\newcommand{\C}{\mathbb{C}}
\newcommand{\N}{\mathbb{N}}

\newcommand{\NN}{\mathbb{N}}

\newcommand{\pB}[1]{\Bigl({#1}\Bigr)}

\newcommand{\hb}[1]{\bigl\{{#1}\bigr\}}

\newcommand{\abs}[1]{\lvert #1 \rvert}

\newcommand{\scalar}[2]{\langle{#1} \mspace{2mu}, {#2}\rangle}

\DeclareMathOperator{\var}{Var}

\DeclareMathOperator{\im}{Im}

\DeclareMathOperator{\OO}{O}

\DeclareMathOperator{\UU}{U}

\theoremstyle{plain} %plain, definition, remark
\newtheorem{theorem}{Theorem}[section]
\newtheorem*{theorem*}{Theorem}
\newtheorem{lemma}[theorem]{Lemma}
\newtheorem*{lemma*}{Lemma}
\newtheorem{corollary}[theorem]{Corollary}
\newtheorem*{corollary*}{Corollary}
\newtheorem{proposition}[theorem]{Proposition}
\newtheorem*{proposition*}{Proposition}

\newtheorem{definition}[theorem]{Definition}
\newtheorem*{definition*}{Definition}

\newtheorem*{example*}{Example}
\newtheorem{remark}[theorem]{Remark}

\newtheorem*{remark*}{Remark}
\newtheorem*{remarks*}{Remarks}

\makeatletter
\renewcommand{\subsection}{\@startsection
{subsection}%                   % the name
{2}%                         % the level
{0mm}%                       % the indent
{-\baselineskip}%            % the before skip
{0 \baselineskip}%          % the after skip
{\normalfont\itshape}} % the style
\makeatother

\usepackage{dsfont}
\usepackage{stmaryrd}

\newcommand{\nc}{\normalcolor}

\def\bN{{\mathbb N}}

\newcommand{\Tr}{\mbox{Tr\,}}

\newcommand{\bq}{{\bf q}}

\renewcommand{\b}[1]{\boldsymbol{\mathrm{#1}}} %bold

\def\@empty{}

\def\author#1{\par
    {\centering{\authorfont#1}\par\vspace*{0.05in}}
}

\def\titlefont{\fontsize{13}{15}\bfseries\boldmath\selectfont\centering{}}
\def\authorfont{\fontsize{13}{15}}
\def\abstractfont{\fontsize{8}{10}}

\let\affiliationfont\rhfont

\def\address#1{\par
    {\centering{\affiliationfont#1\par}}\par\vspace*{11pt}
}

\def\keywords#1{\par
    \vspace*{8pt}
    {\authorfont{\leftskip18pt\rightskip\leftskip
    \noindent{\it\small{Keywords}}\/:\ #1\par}}\vskip-12pt}

\def\body{
\setcounter{footnote}{0}
\def\thefootnote{\alph{footnote}}
\def\@makefnmark{{$^{\rm \@thefnmark}$}}
}

\def\title#1{
    \thispagestyle{plain}
    \vspace*{-14pt}
    \vskip 79pt
    {\centering{\titlefont #1\par}}%
    \vskip 1em
}

\renewenvironment{abstract}{\par%
    \vspace*{6pt}\noindent %{\bf Abstract}
    \abstractfont
    \noindent\leftskip18pt\rightskip18pt
}{%
  \par}

\renewcommand{\b}[1]{\boldsymbol{\mathrm{#1}}} %bold
\newcommand{\f}[1]{\boldsymbol{\mathrm{#1}}} 
\newcommand{\for}{\qquad \text{for} \quad}

\usepackage{tocloft}

\makeatletter
\renewcommand{\section}{\@startsection
{section}%                   % the name
{1}%                         % the level
{0mm}%                       % the indent
{-2\baselineskip}%            % the before skip
{1\baselineskip}%          % the after skip
{\normalfont\large\scshape\centering}} % the style
\makeatother

\begin{document}

%~\vspace{-2cm}

\title{The Eigenvector Moment Flow and local Quantum Unique Ergodicity}

\vspace{1.2cm}
\noindent\begin{minipage}[b]{0.5\textwidth}

 \author{P. Bourgade}

\address{Cambridge University and Institute for Advanced Study\\
   E-mail: bourgade@math.ias.edu}
 \end{minipage}
\begin{minipage}[b]{0.5\textwidth}

 \author{H.-T. Yau}

\address{Harvard University and Institute for Advanced Study\\
   E-mail: htyau@math.harvard.edu}

 \end{minipage}

~\vspace{0.3cm}

\begin{abstract}
We prove that  the distribution  of  eigenvectors of generalized Wigner matrices  is universal  both  in the bulk and at the edge. This includes a probabilistic version of local 
quantum unique ergodicity  and asymptotic normality  
of the eigenvector entries. The proof relies on analyzing the eigenvector flow under 
the Dyson Brownian motion. The key new ideas are:  (1) the introduction of the eigenvector moment flow,
a multi-particle random walk in a random environment, (2) an effective estimate on the regularity of this flow 
based on maximum principle  and (3)  optimal  finite speed of  propagation 
holds for  the eigenvector moment flow with very high probability.

\end{abstract}

\keywords{Universality, Quantum unique ergodicity, Eigenvector moment flow.} 
\vspace{.5cm}

\tableofcontents

\let\thefootnote\relax\footnote{\noindent The work of P. B. is partially supported by the NSF grant DMS1208859. The work of H.-T. Y. is partially supported by the NSF grant DMS1307444 and the Simons Investigator Fellowship.}

\newpage

\section{Introduction}

Wigner envisioned that the laws of  the eigenvalues of  large random matrices are new paradigms for  universal statistics  of 
large correlated quantum systems. Although this vision has not been proved for any truly interacting quantum system, 
it is generally considered to be valid for a wide range of models.  For example, the quantum chaos conjecture by 
Bohigas-Giannoni-Schmit \cite{BohGiaSch1984} asserts that the eigenvalue statistics of the Laplace operator on a domain or manifold
are given by the random matrix statistics, provided that the corresponding classical dynamics are chaotic. Similarly, 
one expects that the  eigenvalue statistics of random Schr{\"o}dinger operators (Anderson tight binding models) are  given by the random matrix statistics  
in the delocalization regime. Unfortunately, both  conjectures are far beyond the reach of the current mathematical technology.

In Wigner's original theory,  the eigenvector behaviour plays no role. As suggested by the Anderson model, 
random matrix statistics  coincide with delocalization of eigenvectors. 
A strong notion of delocalization, at least in terms of ``flatness of the eigenfunctions",  is the quantum ergodicity.
For the Laplacian on a negative curved compact Riemannian 
manifold, Shnirel'man \cite{Shn1974}, Colin de Verdi\`ere \cite{Col1985} and  Zelditch \cite{Zel1987} proved that 
quantum ergodicity holds. More precisely, let 
$(\psi_k)_{k\geq 1}$ denote an orthonormal basis of eigenfunctions of the Laplace-Beltrami operator, associated with increasing eigenvalues, on a negative curved manifold 
$\mathcal{M}$ 
(or  more generally, assume only that the  geodesic flow of $\mathcal{M}$ is ergodic)  with volume measure  $\mu$.  Then, 
for any open set $A\subset \mathcal{M}$, one has 
$$
\lim_{\lambda \to \infty  } \frac 1 {N(\lambda)} \sum_{j: \lambda_j \le \lambda} \Big |  \int_A    |\psi_j (x)|^2 \mu(\rd x) - \int_A   \mu(\rd x) \Big |^2 
= 0,
$$
where $N(\lambda) = | \{j:  \lambda_j \le \lambda\}|$.
Quantum ergodicity was also proved for  $d$-regular graphs under certain assumptions  on the injectivity radius and  spectral gap of the adjacency matrices \cite{AnaLeM2013}. Random graphs are
considered a good paradigm for many ideas related to quantum chaos \cite{KotSmi1997}.

An even stronger notion of delocalization is the quantum unique ergodicity conjecture (QUE) proposed by Rudnick-Sarnak \cite{RudSar1994},  
i.e.,  for any negatively curved compact Riemannian manifold $\mathcal{M}$, the eigenstates  become equidistributed with respect to the volume measure $\mu$:  for any open $A\subset \mathcal{M}$ we have
\begin{equation}\label{eqn:QUE}
\int_A|\psi_k(x)|^2\mu(\rd x)\underset{k\to\infty}{\longrightarrow}\int_A\mu(\rd x).
\end{equation}
Some numerical evidence exists for both eigenvalue statistics and the QUE, but
a proper understanding of the semiclassical limit of chaotic systems is still missing. 
One case for which QUE was rigorously proved concerns arithmetic surfaces, 
thanks to tools from number theory and ergodic theory on homogeneous spaces 
\cite{Lin2006,Hol2010,HolSou2010}. For results in the case of general 
compact Riemannian manifolds whose geodesic flow is Anosov, see \cite{Ana2008}.

A major class of matrices for which one expects that  Wigner's vision holds is  the Wigner matrices,  i.e., random matrices with matrix elements distributed by 
identical mean-zero random variables. For this class of matrices, 
the Wigner-Dyson-Mehta conjecture states that the local statistics
are independent of 
the  laws of the matrix elements and depend only on the symmetry class. 
This conjecture was recently solved for an even more general class: the 
generalized Wigner matrices for which the distributions of matrix entries can vary and have 
 different  variances. (See \cite{ErdYauYin2012Univ,EYYBernoulli} and \cite{ErdYau2012} for a review.
For earlier results on this conjecture for Wigner matrices, see 
\cite{ErdSchYau2011,TaoVu2011} for the bulk of the spectrum and  \cite{Sos1999,TaoVu2010,  ErdYauYin2012Rig} for the edge).
One key ingredient of the method initiated in  \cite{ErdSchYau2011}  proceeds by interpolation between Wigner and Gaussian ensembles through Dyson Brownian motion, a matrix process that induces an autonomous evolution of eigenvalues.
The fundamental conjecture for Dyson Brownian motion,  the Dyson conjecture,  states  that the time 
to local equilibrium is of order $ t \gtrsim 1/N$, where $N$ is the size of the matrix. This conjecture was resolved 
in   \cite{ErdYauYin2012Rig}  (see  \cite{ErdSchYau2011} for the earlier results) and is the underlying reason for the universality. 

 Concerning the eigenvectors distribution, 
complete delocalization  was proved in  \cite{ErdYauYin2012Rig}
for generalized Wigner matrices  
 in the following sense
 : with very high  probability 
$$
\max |u_i(\alpha)|\leq \frac{(\log N)^{C \log \log N}}{\sqrt{N}},
$$
 where $C$ is a fixed constant and the maximum ranges over all coordinates $\alpha$ of the $\LL^2$-normalized eigenvectors, $u_1,\dots,u_N$ (a stronger estimate was obtained for Wigner matrices in \cite{ErdSchYau2009}, 
 see also \cite{BroLin2013} for a delocalization bound for the Laplacian on deterministic regular  graphs).
Although this bound prevents concentration of eigenstates onto a set of size less than
$N(\log N)^{- C \log\log N}$, it does not imply 
the ``complete flatness" of type (\ref{eqn:QUE}). In fact, if the eigenvectors are distributed by the 
Haar measure  on the orthogonal group,  the weak convergence 
\begin{equation}\label{eqn:Borel}
\sqrt{N}u_i(\alpha)\to
\mathscr{N}
\end{equation}
holds, where $\mathscr{N}$ is a standard Gaussian random variable and the eigenvector components are asymptotically independent.
Since the eigenvectors of GOE are distributed by the Haar measure  on the orthogonal group, 
this asymptotic normality (\ref{eqn:Borel}) holds for GOE (and a similar statement holds for GUE). 
For Wigner ensembles, by comparing with GOE, this property was proved for eigenvectors in the bulk by Knowles-Yin and Tao-Vu \cite{KnoYin2011, TaoVu2012} 
under the condition 
that the first  four moments of the matrix elements of the  Wigner ensembles match those of the standard normal distribution. 
For eigenvectors near the edges, the matching condition can be reduced to only  the first two moments \cite{KnoYin2011}.

In this paper, we develop a completely new  method to show that this asymptotic normality (\ref{eqn:Borel}) and independence
of eigenvector components hold  for generalized Wigner matrices without any moment matching condition. 
In particular, even the second moments are allowed to vary as long as the matrix stays inside 
the generalized Wigner class.
From the law of large numbers of independent random variables, 
this implies the {\it local  quantum unique ergodicity}, to be specified below,
with  high probability.  In fact, we will prove a stronger form of asymptotic normality in the sense 
that any projection of the eigenvector is asymptotically normal, see Theorem  \ref{thm:main}. 
This can be viewed as the eigenvector universality for the  generalized Wigner ensembles.

The key idea in this new approach is to analyze the ``Dyson eigenvector flow". More precisely, 
the Dyson Brownian motion is induced by the dynamics in which matrix elements undergo independent Brownian motions. 
The same dynamics on matrix elements yield a flow on the eigenvectors. This eigenvector flow, 
which we will call the {\it Dyson eigenvector flow}, was computed   
in the context of Brownian motion on ellipsoids \cite{NorRogWil1986}, real Wishart processes \cite{Bru1989}, and for GOE/GUE in \cite{AndGuiZei2010} (see also \cite{AllGui2013}). This flow is a diffusion process  on a compact Lie group ($O(N)$ or $U(N)$) 
endowed with a Riemannian metric. This diffusion process roughly speaking can be described as follows. 
We first randomly choose two eigenvectors, $u_i$ and $u_j$. Then we randomly rotate these two vectors 
on the circle spanned by them with a rate  $(\lambda_i-\lambda_j)^{-2}$ depending on the eigenvalues. 
Thus the eigenvector flow depends on the eigenvalue dynamics. If we freeze the eigenvalue flow, 
the eigenvector flow is a  diffusion with time dependent singular coefficients depending on the eigenvalues.

Due to its complicated structure, the Dyson eigenvector flow has never been analyzed.
Our key observation is that 
the dynamics of the  moments of the eigenvector entries can be viewed as a \nc 
multi-particle random walk in a random environment.  The number of particles 
of this flow is  one half of  the degree of polynomials in the eigenvector entries,
and   the (dynamic) random environment is  given by  jump rates depending on the eigenvalues. 
We shall call this flow the  {\it eigenvector moment flow}. If  there is only one particle,
this flow is the random walk with the random jump rate  $(\lambda_i-\lambda_j)^{-2}$ between two integer locations $i$ and $j$. 
This one dimensional random walk process   was analyzed locally in \cite{ErdYau2012singlegap}  
for the purpose of the single gap universality between eigenvalues.
An important 
result of \cite{ErdYau2012singlegap} is the H{\"o}lder regularity of the solutions. 
In higher dimensions, the jump rates depend on the locations of nearby particles and the flow is not a simple tensor product of 
the one dimensional process. Fortunately, we find   that this flow is  reversible with respect to an explicit equilibrium measure. 
The H{\"o}lder regularity argument in \cite{ErdYau2012singlegap} can  be extended to any  dimension to prove that 
the solutions of 
the moment flow are locally H{\"o}lder continuous. 
From this result and the local semicircle law (more precisely, the isotropic local semicircle law proved in \cite{KnoYin2012} and \cite{BloErdKnoYauYin2013}), 
one can  obtain that the  bulk eigenvectors  generated by a Dyson eigenvector flow satisfy local quantum unique ergodicity, and the law of the entries  of the eigenvectors 
are  Gaussian.

Instead of showing the H{\"o}lder regularity, we will directly prove that the solution to the eigenvector moment flow  converges 
to a constant. This proof is based on 
a maximum principle for parabolic differential equations and the local isotropic law \cite{BloErdKnoYauYin2013} previously mentioned. 
It yields the convergence of  the eigenvector moment flow to a constant for   $t\gtrsim N^{-1/4}$ with explicit error bound. 
This immediately implies that  all eigenvectors (in the bulk and at the edge)  generated by a Dyson eigenvector flow satisfy local quantum unique ergodicity, and the law of the entries  of the eigenvectors 
are  Gaussian.

The time to equilibrium $t\gtrsim N^{-1/4}$ mentioned  above is not optimal and the correct scaling of relaxation to equilibrium is $t \sim N^{-1}$ 
in the bulk, similar to Dyson's conjecture for  relaxation of bulk eigenvalues  to local equilibrium. In other words, we expect that Dyson's conjecture 
can be extended to the eigenvector flow bulk as well.  We will give a positive answer to this question in
Theorem \ref{thm:deterministic}. A key tool in proving this theorem is a finite speed of propagation estimate for the eigenvector moment flow. 
An estimate of this type   was     first proved  in \cite[Section 9.6]{ErdYau2012singlegap},  
but it requires a difficult level repulsion estimate.  In Section 6, we will prove an optimal  finite speed of propagation 
estimate without using any level repulsion estimate.

In order to prove that the eigenvectors of the original matrix ensemble satisfy quantum ergodicity, it remains to approximate the Wigner matrices by 
Gaussian convoluted ones,  i.e., matrices that are a small time solution to the Dyson Brownian motion. We invoke the Green function comparison theorem in a version similar to the one stated in \cite{KnoYin2011}.
For bulk eigenvectors, we can remove this small Gaussian component by a continuity principle instead of the Green function comparison theorem:
we will show that the Dyson Brownian motion preserves the detailed behavior
of eigenvalues and eigenvectors  up to time $N^{-1/2}$ directly by using the It{\^o} formula. 
This approach is much more direct and there is no need to construct moment matching matrices.

The  eigenvector moment flow developed in this paper can be applied to other random matrix models.   For example, the local  quantum unique ergodicity holds for covariance matrices (for the associated flow and results, see Appendix C)  and a certain class of  Erd{\H o}s-R{\' e}nyi graphs.  To avoid other technical issues, 
 in this paper  we only consider generalized Wigner matrices. 
Before stating the results and giving more details about the proof, we recall the definition of the considered ensemble.\\

\begin{definition}\label{def:wig}
A generalized Wigner matrix $H_N$ is an Hermitian or symmetric $N\times N$ matrix whose upper-triangular matrix elements $h_{ij}=\overline{h_{ji}}$, $i\leq j$, are independent random variables with mean zero and variance $\sigma_{ij}^2=\E(|h_{ij}|^2)$ satisfying the following additional two conditions:
\begin{enumerate}[(i)]
\item Normalization: for any $j\in\llbracket 1,N\rrbracket$, $\sum_{i=1}^N\sigma_{ij}^2=1$.
\item Non-degeneracy: there exists a constant $C$, independent of $N$, such that $C^{-1}N^{-1}\leq \sigma_{ij}^2\leq C N^{-1}$ for all $i,j\in\llbracket 1,N\rrbracket$. In the Hermitian case, we furthermore
assume that, for any $i<j$, $\E((\bh_{ij})^*\bh_{ij})\geq c N^{-1}$ in the sense of inequality between 
$2\times2$ positive matrices, where $\bh_{ij}=(\Re(h_{ij}),\Im(h_{ij}))$.
\end{enumerate}
\end{definition}

\noindent Moreover, we assume that all moments of the entries are finite:
for any $p\in\mathbb{N}$ there exists a constant $C_p$ such that for any $i,j,N$ we have
\begin{equation}\label{eqn:moments}
\E(|\sqrt{N}h_{ij}|^p)<C_p.
\end{equation}
In the following, $(u_i)_{i=1}^N$ denotes  an orthonormal
eigenbasis for $H_N$, a matrix from the (real or complex) generalized Wigner ensemble. The eigenvector $u_i$ is associated with the eigenvalue  $\la_i$, where $\la_1\leq\dots\leq \la_N$.

\begin{theorem}\label{thm:main}
Let $(H_N)_{N\geq 1}$ be a sequence of generalized Wigner matrices. Then there is a $\delta> 0$ such that for any 
 $m\in\mathbb{N}$,  $I\subset \mathbb{T}_N:=  \llbracket 1,N^{1/4} \rrbracket  \cup 
\llbracket  N^{1-\delta} , N- N^{1-\delta} \rrbracket \cup \llbracket N-N^{1/4},N\rrbracket$ with  $|I|=m$ and  for any 
 unit vector $\bq$  in $\RR^N$, we have 
\begin{align}
&\sqrt{N}(|\langle \bq,u_{k}\rangle|)_{k\in I}\to(|\mathscr{N}_j|)_{j=1}^m\hspace{2.9cm}\mbox{in the symmetric case},\label{eqn:convG}\\
&\sqrt{2N}(|\langle \bq,u_{k}\rangle|)_{k\in I}\to(|\mathscr{N}^{(1)}_j+\ii\mathscr{N}^{(2)}_j|)_{j=1}^m\hspace{1cm}\mbox{in the Hermitian case,}\notag
\end{align}
in the sense of convergence of moments, where all $\mathscr{N}_j,\mathscr{N}^{(1)}_j,\mathscr{N}^{(2)}_j$, are independent standard Gaussian random variables.
This convergence holds uniformly in $I$ and  $|\bq|=1$.
More precisely, 
for any polynomial $P$ in $m$ variables, there exists $\e=\e(P)>0$ such that for large enough $N$ we have
\begin{align}\label{eqn:mainSym}
&\sup_{I\subset  \mathbb{T}_N\nc,|I|=m,|\bq|=1}
\left|\E\left(P\left(\left(N|\langle \bq,u_{k}\rangle|^2\right)_{k\in I}\right) \right)-\E \left(P\left((|\mathscr{N}_j|^2)_{j=1}^m\right)\right)\right|\leq N^{-\e},\\
&\sup_{I\subset \mathbb{T}_N\nc,|I|=m,|\bq|=1}\notag
\left|\E\left(P\left(\left(2N|\langle \bq,u_{k}\rangle|^2\right)_{k\in I}\right)  \right)-\E P\left((|\mathscr{N}^{(1)}_j|^2+|\mathscr{N}^{(2)}_j|^2)_{j=1}^m)\right)\right|\leq N^{-\e},
\end{align}
respectively for the real and complex generalized Wigner ensembles. 
\end{theorem}

The restriction on the eigenvector in the immediate regime $\llbracket N^{1/4} ,  N^{1-\delta} \rrbracket$  (and similarly for its reflection) 
was due to that near  the edges the level repulsion estimate, Definition \ref{lr},  (or the gap universality) was only written  in the region  $\llbracket 1,N^{1/4} \rrbracket$
(see the discussion after Definition \ref{lr} for references regarding this matter).
There is no doubt that these results can be extended to the the immediate regime with only minor  modifications in the proofs. 
Here we state our theorem based on existing written results.

The normal convergence (\ref{eqn:convG}) was proved in \cite{TaoVu2012} under the assumption that the entries of $H_N$ have moments matching the standard Gaussian distribution up to order four, and if their distribution is symmetric (in particular the fifth moment vanishes).

 This convergence of moments implies in particular joint weak convergence.
Choosing $\bq$ to be an element of the canonical basis, Theorem \ref{thm:main} implies in particular that any entry of an  eigenvector is asymptotically normally distributed, modulo the (arbitrary) phase choice. Because the above convergence holds for any $|\bq|=1$, 
asymptotic joint normality of the eigenvector entries also holds.
Since eigenvectors are defined only up to a phase, 
we define the 
equivalence relation $u\sim v$ if  $u=\pm v$ in the symmetric case and $u=\lambda v$ for some $|\lambda|=1$
in the Hermitian case.

\begin{corollary}[Asymptotic normality  of eigenvectors for generalized Wigner matrices]\label{cor:Gauss}
Let $(H_N)_{N\geq 1}$ be a sequence of generalized Wigner matrices, $\ell\in\mathbb{N}$. Then for any  $k\in  \mathbb{T}_N$ and  $J\subset \llbracket 1 ,  N\rrbracket $ with 
$|J|=\ell$, we have
\begin{align*}
&\sqrt{N}(u_k(\alpha))_{\alpha\in J}
\to(\mathscr{N}_j)_{j=1}^\ell\hspace{2.4cm}\mbox{for the real generalized Wigner ensemble},\\
&\sqrt{2N}(u_k(\alpha))_{\alpha\in J}
\to(\mathscr{N}^{(1)}_j+\ii\mathscr{N}^{(2)}_j)_{j=1}^\ell\hspace{0.5cm}\mbox{for the complex generalized Wigner ensemble,}
\end{align*}
in the sense of convergence of moments  modulo $\sim$, where all $\mathscr{N}_j, \mathscr{N}^{(1)}_j,
\mathscr{N}^{(2)}_j$, are independent standard Gaussian variables.
 More precisely,  \nc 
for any polynomial $P$ in $\ell$ variables (resp. $Q$ in $2\ell$ variables) there exists $\e$ 
depending on $P$ (resp. $Q$) such that, for large enough $N$,
\begin{align}
&\underset{J\subset\llbracket 1,N\rrbracket,|J|=\ell,\atop k\in   \mathbb{T}_N}{\sup}\label{eqn:corSym}
\left|\E\left(P\left(\sqrt{N}(e^{\ii\omega}u_k(\alpha))_{\alpha\in J}\right)  \right)-\E P\left((\mathscr{N}_j)_{j=1}^\ell\right)\right|\leq N^{-\e},\\
&\underset{ J\subset\llbracket 1,N\rrbracket,|J|=\ell,\atop k\in   \mathbb{T}_N}{\sup}
\left|\E\left(Q\left(\sqrt{2N}(e^{\ii\omega}u_k(\alpha),e^{-\ii\omega}\overline{u_k(\alpha)})_{\alpha\in J}\right) \right)-\E Q\left((\mathscr{N}^{(1)}_j+\ii\mathscr{N}^{(2)}_j,\mathscr{N}^{(1)}_j-\ii\mathscr{N}^{(2)}_j)_{j=1}^\ell\right)\right|\leq N^{-\e},\notag
\end{align}
for the symmetric (resp. Hermitian) generalized Wigner ensembles. Here $\omega$ is independent of $H_N$ and
uniform on the binary set  $\{0,\pi\}$ (resp. $(0,2\pi)$). 
\end{corollary}

By characterizing the joint distribution of the entries of the eigenvectors, Theorem $\ref{thm:main}$ and Corollary \ref{cor:Gauss} imply that for any eigenvector a probabilistic equivalent of $(\ref{eqn:QUE})$ holds.
For $a_N:\llbracket 1,N\rrbracket\to[-1,1]$ we denote $|a_N|=
|\{1\leq \alpha\leq N:a_N(\alpha)\neq 0\}|$ the size of the support of $a_N$, and
$\langle u_k,a_N u_k\rangle=\sum|u_k(\alpha)|^2 a_N(\alpha)$.

\begin{corollary}[Local quantum unique ergodicity for generalized Wigner matrices]\label{cor:QUE}
Let $(H_N)_{N\geq 1}$ be a sequence of generalized (real or complex) Wigner matrices.
Then there exists $\e>0$ such that for any
$\delta>0$, there exists $C>0$
such that the following holds: for any
 $(a_N)_{N\geq 1}$, $a_N:\llbracket 1,N\rrbracket\to [-1,1]$ with $\sum_{\al=1}^N a_N(\alpha)=0$
and $k\in   \mathbb{T}_N$, 
we have
\be\label{localque}
\Prob\left(\left|\frac{N}{|a_N|}\langle u_k,a_N u_k \rangle\right|>\delta\right)\leq C\, \left(N^{-\e}+|a_N|^{-1}\right).
\ee
\end{corollary}

Under the condition that  the first  four moments of the matrix elements of the  Wigner ensembles match those of the standard normal distribution, 
\eqref{localque} can also be proved  from the results in  \cite{KnoYin2011, TaoVu2012};  the four moment matching were  reduced 
to  two moments for eigenvectors near the edges \cite{KnoYin2011}.

The quantum ergodicity for a class of sparse regular  graphs was proved  by Anantharaman-Le Masson \cite{AnaLeM2013}, 
partly based  on pseudo-differential calculus on graphs from \cite{LeM2013}. 
The main result in 
 \cite{AnaLeM2013} is for deterministic graphs, but for the purpose of this paper 
we only state its application to random graphs (see \cite{AnaLeM2013} for details and more general statements).
If $u_1,\dots,u_N$ are the ($\LL^2$-normalized) eigenvectors  of the discrete Laplacian
of a uniformly chosen $(q+1)$-regular graph with $N$ vertices, then for any fixed $\delta>0$ 
we have,  for any $q \ge 1$ fixed, 
$$
\Prob\left( \sharp \{ k: |\langle u_k, a_N u_k\rangle| >\delta\} > \delta N\right) \underset{N\to\infty}{\longrightarrow} 0,
$$
where $a_N$ may be random (for instance, it may depend on the graph).
The results in  \cite{AnaLeM2013} were focused on very sparse deterministic regular graphs and are very different from 
our setting for generalized Wigner matrices.

Notice that our result \eqref{localque} allows the test function to have a very small support
and it is valid for any $k$.  
This means that eigenvectors are flat even in ``microscopic scales".  However, 
the equation  \eqref{localque} does not imply that all eigenvectors are completely flat simultaneously 
with high probability, i.e., we have not proved  
the following statement:
$$
\Prob\left(\sup_{1\leq k\leq N}|\langle u_k,a_N u_k\rangle|>\delta\right)\to 0
$$
for $a_N$ with   support of order $N$.  
This strong form of QUE, however,  holds for the Gaussian ensembles.\\

In the following section, we will  define the Dyson vector flow and,
for the sake of completeness, prove the well-posedness of the eigenvector stochastic evolution.
In Section 3 we will  introduce the eigenvector moment flow
and prove the existence of an explicit reversible measure.
In Section 4, we  will prove  Theorem \ref{thm:main} under the additional assumption that $H_N$ is the sum 
of a generalized Wigner matrix and a Gaussian matrix with small variance. 
The proof in this section relies on a maximum principle for the eigenvector moment flow. 
We will prove  Theorem \ref{thm:main}  by using  a Green function comparison theorem in Section 5.
In Section 6, we will prove that the speed of propagation for  the eigenvector moment flow is  finite  with very high  probability. 
This estimate will enable us to prove in Section 7 that the 
relaxation to equilibrium for  the eigenvector moment flow  in the bulk  is of order  $t\gtrsim N^{-1}$. 
The appendices contain a continuity estimate for the Dyson Brownian motion up to  time $N^{-1/2}$, and  some basic  results concerning the generator of  the Dyson vector flow as well as analogue results for covariance matrices.

\section{Dyson Vector Flow}

In this section, we  first state  the stochastic differential equation for the eigenvectors 
under the Dyson Brownian motion. This evolution is 
given by (\ref{eqn:eigenvectorsSymmetric}) and (\ref{eqn:eigenvectorsHermitian}). We then give a concise form of the generator for this Dyson vector flow.
We will follow  the usual slight ambiguity of terminology  by naming both the matrix flow and the eigenvalue flow a Dyson Brownian motion. In case we wish to distinguish them, we will use matrix Dyson Brownian motion for the matrix flow.

\begin{definition}\label{def:OU}
Hereafter is our choice of normalization for the Dyson Brownian motion.
\begin{enumerate}[(i)]
\item
Let $B^{(s)}$ be a $N\times N$ matrix such that $B^{(s)}_{ij} (i<j)$ and $B^{(s)}_{ii}/\sqrt{2}$ are independent standard Brownian motions, and $B^{(s)}_{ij}=B^{(s)}_{ji}$.
The $N\times N$ symmetric Dyson Brownian motion $H^{(s)}$ with initial value $H^{(s)}_0$ is defined as
\begin{equation}\label{eqn:SymSDE}
H^{(s)}_{t}=H^{(s)}_0+\frac{1}{\sqrt{N}} B^{(s)}_t,
\end{equation}
\item Let $B^{(h)}$ be a $N\times N$ matrix such that $\Re(B^{(h)}_{ij}),\Im(B^{(h)}_{ij})  (i<j)$ and $B^{(h)}_{ii}/\sqrt{2}$ are independent standard Brownian motions, and $B^{(h)}_{ji}=
(B^{(h)}_{ij})^*$. The $N\times N$ Hermitian Dyson Brownian motion $H^{(h)}$ with initial value $H^{(t)}_0$
is
$$
H^{(h)}_{t}=H^{(h)}_{0}+\frac{1}{\sqrt{2N}}B^{(h)}_t,
$$
\end{enumerate}
\end{definition}

\begin{definition}We refer to the following stochastic differential equations 
as the Dyson Brownian motion for (\ref{eqn:eigenvaluesSymmetric}) and (\ref{eqn:eigenvaluesHermitian}) 
and
the Dyson vector flow for (\ref{eqn:eigenvectorsSymmetric}) and (\ref{eqn:eigenvectorsHermitian}).
\begin{enumerate}[(i)]
\item Let  $\bla_0\in\Sigma_N=\{\lambda_1<\dots<\lambda_N\}$, $\boldu_0\in \OO(N)$, and $B^{(s)}$
be as in Definition \ref{def:OU}. The symmetric Dyson Brownian motion/vector flow  with initial condition  $(\lambda_1,\dots,\lambda_N)=\bla_0$, $(u_1,\dots,u_N)=\boldu_0$,  is
\begin{align}
\rd\la_k&=\frac{\rd B^{(s)}_{kk}}{\sqrt{N}}+\left(\frac{1}{N}\sum_{\ell\neq k}\frac{1}{\la_k-\la_\ell}\right)\rd t\label{eqn:eigenvaluesSymmetric},\\
\rd u_k&=\frac{1}{\sqrt{N}}\sum_{\ell\neq k}\frac{\rd B^{(s)}_{k\ell}}{\lambda_k-\lambda_\ell}u_\ell
-\frac{1}{2N}\sum_{\ell\neq k}\frac{\rd t}{(\la_k-\la_\ell)^2}u_k\label{eqn:eigenvectorsSymmetric}.
\end{align}
\item Let  $\bla_0\in\Sigma_N$, $\boldu_0\in \UU(N)$, and $B^{(h)}$
be as in Definition \ref{def:OU}. The Hermitian Dyson Brownian motion/vector flow with initial condition  $(\lambda_1,\dots,\lambda_N)=\bla_0$, $(u_1,\dots,u_N)=\boldu_0$,  is
\begin{align}
\rd\la_k&=\frac{\rd B^{(h)}_{kk}}{\sqrt{2N}}+\left(\frac{1}{N}\sum_{\ell\neq k}\frac{1}{\la_k-\la_\ell}\right)\rd t\label{eqn:eigenvaluesHermitian},\\
\rd u_k&=\frac{1}{\sqrt{2N}}\sum_{\ell\neq k}\frac{\rd B^{(h)}_{k\ell}}{\lambda_k-\lambda_\ell}u_\ell
-\frac{1}{2N}\sum_{\ell\neq k}\frac{\rd t}{(\la_k-\la_\ell)^2}u_k\label{eqn:eigenvectorsHermitian}.
\end{align}
\end{enumerate}
\end{definition}

The  theorem below  contains the following results. (a) The above stochastic differential equations admit a unique strong solution, this relies on classical techniques and an argument originally by McKean \cite{McK1969}. (b) The matrix Dyson Brownian motion induces the standard    Dyson Brownian motion (for the eigenvalues)  and  Dyson
eigenvector flow.
This statement was already proved in \cite{AndGuiZei2010}.  (c) For calculation purpose, one can condition on the trajectory of the eigenvalues to study the eigenvectors evolution. For the sake of completeness, this theorem is proved in the appendix.

With a slight abuse of notation, we will write $\bla_t$
either for $(\lambda_1(t),\dots,\lambda_N(t))$ or for the $N\times N$ diagonal matrix with entries 
$\lambda_1(t),\dots,\lambda_N(t)$.

\begin{theorem}\label{thm:PCE}
The following statements about the  Dyson Brownian motion and eigenvalue/vector flow hold.
\begin{enumerate}[(a)]
\item Existence and strong uniqueness hold for the  system of stochastic differential equations (\ref{eqn:eigenvaluesSymmetric}), (\ref{eqn:eigenvectorsSymmetric}).
Let $(\bla_t,\boldu_t)_{t\geq 0}$  be the solution. Almost surely, for any $t\geq 0$ we have $\bla_t\in\Sigma_N$ and $\boldu_t\in\OO(N)$.
\item Let $(H_t)_{t\geq 0}$ be
a symmetric Dyson Brownian motion with initial condition $H_0=\boldu_0\bla_0\boldu_0^*$, $\bla_0\in\Sigma_N$.
Then the processes $(H_t)_{t\geq 0}$ and $(\boldu_t\bla_t\boldu_t^*)_{t\geq 0}$ have the same distribution.
\item
Existence and strong uniqueness hold for (\ref{eqn:eigenvaluesSymmetric}). For any $T>0$, let $\nu_T^{H_0}$ be the distribution of $(\bla_t)_{0\leq t\leq T}$ with initial value the spectrum of a matrix $H_0$.
For $0\leq T\leq T_0$ and any given continuous trajectory $\bla=(\bla_t)_{0\leq t\leq T_0}\subset\Sigma_N$, existence and strong uniqueness holds for
(\ref{eqn:eigenvectorsSymmetric}) on $[0,T]$. Let $\mu^{H_0, \bla}_T$ be the distribution of
$(\boldu_t)_{0\leq t\leq T}$ with the initial  matrix $H_0$ and the path $\bla$  given. 

Let $F$ be continuous bounded, from the set of continuous paths (on $[0,T]$) on $N\times N$ symmetric matrices to $\RR$. 
Then for any initial matrix $H_0$ we have
\begin{equation}\label{eqn:dual}
\E^{H_0}(F((H_t)_{0\leq t\leq T}))=\int \rd \nu_{T}^{H_0}(\bla)  \int\rd\mu_T^{H_0, \bla}(\boldu)F((\boldu_t\bla_t\boldu_t^*)_{0\leq t\leq T}).
\end{equation}
\end{enumerate}
The  analogous statements hold in the Hermitian setting.
\end{theorem}

We will omit the subscript $T$ when it is obvious.
The previous theorem reduces the study of the eigenvector dynamics to the stochastic differential equations (\ref{eqn:eigenvectorsSymmetric}) and (\ref{eqn:eigenvectorsHermitian}).
The following lemma gives a concise form of the generators of these diffusions. It is very similar 
to the well-known forms of the generator for the Brownian motion on the unitary/orthogonal groups up to the following difference: weights vary depending on eigenvalue pairs.

We will need the following notations (the dependence in $t$ will often be omitted for $c_{k\ell}$, $1\leq k<\ell\leq N$):
\begin{align}
c_{k\ell}(t)&=\frac{1}{N(\la_k(t)-\la_\ell(t))^2},\label{eqn:cij}\\
u_k\partial_{u_\ell}&=\sum_{\al=1}^Nu_k(\al)\partial_{u_\ell(\al)},\notag\ u_k\partial_{\overline{u}_\ell}=\sum_{\al=1}^Nu_k(\al)\partial_{\overline{u}_\ell(\al)},\notag\\
X_{k\ell}^{(s)}&=u_k\partial_{u_\ell}-u_\ell\partial_{u_k},\label{eqn:Xkl}\\
X_{k\ell}^{(h)}&=u_k\partial_{u_\ell}-\overline{u}_\ell\partial_{\overline{u}_k},\notag\
\overline{X}_{k\ell}^{(h)}=\overline{u}_k\partial_{\overline{u}_\ell}-u_\ell\partial_{u_k}.\notag
\end{align}
Here  $\partial_{\overline{u}_\ell}$ and $\partial_{\overline{u}_\ell}$ are defined by considering 
$u_\ell$ as a complex number, i.e., if we write $u_\ell= x+ iy$ then 
$\partial_{\overline{u}_\ell}= \frac 1 2 \partial_x + \frac  i 2 \partial_y $.

\begin{lemma}\label{lem:generator}
For the diffusion (\ref{eqn:eigenvectorsSymmetric}) (resp. (\ref{eqn:eigenvectorsHermitian})),
the generators acting on smooth functions $f((u_i(\al))_{1\leq i,\al\leq N}):\RR^{N^2}\to\RR$ (resp. $\mathbb{C}^{N^2}\to\RR$) are respectively
\begin{align}
&\LL_t^{(s)}=\sum_{1\leq k<\ell\leq N}c_{k\ell}(t)(X_{k\ell}^{(s)})^2, \notag\\
&\LL_t^{(h)}=\frac{1}{2}\sum_{1\leq k<\ell\leq N}c_{k\ell}(t)\left(X_{k\ell}^{(h)}\overline{X}_{k\ell}^{(h)}+\overline{X}_{k\ell}^{(h)}X_{k\ell}^{(h)}\right).  \label{LH}
\end{align}
\end{lemma}
\noindent The above lemma means
$\rd\E( g (\boldu_t))/\rd t=\E(\LL_t^{(s)}g (\boldu_t))$ (resp.
$\rd\E( g (\boldu_t))/\rd t=\E(\LL_t^{(h)} g(\boldu_t))$)
for the stochastic differential equations (\ref{eqn:eigenvectorsSymmetric}) (resp. (\ref{eqn:eigenvectorsHermitian})). It relies on a direct calculation via It{\^o}'s formula. The details are given in the appendix.

\section{Eigenvector Moment Flow}

\subsection{The moment flow.\ }  Our observables will be moments of projections of the eigenvectors  onto  a  given direction. More precisely,
for any  fixed $\bq \in\mathbb{R}^N$ and  for any $1\leq k\leq N$, define \nc
$$
z_k(t)=\sqrt{N}\langle\bq, u_k(t)\rangle=\sum_{\alpha=1}^N \bq(\alpha) u_k(t, \alpha).
$$
With this $\sqrt{N}$ normalization, the typical size of $z_k$ is of order $1$.
We assume that the eigenvalue trajectory $(\lambda_{k}(t),0\leq t\leq T_0)_{k=1}^N$  in the simplex $\Sigma^{(N)}$  is given. Furthermore,   $\boldu$ is the unique strong solution of the stochastic differential equation (\ref{eqn:eigenvectorsSymmetric}) (resp. (\ref{eqn:eigenvectorsHermitian})) with the given eigenvalue trajectory. \nc
Let $\PP^{(s)}(t)  =\PP^{(s)}(z_1,\dots,z_N)(t) $ and $\PP^{(h)}=\PP^{(h)}(z_1,\dots,z_N)(t)$ 
be smooth functions. 
Then a simple calculation yields
\begin{align}
&X^{(s)}_{k\ell}\PP^{(s)}=(z_k\partial_{z_\ell}-z_\ell\partial_{z_k})\PP^{(s)},\label{eqn:genSym}\\
&X_{k\ell}^{(h)}\PP^{(h)}=(z_k\partial_{z_\ell}-\overline{z}_\ell\partial_{\overline{z}_k})f,\
\overline{X}^{(h)}_{k\ell}\PP^{(h)}=(\overline{z}_k\partial_{\overline{z}_\ell}-z_\ell\partial_{z_k})\PP^{(h)}.\label{eqn:genHer}
\end{align}
For $m\in\llbracket 1,N\rrbracket$, denote by $j_1,\dots,j_m$
positive integers and let  $i_1,\dots,i_m$ in $\llbracket 1,N\rrbracket$ be $m$ distinct indices. The test functions we will consider  are:
\begin{align*}
&{\PP^{(s)}}_{i_1,\dots,i_m}^{j_1,\dots,j_m}(z_1,\dots,z_N)=\prod_{\ell=1}^mz_{i_\ell}^{2j_\ell},\\
&{\PP^{(h)}}_{i_1,\dots,i_m}^{j_1,\dots,j_m}(z_1,\dots,z_N)=\prod_{\ell=1}^mz_{i_\ell}^{j_\ell}\overline{z}_{i_\ell}^{j_\ell}.
\end{align*}
For any $m$ fixed, linear combinations of such polynomial functions are stable under the action of the generator. 
More precisely, the following  formulas hold.

\begin{enumerate}[(i)]
\item In the symmetric setting, one can use (\ref{eqn:genSym}) to evaluate the action of the generator. If 
neither $k$ nor $\ell$ are in $\{i_1,\dots,i_m\}$, then $(X_{k\ell}^{(s)})^2 {\PP^{(s)}}_{i_1,\dots,i_m}^{j_1,\dots,j_m}=0$;   the other cases are covered by:
\begin{align*}
(X_{i_1\ell}^{(s)})^2 {\PP^{(s)}}_{i_1,\dots,i_m}^{j_1,\dots,j_m}&=
2j_1(2j_1-1){\PP^{(s)}}_{\ell,i_1,\dots,i_m}^{1,j_1-1,\dots,j_m}-2j_1{\PP^{(s)}}_{i_1,\dots,i_m}^{j_1,\dots,j_m}\ \  \mathrm{when} \ \ \ell\not\in\{i_1,\dots,i_m\},\\
(X_{i_1i_2}^{(s)})^2 {\PP^{(s)}}_{i_1,\dots,i_m}^{j_1,\dots,j_m}&=
2j_1(2j_1-1){\PP^{(s)}}_{i_1,\dots,i_m}^{j_1-1,j_2+1,\dots,j_m}
+
2j_2(2j_2-1){\PP^{(s)}}_{i_1,\dots,i_m}^{j_1+1,j_2-1,\dots,j_m}\\
&\ \ \ -(2j_1(2j_2+1)+2j_2(2j_1+1)){\PP^{(s)}}_{i_1,i_2,\dots,i_m}^{j_1,j_2,\dots,j_m}.
\end{align*}
\item In the Hermitian setting, we note that the polynomials $\PP^{(h)}$ are invariant under 
the  permutation  $z_i \to \overline{z}_i$.   Thus the action of the generator $\LL^{(h)}$ \eqref{LH} on such functions $\PP^{(h)}$
 simplifies to
$$
\LL^{(h)}_t\PP^{(h)}=\sum_{k<\ell}c_{k\ell}X_{k\ell}^{(h)}\overline{X}_{k\ell}^{(h)}\PP^{(h)}.
$$
Then (\ref{eqn:genHer}) yields
\begin{align*}
X_{i_1\ell}^{(h)}\overline{X}_{k\ell}^{(h)} {\PP^{(h)}}_{i_1,\dots,i_m}^{j_1,\dots,j_m}&=
j_1^2{\PP^{(h)}}_{\ell,i_1,\dots,i_m}^{1,j_1-1,\dots,j_m}-j_1{\PP^{(h)}}_{i_1,\dots,i_m}^{j_1,\dots,j_m} \ \ \mathrm{when} \ \ \ell\not\in\{i_1,\dots,i_m\},\\
X_{i_1i_2\ell}^{(h)}\overline{X}_{k\ell}^{(h)} {\PP^{(h)}}_{i_1,\dots,i_m}^{j_1,\dots,j_m}&=
j_1^2{\PP^{(h)}}_{i_1,\dots,i_m}^{j_1-1,j_2+1,\dots,j_m}
+
j_2^2{\PP^{(h)}}_{i_1,\dots,i_m}^{j_1+1,j_2-1,\dots,j_m}\\
&-(j_1(j_2+1)+j_2(j_1+1)){\PP^{(h)}}_{i_1,i_2,\dots,i_m}^{j_1,j_2,\dots,j_m}
\end{align*}
\end{enumerate}
We now
normalize the polynomials by  defining
\begin{align}
&{\QQ^{ (s)}_{ t}}_{ i_1,\dots,i_m}^{j_1,\dots,j_m}= {\PP^{(s)}}_{i_1,\dots,i_m}^{j_1,\dots,j_m}  (t) \label{eqn:rescaleSym}
\prod_{\ell=1}^m a(2j_\ell)^{-1}\ \mbox{where}\ a(n)=\prod_{k\leq n,k\ {\rm odd}}k,\\
&{\QQ^{(h)}_{ t}}_{ i_1,\dots,i_m}^{j_1,\dots,j_m}= {\PP^{(h)}}_{i_1,\dots,i_m}^{j_1,\dots,j_m}(t) \label{eqn:rescaleHer}
\prod_{\ell=1}^m (2^{j_\ell}j_\ell!)^{-1}.
\end{align}
Note that $a(2n)=\E(\mathscr{N}^{2n})$ and $2^n n!=\E(|\mathscr{N}_1+\ii\mathscr{N}_2|^{2n})$, with 
$\mathscr{N}$, $\mathscr{N}_1$, $\mathscr{N}_2$ independent standard Gaussian random variables.
The above discussion implies the following evolution of  $\QQ^{(s)}$ (resp. $\QQ^{(h)}$)
along the Dyson eigenvector flow (\ref{eqn:eigenvectorsSymmetric}) (resp. (\ref{eqn:eigenvectorsHermitian})).

\begin{enumerate}[(i)]
\item Symmetric case: $  \LL^{(s)}_t{\QQ^{  (s)}_{  t}}=\sum_{k<\ell}c_{k\ell}(X_{k\ell}^{(s)})^2{\QQ^{  (s)}_{ t}}$ where 
\begin{align*}
(X_{i_1\ell}^{(s)})^2 {\QQ^{  (s)}_{ t}}_{i_1,\dots,i_m}^{j_1,\dots,j_m}&=
2j_1{\QQ^{  (s)}_{ t}}_{\ell,i_1,\dots,i_m}^{1,j_1-1,\dots,j_m}-2j_1{\QQ^{  (s)}_{ t}}_{i_1,\dots,i_m}^{j_1,\dots,j_m}\ \  \mathrm{when} \ \ \ell\not\in\{i_1,\dots,i_m\},\\
(X_{i_1i_2}^{(s)})^2 {\QQ^{  (s)}_{ t}}_{i_1,\dots,i_m}^{j_1,\dots,j_m}&=
2j_1(2j_2+1){\QQ^{  (s)}_{ t}}_{i_1,\dots,i_m}^{j_1-1,j_2+1,\dots,j_m}
+
2j_2(2j_1+1){\QQ^{  (s)}_{ t}}_{i_1,\dots,i_m}^{j_1+1,j_2-1,\dots,j_m}\\
&\ \ \ -(2j_1(2j_2+1)+2j_2(2j_1+1)){\QQ^{  (s)}_{ t}}_{i_1,i_2,\dots,i_m}^{j_1,j_2,\dots,j_m}.
\end{align*}
\item Hermitian case: $\LL^{(h)}_t{\QQ^{   (h)}_{  t}}=\sum_{k<\ell}c_{k\ell}X_{k\ell}^{(h)}\overline{X}_{k\ell}^{(h)}{\QQ^{   (h)}_{  t}}$ where 
\begin{align*}
X_{i_1\ell}^{(h)}\overline{X}_{i_1\ell}^{(h)} {\QQ^{   (h)}_{  t}}_{i_1,\dots,i_m}^{j_1,\dots,j_m}&=
j_1{\QQ^{   (h)}_{  t}}_{\ell,i_1,\dots,i_m}^{1,j_1-1,\dots,j_m}-j_1{\QQ^{   (h)}_{  t}}_{i_1,\dots,i_m}^{j_1,\dots,j_m}\ \  \mathrm{when} \ \ \ell\not\in\{i_1,\dots,i_m\},\\
X_{i_1i_2}^{(h)}\overline{X}_{i_1i_2}^{(h)}{\QQ^{   (h)}_{  t}}_{i_1,\dots,i_m}^{j_1,\dots,j_m}&=
j_1(j_2+1){\QQ^{   (h)}_{  t}}_{i_1,\dots,i_m}^{j_1-1,j_2+1,\dots,j_m}
+
j_2(j_1+1){\QQ^{   (h)}_{  t}}_{i_1,\dots,i_m}^{j_1+1,j_2-1,\dots,j_m}\\
&\ \ \ -(j_1(j_2+1)+j_2(j_1+1)){\QQ^{   (h)}_{  t}}_{i_1,i_2,\dots,i_m}^{j_1,j_2,\dots,j_m}.
\end{align*}
\end{enumerate}
Thanks to the scalings (\ref{eqn:rescaleSym}) and (\ref{eqn:rescaleHer}),  on   the right hand sides of the above four equations, the sums of the coefficients vanish.
  This allows us to interpret them as  multi-particle random walks (in  random environments) in the next subsection.

\subsection{Multi-particle random walk.\ }
Consider the following notation, $\boeta:  \llbracket 1,N\rrbracket \to \bN$ where
$\eta_j:=\boeta(j)$ is interpreted as the number of particles at the site $j$. Thus $\boeta$ denotes the configuration space 
of particles. We denote $\cN(\boeta) = \sum_j  \eta_j$.

Define $\boeta^{i, j}$ to be the configuration by moving one particle from $i$ to $j$. 
If there is no particle at $i$ then $\boeta^{i, j} = \boeta$. Notice that there is a direction and the particle is moved  from $i$ to $j$. 
Given $n>0$, there is a one to one correspondence between  (1) $\{(i_1,  j_1),  \dots, (i_m , j_m) \}$ with distinct $i_k$'s and positive $j_k$'s summing to $n$, and (2) $\boeta$ with $\cN(\boeta)=n$:  we map 
$ \{(i_1,  j_1),  \dots, (i_m , j_m ) \}$ to $\boeta$ with $\boeta_{i_k} = j_k$ and $\boeta_\ell = 0$ if $\ell \not \in \{ i_1, \ldots i_m \}$. 
We define 
\be\label{feq}
  f^{H_0, (s)}_{\bla, t} (\boeta)  
  =\E^{H_0} ( {\QQ^{(s)}_{ t}}_{i_1,\dots,i_m}^{j_1,\dots,j_m} (t)\mid \bla ), \quad f^{H_0, (h)}_{\bla, t} (\boeta)  
   = \E^{H_0}  ({\QQ^{(h)}_{ t}}_{i_1,\dots,i_m}^{j_1,\dots,j_m} \mid \bla ),  \\
\ee
if the configuration of $\boeta$ is the same as the one given by the $i, j$'s.
Here $\bla$ denotes the whole path of eigenvalues for $0\leq t\leq 1$.
 The dependence in the initial matrix $H_0$ will often be omitted so that we write  $f^{(s)}_{\bla, t}=f^{H_0, (s)}_{\bla, t}$,
 $f^{(h)}_{\bla, t}=f^{H_0, (h)}_{\bla, t}$. 
The following theorem summarizes the results from the previous subsection. It also defines the eigenvector moment flow, through the generators (\ref{momentFlotSym}) and (\ref{momentFlotHer}). They are multi-particles random walks (with $n=\mathcal{N}(\boeta)$ particles) in  random environments with jump rates depending on the eigenvalues.

\begin{theorem}[Eigenvector moment flow] Let $\bq\in\mathbb{R}^N$, $z_k=\sqrt{N}\langle\bq, u_k (t) \rangle$ and $c_{ij}(t)= \frac 1  { N(\lambda_i - \lambda_j)^2(t)}$.
\begin{enumerate}[(i)]
\item
Suppose that $\boldu$ is the solution to the symmetric Dyson vector flow (\ref{eqn:eigenvectorsSymmetric})
and $ f^{(s)}_{\bla, t} (\boeta)$ is given by \eqref{feq} where   
$\boeta$ denote  the  configuration  $ \{(i_1,  j_1),  \dots, (i_m , j_m ) \}$. \nc Then 
$f^{(s)}_{\bla, t}$
satisfies the equation  
\begin{align}
\label{ve}
&\partial_t f^{(s)}_{\bla, t} =  \mathscr{B}^{(s)}(t)  f^{(s)}_{\bla, t},\\
\label{momentFlotSym}&\mathscr{B}^{(s)}(t)  f(\boeta) = \sum_{i \neq j} c_{ij}(t) 2 \eta_i (1+ 2 \eta_j) \left(f(\boeta^{i, j})-f(\boeta)\right).
\end{align}
\item
Suppose that $\boldu$ is the solution to the Hermitian Dyson vector flow (\ref{eqn:eigenvectorsHermitian}),
and $ f^{(h)}_{\bla, t}$ is given by \eqref{feq}.
Then 
it satisfies the equation  
\begin{align}
\notag
&\partial_t  f^{(h)}_{\bla, t}  = \mathscr{B}^{(h)}(t) f^{(h)}_{\bla, t} ,\\
\label{momentFlotHer}&\mathscr{B}^{(h)}(t)  f(\boeta) = \sum_{i \neq j} c_{ij}(t) \eta_i (1+ \eta_j) \left( f(\boeta^{i, j})-f(\boeta)\right).
\end{align}
\end{enumerate}
\end{theorem}
\noindent An important property of the eigenvector moment flow is reversibility with respect to  a simple explicit equilibrium measure. In the Hermitian case, this is simply the uniform measure on the configuration space.

Recall  that a measure $\pi$ on the configuration space is said to be reversible with respect to a generator $\LL$ if $\sum_{\boeta}\pi(\boeta)g(\boeta)\LL f(\boeta)=\sum_{\boeta}\pi(\boeta)f(\boeta)\LL g(\boeta)$ for any functions $f$ and $g$. 
We then define the Dirichlet form by 
$$
D^{\pi}(f)=-\sum_{\boeta}\pi(\boeta)f(\boeta)\LL f(\boeta).
$$

\begin{proposition} \label{prop:rev} For the eigenvector moment flow,  the following properties hold.

\begin{enumerate}[(i)]
\item Define a  measure on the configuration space by assigning the  weight 
\be\label{eqn:weight}
\pi^{(s)}(\boeta) = \prod_{x=1}^N \phi(\eta_x), \ \phi(k) =\prod_{i=1}^k\left(1-\frac{1}{2i}\right).
\ee
Then $\pi^{(s)}$ is a reversible measure for $\mathscr{B}^{(s)}$ and the Dirichlet form is given by 
$$
D^{\pi^{(s)}}(f)=\sum_{\boeta}  \pi^{(s)}(\boeta) \sum_{i \neq j}  c_{ij}  \eta_i (1+ 2 \eta_j) \left(f(\boeta^{i, j}) - f(\boeta)\right)^2.
$$
\item The uniform measure ($\pi^{(h)}(\boeta)=1$ for all $\boeta$) is reversible with respect to $\mathscr{B}^{(h)}$. The associated Dirichlet form is
$$
D^{\pi^{(h)}}(f)=\frac{1}{2}\sum_{\boeta} \sum_{i \neq j}  c_{ij}  \eta_i (1+ \eta_j) \left(f(\boeta^{i, j}) - f(\boeta)\right)^2.
$$
\end{enumerate}
\end{proposition}

\begin{proof} We first consider $(i)$, concerning  the symmetric eigenvector moment flow.
The measure $\pi^{(s)}$ is reversible for $\mathscr{B}^{(s)}$ for any choice of the coefficients
satisfying $c_{ij}=c_{ji}$ if and only if, for any $i<j$,
\begin{multline*}
\sum_{\boeta}\pi^{(s)}(\boeta)g(\boeta)\left(2\eta_i(1+2\eta_j)f(\boeta^{ij})+2\eta_j(1+2\eta_i)f(\boeta^{ji})\right)\\
=\sum_{\boeta}\pi^{(s)}(\boeta)f(\boeta)\left(2\eta_i(1+2\eta_j)g(\boeta^{ij})+2\eta_j(1+2\eta_i)g(\boeta^{ji})\right).
\end{multline*}
A sufficient condition is clearly that both of the following equations hold:
\begin{align*}
&\sum_{\boeta}\pi^{(s)}(\boeta)g(\boeta)2\eta_i(1+2\eta_j)f(\boeta^{ij})=\sum_{\boeta}\pi^{(s)}(\boeta)f(\boeta)2\eta_j(1+2\eta_i)g(\boeta^{ji}),\\
&\sum_{\boeta}\pi^{(s)}(\boeta)g(\boeta)2\eta_j(1+2\eta_i)f(\boeta^{ji})=\sum_{\boeta}\pi^{(s)}(\boeta)f(\boeta)2\eta_i(1+2\eta_j)g(\boeta^{ij}).
\end{align*}
Consider the left hand side of  
the first one of these two  equations. Let $\boxi=\boeta^{ij}$. If $ \xi_j>0$ then
$\boeta=\boxi^{ji}$, $\eta_i=\xi_i+1$ and $\eta_j=\xi_j-1$. For the right hand side of the second equation, we make the change of variables $\boxi=\boeta^{ji}$. 
Finally, rename all the variables on the right hand sides by $\xi$. Thus 
the above equations are equivalent to
\begin{align*}
&\sum_{\boxi}\pi^{(s)}(\boxi^{ji})g(\boxi^{ji})2(\xi_i+1)(2\xi_j-1)f(\boxi)=\sum_{\boxi}\pi^{(s)}(\boxi)f(\boxi)2\xi_j(1+2\xi_i)g(\boxi^{ji}),\\
&\sum_{\boxi}\pi^{(s)}(\boxi^{ij})g(\boxi^{ij})2(\xi_j+1)(2\xi_i-1)f(\boxi)=\sum_{\boxi}\pi^{(s)}(\boxi)f(\boxi)2\xi_i(1+2\xi_j)g(\boxi^{ij}).
\end{align*}
Clearly, both equations hold provided that 
\be\label{reverse}
\pi^{(s)}(\boxi^{ji})2(\boxi_i+1)(1+2(\boxi_j-1)) = \pi^{(s)}(\boxi)2\boxi_j(1+2\boxi_i).
\ee
If the measure is of type $\pi^{(s)}(\boeta)=\prod_x\phi(\eta_x)$ and we note $\xi_i=a, \xi_j=b$, this  equation is  equivalent to 
$$
\phi(a+1)\phi(b-1)2(a+1)(2b-1)=\phi(a)\phi(b)2b(2a+1),
$$
and the second equation yields  the same condition with the roles of $a$ and $b$ switched.  This holds for all $a$ and $b$ if $\phi(k+1)=((2k+1)/(2k+2))\phi(k)$, which gives (\ref{eqn:weight}) provided  we normalize  $\phi(0)=1$.
In the case $(ii)$, the same reasoning yields that $\phi$ is constant.

Finally, the Dirichlet form calculation is standard: for example, for $(i)$,
$\sum_{\boeta}\pi^{(s)}(\boeta)\mathscr{B}^{(s)}(f^2)(\boeta)=0$ by reversibility. Noting
$\mathscr{B}^{(s)}(f^2)(\boeta)=2f(\boeta)\mathscr{B}^{(s)}f(\boeta)-\sum\pi^{(s)}(\boeta)2\eta_i(1+2\eta_j)(f(\boeta)-f(\boeta^{ij}))^2$ allows to conclude.
\end{proof}

\section{Maximum principle}\label{sec:max}

From now on we only consider the symmetric ensemble. 
The Hermitian case can be treated with the same arguments and only notational changes. 
Given a typical path $\bla$,  we will prove in this section that the solution to the eigenvector moment flow (\ref{momentFlotSym})  converges 
uniformly to $1$ for $t=N^{-1/4+\e}$.
It is clear that the maximum (resp. minimum)  of $f$ over $\boeta$ decreases (resp. increases). We can quantify this decrease (resp. increase)  in terms of the maximum and minimum  themselves (see (\ref{eqn:StQUE})). This yields an
explicit convergence speed to 1 by a Gronwall argument.

\subsection{Isotropic local semicircle law.\ }

Fix a (small) $\omega > 0$ and define
\begin{equation}\label{eqn:omega}
\b S = \b S(\omega, N)=\hb{z=E+\ii\eta \in\C \col \abs{E}\leq \omega^{-1}\,,\, N^{-1 + \omega} \leq \eta \leq \omega^{-1}}.
\end{equation}
In the statement below, we will also need $m(z)$, the Stieltjes transform of the semicircular distribution, i.e. 
$$
m(z)=\int\frac{\varrho(s)}{s-z}\rd s=\frac{-z+\sqrt{z^2-4}}{2}, \quad \varrho(s)=\frac{1}{2\pi}\sqrt{(4-s^2)_+}, 
$$
where the square root is chosen so that $m$ is holomorphic in the upper half plane and $m(z)\to 0$ as $z\to\infty$.
The following  isotropic local semicircle law  (Theorem 4.2 in  \cite{BloErdKnoYauYin2013}) gives 
 very useful  bounds on $\langle\bq, u_k\rangle$ for any eigenvector $u_k$ via estimates on the associated Green function.

\begin{theorem}[Isotropic local semicircle law \cite{BloErdKnoYauYin2013}] \label{thm:ILSC}
 Let $H$ be an element from the generalized Wigner ensemble and $G(z)=(H-z)^{-1}$. Suppose that  \eqref{eqn:moments} holds.  Then
for any (small) $\xi>0$ and (large) $D>0$ we have, for large enough $N$,
\begin{equation} \label{ISSC_estimate}
\sup_{|\bq|=1,z\in\b S}\P\left(|\scalar{\bq}{G(z)\bq}-m(z)|>N^\xi \left(\sqrt{\frac{\im m(z)}{N\eta}}+\frac{1}{N\eta}\right)\right)\leq N^{-D}.
\end{equation}
\end{theorem}

An important consequence of this theorem, to be used  in multiple occasions, is the following isotropic delocalization of eigenvectors:
under the same assumptions as Theorem \ref{thm:ILSC}, for any $\xi>0$ and $D>0$, we have
$$
\sup_{|\bq|=1,k\in\llbracket 1,N\rrbracket}\P\left(|\scalar{\bq}{u_k}|>N^{-1+\xi}
\right)\leq N^{-D}.
$$ 
Under the same assumptions the Stieltjes transform  was shown  \cite{ErdYauYin2012Rig} to satisfy the estimate
\begin{equation}\label{eqn:optRig}
\sup_{z\in{\b S}}\P\left(\left|\frac{1}{N}\Tr  G(z) -m(z)\right|>\frac{N^\xi}{N\eta}\right)\leq N^{-D}.
\end{equation}

\subsection{Rescaling.\ }
Recall the definition \eqref{eqn:SymSDE} of the evolution matrix  $H^{(s)}_{t}$.
The variance $\sigma_{ij}^2(t)$  of the matrix element $h_{ij}(t)$  is given by $\sigma_{ij}^2(t)= 
\sigma_{ij}^2+ t/N$  if $i\neq j$, $\sigma_{ij}^2(t)= 
\sigma_{ij}^2+ 2t/N$ if $i=j$.  Denote by 
$
\alpha(t) = \left(1+\frac{N+1}{N} t\right)^{-1/2}.$
 Then  
 $\alpha(t)  H^{(s)}_{t} $
is a generalized Wigner ensemble. In particular, the previously mentionned rigidity estimates hold along our dynamics  if we rescale $H^{(s)}_{t}$ into $\alpha(t)  H^{(s)}_{t}$. Consider the simple time change of our dynamics $u(t)=\int_0^t\alpha(s)^{-2}\rd s$. Then $\widetilde f_{t}(\boeta):=f_{u(t)}(\boeta)$ satisfies 
$$
\partial_t \widetilde f(\boeta)=\sum_{i\neq j}\frac{1}{N(\al(t)\la_i(t)-\al(t)\la_i(t))^2}2\eta_i(1+2\eta_j)\left(\widetilde f(\boeta^{ij})-\widetilde f(\boeta)\right).
$$
In the rest of the paper
it will always be understood that the above time  
rescaling $t\to u(t)$ and matrix scaling   $H^{(s)}_{t}\to\alpha(t)  H^{(s)}_{t}$ are performed so that all rigidity estimates hold as presented in the previous subsection,
for all time.

\subsection{Maximum Principle and regularity.\ }

Let $H_0$ be a symmetric generalized Wigner matrix with eigenvalues $\bla_0$ and an eigenbasis $\boldu_0$. Assume that $\bla,\boldu$ satisfy (\ref{eqn:eigenvaluesSymmetric}) (\ref{eqn:eigenvectorsSymmetric}) with initial condition $\bla_0,\boldu_0$.
Let   $ G(z)= G(z, t)=(\boldu^*\bla\boldu -z)^{-1}(t)$ be the Green function. For  $\omega>\xi>0$ and $\bq\in\RR^N$, consider the following three conditions (remember the notation (\ref{eqn:omega}) for ${\b S}(\omega,N)$):
\begin{align}\notag
&A_1(\bq,\omega,\xi,N)=\Big\{|\langle\bq, G(z) \bq\rangle- m(z)|<N^{\xi}
\left(
\sqrt{\frac{ \Im m(z)}{N\eta}}+\frac{1}{N\eta}
\right),\\
\label{eqn:condition1}&\hspace{3cm}\left|\frac{1}{N}\Tr  G(z) - m(z)\right|<
\frac{N^\xi}{N\eta}\ \mbox{for all}\ t\in[0,1],  z\in {\b S}(\omega,N)
\Big\},\\
&A_2(\omega,N)=\left\{|\la_k(t)-   \gamma_k|<N^{-\frac{2}{3}+\omega}(\hat k)^{-\frac{1}{3}}\ \mbox{for all}\ t\in[0,1], k\in\llbracket 1, N\rrbracket\right\},\label{eqn:condition2}\\
&A_3(\omega,N)=\left\{\langle \bq,u_k(t)\rangle^2< N^{-1+\omega}\ \mbox{for all}\ t\in[0,1],  k\in\llbracket 1, N\rrbracket\right\}.\label{eqn:condition3}
\end{align}
Note that   the   two conditions (\ref{eqn:condition2}) and (\ref{eqn:condition3}) follow from  (\ref{eqn:condition1}), i.e., $A_1\subset A_2\cap A_3$ ,   by  standard 
arguments.   More precisely,  (\ref{eqn:condition3}) can be proved by the argument in  the proof of Corollary 3.2 in 
\cite{ErdYauYin2012Univ}. 
The condition   (\ref{eqn:condition2}) is exactly the content of the rigidity of eigenvalues, i.e., Theorem 2.2 in 
\cite{ErdYauYin2012Rig}. Its proof in  
Section 5 of \cite{ErdYauYin2012Rig} used only the estimate  (\ref{eqn:condition1}).

The following lemma shows that  these conditions hold with high probability.

\begin{lemma}\label{lem:A13}
For any $\omega>\xi> 0,   D >0$ and  $N$ large enough,   we have 
$$
\inf_{\bq\in\RR^N, |\bq|=1} \P\left(A_1(\bq,\omega,\xi,N)\right)\geq 1-N^{- D}, 
$$
where the probability denotes the joint law of the random variable $H_0$ and the paths of $\bla,\boldu$.
\end{lemma}

\begin{proof}
For any fixed time, by  (\ref{eqn:dual})  (\ref{ISSC_estimate}) and (\ref{eqn:optRig}),
the condition (\ref{eqn:condition1}) holds with probability $1-N^{-C}$ for any $C$.  
As $C$ can be arbitrary, the same condition hold for any time and $z$ in a discrete set of size $N^{C/2}$, say. 
 For any two matrices $H$ and $H'$ with Green functions $G(z)$ and $G'(z)$, we have 
$$
[G(z) - G'(z)]_{ij} = - \sum_{k, \ell} G(z)_{ik} (H-H')_{k \ell} G'(z)_{\ell j}
$$
Since 
$$
|G(z)_{ab}| \le \sum_k \frac { |u_k(a) u_k(b)| } { |\lambda_k - z|}  
\le (2\eta)^{-1}  \sum_k   (|u_k(a)|^2 + |u_k(b)|^2) \le N \eta^{-1}  ,   \quad \im z := \eta,
$$
we have 
$$
\Big | [G(z) - G'(z)]_{ij} \Big | = N^3 \eta^{-2}   \sqrt { \sum_{k, \ell}  | (H-H')_{k \ell} |^2}.
$$
Applying  this inequality to $H^{(s)}_{t}$ and $H^{(s)}_{s}$, we have with very high probability that 
$$
\sup_{ |s-t| \le |t-t'|}\Big | \langle\bq, G(z) \bq\rangle - \langle\bq, G(z) \bq\rangle \Big | \le C N^6 |t-t'|^{1/2}
$$
$N \eta \ge 1$. Here we have used the standard property  that the sup over $[0,t-t']$ of a standard Brownian motion
has size order $|t-t'|^{1/2}$ and Gaussian tails.  
Therefore,  we can use  a continuity argument to  extend the estimate (\ref{eqn:condition1})   to all $z\in {\b S}$ and all time between $0$ and $1$. This proves (\ref{eqn:condition1}). 
\end{proof}

\noindent We define the set
\begin{equation}\label{eqn:A}
A(\bq,\omega,\xi, \nu ,   N)= \Big \{(H_0,\bla) : 
  \P\left(A_1(\bq,\omega,\xi,N) \Big | (H_0,\bla) \right) 
 \geq  1-N^{-\nu}  \Big \}.
\end{equation}
From the previous lemma, one easily sees that for any $\omega>\xi$, $\nu$ and $ D>0$,
we have, for large enough $N$, 
\begin{equation}\label{eqn:largeProb}
\inf_{\bq\in\RR^N}\P\left(A(\bq,\omega,\xi,  \nu,N)\right)\geq 1-N^{-  D}.
\end{equation}

\begin{theorem}\label{thm:maxPrinciple}
Let $n\in\NN$ and $f$ be a solution of the eigenvector moment flow (\ref{ve}) with initial matrix $H_0$ and path $\bla$ in $A(\bq,\omega,\xi, \nu ,N)$ for some $\nu>2$.
Let $t=N^{-\frac{1}{4}+\delta}$, where $\delta\in(\frac{n\omega}{2},1/4]$ and  we assume
that $\omega>\xi$ and  $n\omega<1/2$. Then for any $\e>0$ and large enough $N$ we have
\be\label{fest}
\sup_{\boeta:\mathcal{N}(\boeta)=n}\left|f_t(\boeta)-1\right|\leq C N^{n\omega+\e-2\delta}.
\ee
The constant $C$ depends on $\e,\omega,\delta$ and $n$ but not on $\bq$. 
\end{theorem}

We have the following
asymptotic normality for eigenvectors of a Gaussian divisible Wigner ensemble with a small Gaussian component.

\begin{corollary}\label{cor:GaussDivisible}
Let $\delta$  be an arbitrarily small constant and $t=N^{-1/4+\delta}$. Let $H_t$ be the solution
to (\ref{eqn:SymSDE}) and  $(u_1(t),\dots,u_N(t))$ be an eigenbasis of
$H_t$. The initial condition $H_0$ is assumed to be a symmetric generalized Wigner matrix.
Then for any polynomial $P$ in $m$ variables and any $\e>0$, for large enough $N$ we have
\begin{equation}\label{eqn:loc}
\sup_{I\subset\llbracket 1, N\rrbracket,|I|=m, |\bq|=1}
\left|\E\left(P\left((N\langle \bq,u_k(t)\rangle^2)_{k\in I}\right)  \right)-\E P\left((\mathscr{N}_j^2)_{j=1}^m\right)\right|\leq C N^{\e- 2 \delta }.
\end{equation}

\end{corollary}

\begin{proof} Since $H_0$ is a generalized Wigner matrices, the isotropic local semicircle law, Theorem \ref{thm:ILSC},  holds  for all time with $\xi$ arbitrarily small. With  $\om = 2 \xi$,
and noticing that Lemma \ref{lem:A13} holds for arbitrary large $\nu>0$,  
\eqref{fest} implies that  \eqref{eqn:loc} holds.
\end{proof}

\begin{proof}[Proof of Theorem \ref{thm:maxPrinciple}]
Because of (\ref{eqn:largeProb})  and $A_1\subset A_2\cap A_3$,  we can assume in this proof that the trajectory $(H_t)_{0\leq t\leq 1}$ is in 
$A_1(\bq,\omega,\xi,N)\cap A_2(\omega,N)\cap A_3(\omega,N)$;  the complement of this set induces  an additional error $\OO(N^{-\nu+\xi})$ in (\ref{fest}), negligible compared to $N^{n\omega+\e-2\delta}$.

We begin with the case $n=1$. Let $f_s(k)=f_s(\boeta)$, where $\boeta$ is the configuration with one particle at 
the lattice point $k$.  The equation \eqref{ve} becomes 
\be\label{ve3}
\partial_s f_s(k)=\frac{1}{N}\sum_{j\neq k}\frac{f_s(j)-f_s(k)}{(\la_j-\la_{k})^2}
\ee 
Assume that
$$
\max_{k\in\llbracket 1,N\rrbracket}f_s(k)=f_s(k_0) 
$$ 
for some $k_0$ ($k_0$ is not unique in general).  Clearly, we have 
\[
\frac{f_s(j)-f_s(k_0)}{(\la_j-\la_{k_0})^2}  \le \frac{f_s(j)-f_s(k_0)}{(\la_j-\la_{k_0})^2 + \eta^2} 
\]
Together with  (\ref{ve3}),  for any $\eta>0$ we have
\begin{equation}\label{eqn:fun}
\partial_s f_s(k_0)=\frac{1}{N}\sum_{j\neq k_0}\frac{f_s(j)-f_s(k_0)}{(\la_j-\la_{k_0})^2}
\leq
\frac{1}{N\eta}\sum_{j\neq k_0}\frac{\eta f_s(j)}{(\la_j-\la_{k_0})^2+\eta^2}
-
f_s(k_0)\frac{1}{N\eta}\sum_{j\neq k_0}\frac{\eta}{(\la_j-\la_{k_0})^2+\eta^2}.
\end{equation}
Notice that
$$
\frac{1}{N}\sum_{1\leq j\leq N}\frac{\eta f_s(j)}{(\la_j-\la_{k_0})^2+\eta^2}
=
\E\left(\sum_{j=1}^N\frac{\eta \langle\bq,u_j\rangle^2}{(\la_j-\la_{k_0})^2+\eta^2}\Big | (H_0,\bla)\right).
$$
From the definition of $A(\bq,\omega,\xi, \nu,N)$,  for $N^{-1+\omega}<\eta<1$ we therefore have
$$
\frac{1}{N}\sum_{1\leq j\leq N,  j \neq k_0}\frac{\eta f_s(j)}{(\la_j-\la_{k_0})^2+\eta^2}=
\Im m(\la_{k_0}+\ii\eta)+\OO\left(\frac{N^{\xi}(\Im m(\la_{k_0}+\ii\eta))^{1/2}}{(N\eta)^{1/2}}
+ \frac{N^\omega}{N\eta} \right), 
$$
where the error $N^\omega/(N\eta)$ comes from the missing term $j= k_0$
and we have used that 
for $(H_0,\la)\in A(\bq,\omega,\xi, \nu,N)$,  $N\langle\bq,u_j\rangle^2$ is bounded by
$N^\omega$ with very high  probability.
For the same reason, we have
$$
\frac{1}{N}\sum_{1\leq j\leq N,  j \not = k_0}\frac{\eta}{(\la_j-\la_{k_0})^2+\eta^2}=\Im m(\la_{k_0}+\ii\eta)+\OO\left(\frac{N^{\omega}}{N\eta}\right).
$$
Using these estimates, (\ref{eqn:fun}) yields
$$
\partial_s (f_s(k_0)-1)\leq -c\frac{\Im m(\la_{k_0}+\ii\eta)}{\eta}(f_s(k_0)-1)+\OO\left(\frac{N^{\xi}\Im m(\la_{k_0}+\ii\eta)^{1/2}}{N^{1/2}\eta^{3/2}}\right)+\OO\left(\frac{N^{\omega}}{N\eta}\right).
$$
Moreover, from the definition of $A(\bq,\omega,\xi, \nu ,N)$, we know that $-2-N^{-\frac{2}{3}+\omega}\leq \la_{k_0}\leq 2+N^{-\frac{2}{3}+\omega}$. As our final choice of $\eta$ will satisfy $N^{-\frac{2}{3}+\xi}\leq \eta\leq 1$, this implies that 
$$
\Im m(\la_{k_0}+\ii\eta)\geq c\sqrt{\eta}.
$$
Let 
$S_s=\sup_{k}(f_s(k)-1).$
Note that there may be some $s$ for which $S_s$ is not differentiable (at times when the maximum is obtained for at least two distinct indices). But if we denote
\begin{equation}\label{eqn:derivative}
S'_t=\limsup_{u\to t}\frac{S_t-S_u}{t-u},
\end{equation}
 the above reasoning  
shows
$$
S'_s\leq -\frac{c}{\sqrt{\eta}} S_s+C\frac{N^\xi}{N^{1/2}\eta^{3/2}}+ C\frac{N^\omega}{N \eta } 
\le -\frac{c}{\sqrt{\eta}} S_s+C\frac{N^\om}{N^{1/2}\eta^{3/2}}.
$$
We chose $\eta=N^{-\frac{1}{2}+2\delta-\e}$ for some small $\e\in(0,2\delta-\omega)$ and $t=N^{-\frac{1}{4}+\delta}$. The  Gronwall inequality gives
$$
S_t\leq C \left( e^{-N^{\e/2}}+ N^{\omega+\e-2\delta}\right).
$$
We can do the same reasoning for the minimum of $f$. This concludes the proof for $n=1$. 

For $n\geq2$ the same argument works and we will proceed by induction.
Let $\boxi$ satisfy
$$
\max_{\mathcal{N}(\boeta)=n}f_s(\boeta)=f_s(\boxi).
$$
Assume $\boxi$ is associated to $j_r$ particles at site $k_r$, $1\leq r\leq m$ for some $m\leq n$, where the $k_r$'s are distinct and $j_r\geq 1$.
Then
\begin{equation}\label{eqn:diffmult}
\partial_s f_s(\boxi)
\leq C\ \sum_{r=1}^m\left(\frac{1}{N\eta}\sum_{j\neq k_r}\frac{\eta f_s(\boxi^{k_r j})}{(\la_{k_r}-\la_j)^2+\eta^2}
-f_s(\boxi)\frac{1}{N\eta}\sum_{j\neq k_r}\frac{\eta}{(\la_{k_r}-\la_j)^2+\eta^2}
\right), 
\end{equation}
where $\boxi^{k_r j}$ is defined in Section 3.2. 
We now  estimate  the first term on the right hand side (the second term was estimated in the previous $n=1$ step).
By  \eqref{eqn:condition3},   for $(H_0,\la)\in A(\bq,\omega,\xi, \nu,N)$,  $N\langle\bq,u_j\rangle^2$ is bounded by
$N^\omega$ with very high  probability. Thus  we  have
$$
\frac{1}{N}\sum_{j\neq k_r}\frac{\eta f_s(\boxi^{k_r j})}{(\la_{k_r}-\la_j)^2+\eta^2}=
\frac{1}{N}\sum_{j\not\in\{k_1,\dots,k_m\}}\frac{\eta f_s(\boxi^{k_r j})}{(\la_{k_r}-\la_j)^2+\eta^2}+\OO\left(\frac{N^{n\omega}}{N\eta}\right).
$$
Moreover, by definition the above sum can be estimated by 
\begin{align}
&\E\left(\notag
\left(
\frac{(N\langle \bq,u_{i_r}\rangle^2)^{j_r-1}}{a(2j_{r}-2)}
\prod_{1\leq r\leq m,r\neq r}\frac{(N\langle \bq,u_{i_r}\rangle^2)^{j_r}}{a(2j_{r})}\right)
\left(
\frac{1}{N}\sum_{j\not\in\{k_1,\dots,k_m\}}\frac{\eta(N\langle\bq,u_j\rangle^2)}{(\la_j-\la_{k_r})^2+\eta^2}
\right)
\, \Big | \, (H_0,\bla)
\right)\\
=&\ \label{eqn:condit1}
\E\left(
\left(
\frac{(N\langle \bq,u_{i_r}\rangle^2)^{j_r-1}}{a(2j_{r}-2)}
\prod_{1\leq r\leq m,r\neq r}\frac{(N\langle \bq,u_{i_r}\rangle^2)^{j_r}}{a(2j_{r})}\right)
\Im\langle\bq,G(\la_{k_r}+\ii\eta),\bq\rangle
\, \Big | \,  (H_0,\bla)
\right)
+\OO\left(\frac{N^{n\omega}}{N\eta}\right),
\end{align}
where we first used that extending the indices to $1\leq j\leq N$ induces an error $\OO(N^\omega(N\eta)^{-1})$
and  the bound  $N\langle\bq,u_j\rangle^2 \le N^\omega$ holds with very high  probability. We have also used that 
for $(H_0,\la)\in A(\bq,\omega,\xi, \nu,N)$, we can replace $\Im\langle\bq,G(\la_{k_r}+\ii\eta),\bq\rangle$
by $\Im m(\la_{k_r}+\ii\eta)+\OO(N^\xi(N\eta)^{-1/2})$.  This yields
\begin{equation}\label{eqn:condit2}
\frac{1}{N}\sum_{j\neq k_r}\frac{\eta f_s(\boxi^{k_r j})}{(\la_{k_r}-\la_j)^2+\eta^2}=
f_s(\boxi\backslash k_r)\Im m(\la_{k_r}+\ii\eta)+\OO\left(\frac{N^{n\omega}}{(N\eta)^{1/2}}\right),
\end{equation}
where $\boxi\backslash k_r$ stands for the configuration $\boxi$ with one particle removed from site $k_r$.  
By induction assumption,  we  can use \eqref{fest} to  estimate $f_s(\boxi\backslash i_r)$ for $s\in(t/2,t)$. 
We have thus  proved that
$$
\partial_s (f_s(\boxi)-1)\leq -\frac{c}{\sqrt{\eta}}(f_s(\boxi)-1)+\OO\left(\frac{N^{n\omega}}{N^{1/2}\eta^{3/2}}\right)+\OO\left(\frac{N^{(n-1)\omega+\e-2\delta}}{\eta}\right).
$$
on $(t/2,t)$. 
Notice that by our assumptions on the parameters $\om, \delta, \eta$ and $\xi$, the first error term always dominates the second.
One can now bound $|f_s(\boxi)-1|$ in the same way as in the $n=1$ case.
\end{proof}

If $\omega$ can be chosen arbitrarily small (this is true for generalized Wigner matrices), Theorem 
\ref{thm:maxPrinciple} gives
$\sup_{\boeta:\mathcal{N}(\boeta)=n}|f_t(\boeta)-1|\to 0$ for any $t=N^{-1/4+\e}$. This could be improved to 
$t=N^{-1/3+\e}$ by allowing $\eta$ to depend on $k_0$ in the previous reasoning (chose $\eta=N^{-2/3+\e}{\hat k_0}^{1/3}$).

More generally, our proof  shows that 
the following equation (\ref{eqn:StQUE}) (with the convention \ref{eqn:derivative}) holds.
Let 
\begin{align*}
&\Delta_1(k,\eta)=\E(\langle\bq,G(\la_k+\ii\eta)\bq\rangle-\Im m(\la_k+\ii\eta)\mid (H_0,\bla)),\\
&\Delta_2(k,\eta)=\E(N^{-1}{\rm Tr} G(\la_k+\ii\eta)-\Im m(\la_k+\ii\eta)\mid (H_0,\bla)),
\end{align*}
where all variables depend on $t$  (remember in particular that $G(z)=(\boldu_t^*\bla_t\boldu_t -z)^{-1}$).
Then the  following  maximum  inequality holds: 
\begin{equation}\label{eqn:StQUE}
S'_t
\leq
\max_{k:S_t=f_t(k)}\inf_{\eta>0}
\left\{-\frac{\Im m(\la_k+\ii\eta)}{\eta}\ S_t
+
\frac{|\Delta_1(k,\eta)|}{\eta}
+
\frac{|\Delta_2(k,\eta)|(S_t+1)}{\eta}
+
\frac{N^\omega (S_t+1)}{N\eta^2}\right\}.
\end{equation}
Similar inequalities  for a general number of particles can be obtained.

\section{Proof of the main results} 

\label{sec:Proof}

\subsection{A comparison theorem for eigenvectors.\ }

Corollary \ref{cor:GaussDivisible} asserts  the asymptotic normality 
of eigenvector components for Gaussian divisible ensembles for $t$ not too small. 
In order to prove Theorem \ref{thm:main}, 
we need to remove the small Gaussian components of the matrix elements in this Gaussian divisible ensemble. 
Similar questions occurred in the proof of universality conjecture for Wigner matrices 
and several methods were developed for this purpose (see, e.g., \cite{ErdPecRamSchYau2010}  and  \cite{TaoVu2011}).  
Both methods can be extended to yielding similar eigenvector comparison results.  
In this paper, we  will  use  the Green function comparison theorem  introduced   in 
\cite [Theorem 2.3]{ErdYauYin2012Univ} (the parallel result following the  argument  of  \cite{TaoVu2011} was given in  \cite{TaoVu2012}). 
Roughly speaking, \cite[Theorem 1.10]{KnoYin2011} states that the distributions of eigenvectors 
for two generalized Wigner ensembles are identical provided the first four moments of the matrix elements 
are identical and a level repulsion estimate  holds for {\it one}  of the two ensembles.  
We  note that the level repulsion 
estimates needed in \cite{TaoVu2012} are substantially different. 
We first recall  the following definition.

\begin{definition}[Level repulsion estimate] \label{lr}
 Fix an energy $E$ such that $ \gamma_k \le E \le \gamma_{k+1}$ for some $k \in \llbracket 1, N\rrbracket$. 
A generalized Wigner ensemble is said to satisfy  the level repulsion at the energy $E$ if there exist $\alpha_0>0$ such that for any $0<\alpha<\alpha_0$, there exists $\delta>0$ such that 
$$
\P\left(|\{ j: \la_{ j\nc} \nc \in[E-N^{-2/3-\alpha}\hat k^{-1/3},  E+N^{-2/3-\alpha}\hat k^{-1/3}]\}|\geq 2\right)\leq N^{-\alpha-\delta},
$$
where $\hat k = \min (k, N-k+1)$. 
A matrix ensemble is said to satisfy the  level repulsion estimate uniformly if this property holds for any energy 
$E\in(-2,2)$. 
\end{definition}

We note that such level repulsion estimates for generalized Wigner matrices was proved near the edge 
(more precisely for $ 0 \le \hat k \le N^{1/4}$)   \cite[Theorem 2.7]{BouErdYau2013} and in the bulk  
\cite{ErdYau2012} via the universality of  gap statistics.   In the intermediate regime,  the level repulsion 
in this sense has not been worked out  although it is clear that the techniques developed in these papers 
can be adapted to prove such results.  From now on, we will assume that this level repulsion estimate holds in the region 
$\mathbb{T}_N= \llbracket 1,N^{1/4} \rrbracket  \cup 
\llbracket  N^{1-\delta} , N- N^{1-\delta} \rrbracket \cup \llbracket N-N^{1/4},N\rrbracket$ needed for Theorem \ref{thm:main}
and its corollaries.\\

The following theorem is a slight extension  of \cite[Theorem 1.10]{KnoYin2011}  with the following modifications \nc: (1)
We slightly weaken 
the fourth moment matching condition. (2) The original theorem was only for components of 
eigenvectors;  we allow the eigenvector to project to a fixed direction. (3) We state it for all energies in the entire spectrum. 
(4) We include an error bound for the comparison. (5) We state it only for  eigenvectors with no involvement of eigenvalues. 
Theorem \ref{t2} can be proved  using  the  argument   in \cite{KnoYin2011}; the only modification 
is to replace  the local semicircle law used in \cite{KnoYin2011} by the isotropic local semicircle law, Theorem \ref{thm:ILSC}. 
Since this type of argument based on the Green function comparison theorem has been done several times, 
we will not repeat it here. Notice that near the edge, the four moment matching condition can be replaced by just two moments. But for applications in this paper,   this improvement will not be used and so we refer the interested reader to 
\cite{KnoYin2011}.

\begin{theorem}[Eigenvector Comparison Theorem] \label{t2}
Let $H^{\f v}$ and $H^{\f w}$ be generalized Wigner ensembles where
$H^{\f v}$ satisfies  the level repulsion estimate 
uniformly. Suppose that the first three off-diagonal moments of
 $H^\bv$ and $H^\bw$ are the same, i.e.
$$
\E^{\f v}(h_{ij}^3)=  \E^{\f w}(h_{ij}^3)
\for i \neq j  
$$
and that the first two diagonal moments of
 $H^\bv$ and $H^\bw$ are the same, i.e.
$$
\E^{\f v}(h_{ii}^2)=  \E^{\f w}(h_{ii}^2).
$$
Assume also that the fourth  off-diagonal moments of
 $H^\bv$ and $H^\bw$ are almost the same, i.e., there is an $a > 0$ such that 
$$
	\Big |  \E^{\f v}(h_{ij}^4)- \E^{\f w}(h_{ij}^4)\Big | \le N^{-2-a} 
\for i \neq j. 
$$
Then there is $\e > 0$ depending on $a$ such that for any integer $k$, any $\bq_1, \ldots \bq_k $ and  any choice of indices  $1\le j_1, \ldots, j_k \le N$
we have
$$
 \left(\E^{\f v} - \E^{{\f w}}\right) \Theta \pB{ 
N\langle \bq,u_{j_1} \rangle^2, \ldots, N\langle \bq,u_{j_k}\rangle^2   } =\OO(N^{-\e}),
$$
where $\Theta$ is a smooth function that satisfies
$$
\abs{\partial^m \Theta(x)} \;\leq\; C (1 + \abs{x})^C
$$
for some arbitrary $C$ and all $m\in \N^{k}$ satisfying
$\abs{m} \leq 5$.
\end{theorem}

\subsection{Proof of Theorem \ref{thm:main}\ .}
We now summarize our situation: Given a generalized Wigner ensemble  $\hat H$, 
we wish to prove that \eqref{eqn:mainSym} holds for the eigenvectors of $\hat H$. 
We have proved in \eqref{eqn:loc} that this estimate holds for 
any Gaussian divisible ensemble of type $H_0+\sqrt{t}\ U$, and therefore by simple rescaling for any
ensemble of type
$$
H_t = e^{-t/2} H_0 + (1-e^{-t})^{1/2}\, U,
$$
where $H_0$  is {\it any}  initial generalized  Wigner matrix and 
$U$ is an independent standard  GOE matrix, as long as $t \ge N^{-1/4+\delta}$. 
We fix $\delta$ a small number, say, $\delta = 1/8$. 
Now  we construct a generalized  Wigner matrix $H_0$ such that  
the first three moments of  $H_t$ 
match exactly those of the  target  matrix $\hat H$ and the differences between 
the fourth moments of the two ensembles are less than 
$N^{-c}$ for some $c$ positive. This existence of such an initial random variable 
is guaranteed by, say,  Lemma 3.4 of \cite{EYYBernoulli}. By the eigenvector comparison 
theorem, Theorem \ref {t2}, we have proved \eqref{eqn:mainSym} and this concludes our proof of 
Theorem \ref{thm:main}.

\subsection{Proof of Corollary \ref{cor:Gauss}.\ } \label{secInd}
Let $\mathscr{N}=(\mathscr{N}_1,\dots,\mathscr{N}_N)$ be a Gaussian vector with covariance Id. Let $m,\ell\in\mathbb{N}$, $k\in 
\mathbb{T}_N$ and  $\{i_1,\dots,i_\ell\}: = J\subset\llbracket 1,N\rrbracket$.
For $\bq$ such that $q_i=0$ if $ i\not\in J$, consider the polynomial in $\ell$ variables:
$$
Q(q_{i_1},\dots,q_{i_\ell})=\E\left((N|\langle\bq,u_k\rangle|^2)^m\right)-\E\left(|\langle\bq,\mathscr{N}\rangle|^{2m}\right).
$$
From (\ref{eqn:mainSym}), there exists $\e>0$ such that 
$$
\sup_{|q_{i_1}|^2\leq\frac{1}{\ell},\dots,|q_{i_\ell}|^2\leq\frac{1}{\ell}}|Q(q_{i_1},\dots,q_{i_\ell})|\leq \sup_{|q|=1}|Q(q_{i_1},\dots,q_{i_\ell})|\leq N^{-\e},
$$
where, for the first inequality, we note that that the maximum of $Q$ in the unit ball is achieved on the unit sphere. 
Noting $R(q_{i_1})=Q(q_{i_1},\dots,q_{i_\ell})$ with the coefficients of the polynomial $R$ depending on $q_{i_2},\dots,q_{i_\ell}$, the above bound implies that all the coefficients of $R$ are bounded by $C_1 N^{-\e}$ for some universal constant $C_1$ (indeed, one recovers the coefficients of $R$ from its evaluation at $\ell+1$ different points, by inverting a Vandermonde matrix).

By iterating the above bound on the coefficients finitely many times ($\ell$ iterations), we conclude 
that there is a universal constant $C_\ell$ such that all coefficients of $Q$ are bounded by $C_\ell N^{-\e}$.
This means that  for any
$k\in \mathbb{T}_N $ and $J\subset\llbracket 1,N\rrbracket$ with 
$|J|=\ell$,
$$
\left|\E\left(\prod_{\al\in J}\left(\sqrt{N}u_k(\al)\right)^{m_\al}\right)-
\E\left(\prod_{\al\in J}\left(\mathscr{N}_\al\right)^{m_\al}\right)
\right|\leq C\ N^{-\e}
$$
whenever the integer exponents $m_\al$  satisfy  $\sum m_\al=m$. Here $C$  depends only on $m$, not on the choice of $k$ or  $J$. This concludes the proof of (\ref{eqn:corSym}), in the case of a monomial $P$ with even degree. If $P$ is a monomial of odd degree,  (\ref{eqn:corSym}) is trivial: the left hand side vanishes thanks to the uniform phase choice $e^{\ii\omega}$. This concludes the proof of Corollary \ref{cor:Gauss}.

\subsection{\it Proof of Corollary \ref{cor:QUE}.\ } \label{secQUE}
A second moment calculation yields
\begin{multline*}
\E\left(\left(\frac{N}{|a_N|}\langle u_k,a_N u_k\rangle\right)^2\right)=
\frac{1}{|a_N|^2}\E\left(\left(\sum_{\alpha} a_N(\alpha)(N|u_k(\alpha)|^2-1)\right)^2\right)\\
\leq \max_{\alpha\neq\beta}\E\left(\left(N|u_k(\alpha)|^2-1\right)\left(N|u_k(\beta)|^2-1\right)\right)
+\frac{1}{|a_N|}\max_\alpha\E\left(\left(N|u_k(\alpha)|^2-1\right)^2\right).
\end{multline*}
From (\ref{eqn:corSym}),  the first term of the right hand side  is bounded by  $N^{-\e}$
and the second term is bounded by $1/|a_N|$. 
The Markov inequality then allows us to conclude the proof of  Corollary \ref{cor:QUE}.

\section{Finite speed of propagation}

In this section, we prove a finite speed of propagation estimate for the dynamics \eqref{ve}. 
This estimate will be a key ingredient for proving optimal relaxation time for eigenvectors in the bulk. 
Finite speed of propagation   was first proved  in \cite[Section 9.6]{ErdYau2012singlegap} for \eqref{ve} when  the number of particle $n=1$. 
But it requires a level repulsion estimate which is difficult  to prove.  Our estimate requires only the rigidity of eigenvalues 
(which holds with very high probability) and 
the  speed of propagation obtained   is nearly optimal.  Our key observation is  that  we can construct  weight  functions used in the finite speed estimate 
 depending on 
eigenvalues so that the singularities in the equation \eqref{ve} are automatically cancelled by the choices of   these  weight functions.

We will follow the approach of \cite{ErdYau2012singlegap}  by decomposing the dynamics into a long range part and a short range part. 
The long range part can be controlled by a  general argument based on decay estimate; the  main new idea  is  in the  proof of
a finite speed of propagation for the short range dynamics, which is the content of Lemma \ref{FiniteSpeed}.

\subsection{Long and short range dynamics.\ }

We assume that  for some (small) fixed parameter $\xi>0$ there is a constant $C$ such that for any $|i-j|\ge N^{\xi}$ and $0\leq s\leq 1$ 
the quantity $c_{ij}$ defined  in (\ref{eqn:cij}) satisfies the following estimate 
\be\label{far1}
    c_{ij} (s) \le C\ \frac{N}{(i-j)^2}.
\ee 
If $M_N$ is  distributed as a
generalized Wigner matrix,  then
for any $\xi> 0$, \eqref{far1} holds  \cite{ErdYauYin2012Rig} with probability 
$1-e^{-c(\log N)^2}$  for some  $c> 0$. 
In this section $M_N$ is not assumed  to be distributed as a generalized Wigner matrix. Instead,  we  assume that 
$(\ref{far1})$ holds.
\nc

The following cutoff of the dynamics will be useful. Let $1\ll \ell \ll N $  be a parameter to be specified later.
We split the time dependent operator $\mathscr{B}$ defined in \eqref{momentFlotSym} into a short-range and
a long-range part: $\mathscr{B}=\mathscr{S} + \mathscr{L}$, with 
\begin{align}
&(\mathscr{S} f)(\boeta)  =   \sum_{|j-k| \le \ell }   c_{jk}(s) 2 \eta_j (1+ 2 \eta_k) \left(f(\boeta^{j, k}) - f(\boeta)\right),\label{eqn:shortcut}\\
&(\mathscr{L}f) (\boeta) =    \sum_{|j-k| >  \ell }   c_{jk}(s) 2 \eta_j (1+ 2 \eta_k ) \left(f(\boeta^{j,k}) - f(\boeta)\right) .\notag
\end{align}
 Notice that $\mathscr{S}$ and $\mathscr{L}$ are time dependent. Moreover, $\mathscr{S}$ is also reversible
with respect to $\pi$ (the proof of Proposition \ref{prop:rev} applies to any symmetric $c_{ij}$'s). 
Denote by $ \rU_{\mathscr{S}} ( s,   t)$  
 the  semigroup associated with $\mathscr{S}$ from time $s$ to time $t$,  i.e.
$$
\partial_{t} \rU_{\mathscr{S}} (s,t) = \mathscr{S}(t)  \rU_{\mathscr{S}} (s,t)
$$
for any $s\leq t$, and $\rU_{\mathscr{S}} (s,s)=\Id$.
The notation $\rU_{\mathscr{B}}(s,t)$ is analogous.  
In the following lemma, we prove that the short-range dynamics provide a good approximation of the global  dynamics. Lemmas \ref{cut2} 
follows the same proof as in \cite{ErdYau2012singlegap}, where they were shown for $n=1$.

\begin{lemma}\label{cut2}
Suppose that   the coefficients of $\mathscr{B}$
satisfy
\eqref{far1} for some $\xi>0$ and let $\ell\gg N^\xi$.
Suppose that the  initial data  is  the delta function 
at an arbitrary configuration $\boeta$.
Then for any $s\geq 0$ we have
$$
\|\left(\rU_{\mathscr{B}}  (0, s)- \rU_{\mathscr{S}}  (0, s) \right) \delta_{ \boeta} \|_1 \le C\ \frac{N s}{\ell},
$$
where $C$ only depends on $\xi$ (in particular not on $\boeta$).
\end{lemma}

\begin{proof}
By the Duhamel formula we have
$$
\rU_{\mathscr{B}}( 0, s) \delta_{\boeta}= \rU_{\mathscr{S}}  (0, s) \delta_{\boeta} +  \int_0^s \rU_{\mathscr{B}}( s', s)
\mathscr{L} (s')   \rU_{\mathscr{S}} (0,  s') \delta_{\boeta} \rd s'.
$$
Notice that for $\ell \gg N^\xi$
 we can use \eqref{far1} to get    
$$
\|\mathscr{L} f\|_1 \leq \sum_{\boeta}  \sum_{|j-k | \ge \ell }   
c_{jk} \eta_j (1+ 2 \eta_k ) \left(|f(\boeta^{j,k})|+ | f(\boeta)|\right)  \le C\ N \ell ^{-1} \| f\|_1.
$$
Since $\rU_{\mathscr{B}}$ and $\rU_{\mathscr{S}}$ are contractions in ${\rm L}^1$, this yields
$$
\int_0^s  \Lnorm  1 {   \rU_{\mathscr{B}} ( s',s ) \mathscr{L}(s') \rU_{\mathscr{S}} (0,  s')  \delta_{\boeta} }  \rd s' 
 \le 
C\ N\ell^{-1}  \int_0^s   \Lnorm  1 {\delta_{\boeta}}  \rd s'\leq C\ \frac{Ns}{\ell},
$$
which concludes the proof.
\end{proof}

\subsection{Finite speed of propagation for the short range dynamics. }

Suppose that $\boeta$ is a  configuration with $n$ particles. We  denote 
the particles  in nondecreasing order  by $\bx(\boeta) = ( x_1(\boeta), \ldots,     x_n(\boeta))$ with $\al N\leq x_1 \leq\dots\leq x_n \leq (1-\al)N$. 
We will drop the dependence on $\boeta$ and simply use $( x_1, \ldots,     x_n)$.
%
%
%we can use particle notation, i.e., we use  $\by= ( y_1, \ldots,     y_n)$ to represent $\boxi$
%and $\bx$ for $\boeta$. 
In the same way, we also denote the  configuration $\boxi$ by $\by$ with  $1\leq y_1 \leq \dots\leq y_n\leq N$ where we have dropped the dependence of $\boxi$ in $y_\alpha(\boxi)$. 
This convention will be followed for the rest of this paper. 
%
%For any configuration  $\boeta$ with $n$ particles we denote the particles $1\leq \boeta(1)\leq\dots\leq \boeta(n)\leq N$, 
% in nondecreasing order. Note that this notation is different from $\eta_k$, the number of particles at site $k$.

We  define the following distance 
on the set of configurations with $n$ particles: 
\begin{equation}\label{order}
d(\boeta, \boxi) = \sum_{\alpha=1}^n |x_\alpha-y_\alpha|
=
\min_{\sigma\in\mathscr{S}_n} \sum_{\alpha=1}^n |x_\alpha-y_{\sigma(\alpha)}|.
\end{equation}
For the second equality, observe that for any $x\leq y$ and $a\leq b$, we have $|x-a|+|y-b|\leq |x-b|+|y-a|$.

Before stating our finite speed result, we also need the notation
$r_s(\boeta,\boxi)= (\rU_{\mathscr{S}} (0,  s)  \delta_{\boeta})(\boxi)$.

\begin{lemma}\label{FiniteSpeed} 
Suppose that the eigenvalue $\bla$ satisfies the condition \eqref{eqn:condition2} with exponent $\omega$ such that $N^\omega\ll\ell$.
Let $\al,\e>0$  
and choose $\ell \geq  Nt $ for the short range dynamics cutoff. 
\begin{enumerate}[(i)]
\item
Uniformly in $\boeta$ supported on $\llbracket\al N,(1-\al)N\rrbracket$ and $t>0$, if $d(\boeta,\boxi)\geq N^\e\ell$, we have
\begin{equation}\label{eqn:finitebulk}
\mathbb{P}\left(r_s(\boeta,\boxi)>e^{-N^{\e/2}}\right)=\OO\left(N^{-D}\right)
\end{equation}
for any $D>0$. Here $\mathbb{P}$ denotes integration with respect to the Dyson Brownian Motion.
\item Uniformly in $\boeta$ supported on $\llbracket 1,N\rrbracket$ and $t>0$, if $d(\boeta,\boxi)\geq N^{\frac{1}{3}+\e}\ell^{\frac{2}{3}}$, the finite speed estimate (\ref{eqn:finitebulk}) holds.
\end{enumerate}
\end{lemma}

\begin{proof}
We first consider the case (i) corresponding to $\boeta$ supported in the bulk, but 
the reader may want to  read first the proof of (ii), written for the simpler case $n=1$
for the sake of simplicity.\\

\noindent{\it First step: definitions and dynamics.} 
Let  $\nu=N/\ell$ and $\kappa>0$ be a  fixed parameter such that $-2+\kappa<\gamma_{\alpha N}$.
For any $1\leq i\leq N$ and $x\in \mathbb{R}$, let $d_i(x)=|x-\gamma_{i}|$. Let $g_i(x)=d_i(2-\kappa)$ if $x>2-\kappa$, 
$g_i(x)=d(-2+\kappa)$ if $x<-2+\kappa$ and $g_i(x)=d_i(x)$ if $-2+\kappa\leq x\leq 2-\kappa$.  Take $\chi$ a smooth, nonnegative, compactly supported function with $\int\chi=1$, and $\psi_i(x)=\int g_i(x-y)\nu\chi(\nu y)\rd y$.
Then $\psi_i$ is smooth,
$\|\psi_i'\|_\infty\leq1$ and $\|\psi_i''\|_\infty\leq \nu$.

Moreover, consider the stopping time 
\begin{equation}
\label{eqn:stop}
\tau=\inf\left\{s\geq 0\mid \exists k\in\llbracket 1,N\rrbracket:|\la_k(s)-\gamma_k|>N^{-\frac{2}{3}}(\hat k)^{-\frac{1}{3}} \ell \right\}.
\end{equation}
For any configuration $\boxi$ with $n$ particles we define
\begin{equation}\label{eqn:order}
\psi_s(\boxi)=\sum_{\alpha=1}^n   \psi_{x_\alpha}(\la_{y_\alpha}(s\wedge \tau))=\min_{\sigma\in\mathscr{S}_n}\sum_{\alpha=1}^n\psi_{x_\alpha}(\la_{y_{\sigma(\alpha)}}(s\wedge \tau)),
\end{equation}
similarly to (\ref{order}). For the second equality, observe that 
if $\alpha \leq \beta$ and $a\leq b$, then $\psi_\alpha(a)+\psi_\beta(b)\leq \psi_\alpha(b)+\psi_\beta(a)$ (the function $a\mapsto\psi_\alpha(a)-\psi_\beta(a)$ is nondecreasing).

We define
$$\phi_{s}(\boxi)=e^{\nu\psi_s(\boxi)}, \ \ v_s(\boxi)=\phi_s(\boxi)r_{s\wedge \tau}(\boeta,\boxi).$$
Then we have (we omit the $s$ index)
\begin{align*}
\rd v(\boxi)
=& \sum_{ |j-k| \le \ell}2\xi_k(1+2\xi_j) c_{jk} \left((v(\boxi^{kj}) - v(\boxi))+ \left( \frac {  \phi(\boxi)}  {  \phi(\boxi^{kj})} - 1 \right)  v(\boxi^{kj}) \right) \rd (s\wedge\tau)  +  \left(\rd  \phi(\boxi) \right)   r(\boeta,\boxi)\\
\frac{\rd \phi(\boxi)}{\phi(\boxi)} 
=& \sum_{ \alpha=1}^n\left(\nu \psi_{x_\alpha}'(\la_{y_\alpha})\frac{\rd B_{y_\alpha }(s\wedge\tau)}{\sqrt{N}}+\nu \frac{\psi_{x_\alpha}'(\la_{y_\alpha })}{N}\sum_{j\neq y_\alpha}\frac{\rd (s\wedge\tau)}{\la_{y_\alpha}-\la_j}\right.\\
&\left.+  c_1\frac{\nu}{2N}\psi_{x_\alpha}''(\la_{y_\alpha})\rd (s\wedge\tau)+c_2\frac{\nu^2}{2N}\psi_{x_\alpha}'(\la_{y_\alpha})^2\rd (s\wedge\tau)\right)
\end{align*}
The coefficients $c_1$ and $c_2$ are  non-random positive combinatorial factors  
depending  on the locations of $i$, $\boeta,\boxi$, \nc but we will only need that they are uniformly bounded in $N$. 
We will adopt the convention to use  indices $1 \le \alpha, \beta \le n, 1 \le i, j, k \le N$.
We define
$$
X_s=\sum_{\boxi}\pi(\boxi) v_s(\boxi)^2,
$$
where $\pi=\pi^{(s)}$ is the reversible measure defined in (\ref{eqn:weight}) for the symmetric eigenvector moment flow (which is also reversible 
w.r.t its short-range cutoff version).
Then
\begin{align}
\rd X_s=&2\sum_{\boxi}\pi(\boxi)v(\boxi)\sum_{ |j-k| \le \ell}2\xi_k(1+2\xi_j) c_{jk} \left((v(\boxi^{kj}) - v(\boxi))+ \left(   \frac {  \phi(\boxi)}  {  \phi(\boxi^{kj})}  - 1 \right)  v(\boxi^{kj}) \right) \rd (s\wedge\tau)\label{1stterm}\\
&+2\sum_{\boxi}\pi(\boxi)v(\boxi) \left(\rd  \phi(\boxi) \right)   r(\boeta,\boxi)\label{2ndterm}\\
&+\sum_{\boxi}\pi(\boxi)\rd\langle v(\boxi)\rangle_{s\wedge\tau}\label{3rdterm}.
\end{align}

\noindent{\it Second step: bound on (\ref{1stterm}) and (\ref{3rdterm}).}
Using reversibility with respect to $\pi$, the first term  can be written
\begin{align}
(\ref{1stterm})=&-\sum_{\boxi} \pi(\boxi) \sum_{|j-k|\leq \ell}2\xi_k(1+2\xi_j)c_{jk}(v(\boxi^{ k j})-v(\boxi)))^2\rd (s\wedge\tau)\label{good1st}\\
&+\label{good2nd}
\sum_{\boxi} \pi(\boxi) \sum_{|j-k|\leq \ell}2\xi_k(1+2\xi_j)c_{jk}
\left(\frac{\phi(\boxi^{kj})}{\phi(\boxi)}+\frac{\phi(\boxi)}{\phi(\boxi^{kj})}-2\right)v(\boxi)v(\boxi^{ kj})\rd (s\wedge\tau).
\end{align}
Here the equality  (\ref{good1st}) is a direct application of the reversibility property, while (\ref{good2nd}) also follows from the reversibility 
as follows. Notice that 
\be
\sum_{\boxi}\pi(\boxi)v(\boxi)\sum_{ |j-k| \le \ell}2\xi_k(1+2\xi_j) c_{jk}    \frac {  \phi(\boxi)}  {  \phi(\boxi^{kj})} \nc   v(\boxi^{kj}) = \langle  g,   
\mathscr{S}   r  \rangle_\pi
+ \langle  g,   r  \rangle_\pi, \quad g = \phi^2 r
\ee
One can check that (\ref{good2nd}) follows from $\langle  g,  \mathscr{S}  r  \rangle_\pi = \langle   \mathscr{S}  g,      r  \rangle_\pi$. 

We now estimate  the term $\frac{\phi(\boxi^{kj})}{\phi(\boxi)}+\frac{\phi(\boxi)}{\phi(\boxi^{kj})}-2$ in 
(\ref{good2nd}).
 If it is  nonzero  (and  we assume first that  $j<k$)
 then  there exists $1\leq p< q\leq n$ such that  %$k=  y_q $ , 
 $y_p\leq j< y_{p+1}$, $y_{q-1}<k=y_{q}$ (recall $y_q= y_q(\boxi)$) and
\begin{align}
|\psi_s(\boxi^{kj})-  \psi_s(\boxi)|&=|(\psi_{x_{p+1}}(\la_j)+\psi_{x_{p+2}}(\la_{y_{p+1}}))+\dots+\psi_{x_{q}}(\la_{y_{q-1}})\notag
-(\psi_{x_{p+1}}(\la_{y_{p+1}})+\dots+\psi_{x_{q}}(\la_{y_{q}}))|\\
&\leq 
\sum_{\alpha=p+1}^{q}\left|
\psi_{x_{\alpha}}( \la_{y_{\alpha-1}\vee j})-\notag
\psi_{x_{\alpha}}(\la_{y_\alpha})\right|\\
%\psi_{x_{p+a}}(\la_{y_{p+a}}) %+ \psi_{x_{q+a+1}}(\la_{y_{q+a+1}})
%+ \ldots + \psi_{x_{q-1}}(\la_{y_{q-1}}) 
%-[ \psi_{x_{p+a}}(\la_{y_{p+a+1}})
%+ \ldots + \psi_{x_{q-1}}(\la_{y_{q}}) ]
%\leq \psi_s(\boxi)+C\min(|\la_j(s\wedge\tau)-\la_k(s\wedge\tau)|,\nu^{-1}).
&\leq C\min(|\la_j(s\wedge\tau)-\la_k(s\wedge\tau)|,\nu^{-1}).\label{eqn:estimate}
\end{align}
Here we have used the definition (\ref{eqn:order}) in the first equality and
for the second inequality we used:  (i) $|\psi_{x_\alpha}'|_\infty\leq 1$, (ii) $\psi_{x_\alpha}$ is flat close to the edges and (iii) if $|k-j|\leq \ell$ are bulk indices, then $|\la_k(s\wedge\tau)-\la_j(s\wedge\tau)|\leq C\nu^{-1}$
by  definition of the stopping time $\tau$. 
Note that \eqref{eqn:estimate} also holds if $j>k$, 
with a proof being identical to the case $j< k$ up to notations. 

%Similarly to (\ref{eqn:estimate}), we have $\psi_s(\boxi)\leq \psi_s(\boxi^{kj})+C\min(|\la_j(s\wedge\tau)-\la_k(s\wedge\tau)|,\nu^{-1})$. Using this inequality and 
Thanks to \eqref{eqn:estimate}, we obtain \nc
$$
\left|\frac{\phi(\boxi^{kj})}{\phi(\boxi)}+\frac{\phi(\boxi)}{\phi(\boxi^{kj})}-2\right|\leq C\ \nu^2|\la_k-\la_j|^2.
$$
This allows us to bound
$$
(\ref{good2nd})\leq C\ \frac{\nu^2}{N}\sum_{\boxi}\pi(\boxi)\sum_{k:\xi_k>0}\sum_{|j-k|\leq \ell}\nu(\boxi)\nu(\boxi^{kj})\rd (s\wedge \tau)\leq
\frac{\nu^2\ell}{N}e^{\nu\frac{\ell}{N}}
X_s\rd (s\wedge\tau).
$$
Moreover, the bracket term (\ref{3rdterm}) is easily bounded by
$$
(\ref{3rdterm})\leq C\ \sum_{\boxi}\pi(\boxi)v(\boxi)^2\sum_{\alpha=1}^n \nu^2\frac{\psi_{x_\alpha}'(\la_{y_\alpha})^2}{N}\rd (s\wedge\tau)
\leq
C\ \frac{\nu^2}{N} X_s \rd (s\wedge\tau).
$$

\noindent{\it Third step: bound on (\ref{2ndterm}).}
%The drift part of
We can bound  (\ref{2ndterm}) by
\begin{align}
\notag&2\sum_{\boxi}\pi(\boxi)v(\boxi)^2\sum_{\alpha=1}^n\left(
\nu \frac{\psi_{x_\alpha}'(\la_{y_\alpha})}{N}\sum_{j\neq y_\alpha}\frac{1}{\la_{x_\alpha}-\la_j}
+c_1\frac{\nu}{2N}\psi_{x_\alpha}''(\la_{\boxi(i)})+c_2\frac{\nu^2}{2N}\psi_{x_\alpha}'(\la_{\boxi(i)})^2
\right)\rd (s\wedge\tau)\\
\leq&\notag
\ C\ \frac{\nu^2}{N} X_s\rd (s\wedge\tau)+
2\sum_{\boxi}\pi(\boxi)v(\boxi)^2\sum_{1\leq \alpha\leq n,|j- y_\alpha|> \ell}
\nu\frac{|\psi_{x_\alpha}'(\la_{y_\alpha})|}{N}\frac{\rd (s\wedge\tau)}{|\la_{y_\alpha}-\la_j
|}\\
&+\label{eqn:inter}
2\sum_{\boxi}\pi(\boxi)v(\boxi)^2\sum_{1\leq \alpha \leq n,| j-y_\alpha|\leq \ell}
\frac{\nu}{N}\frac{\psi_{x_\alpha}'(\la_{y_\alpha})}{\la_{y_\alpha}-\la_{j}}\rd (s\wedge\tau).
\end{align}
As rigidity holds when $\tau>s$, the above sum over $|j-y_\alpha|> \ell$ is at most
$C\ \nu(\log N) \rd (s\wedge\tau)$. 

To bound the contribution of $|j-y_\alpha|\leq \ell$,  we symmetrize the summands of (\ref{eqn:inter}) into
\begin{multline}
 \frac{\nu}{N}\sum_{i< j: |i-j|\leq \ell}\frac{1}{\la_i-\la_j}
\sum_{\boxi}\pi(\boxi)v(\boxi)^2 \sum_{\alpha: y_\alpha=i}\psi'_{x_\alpha}(\la_i)
+\frac{\nu}{N}\sum_{i>  j: |i-j |\leq \ell}\frac{1}{\la_i-\la_j}
\sum_{\boxi}\pi(\boxi)v(\boxi)^2 \sum_{\alpha: y_\alpha=i} \psi'_{x_\alpha}(\la_i)
\\
=  \frac{\nu}{N}\sum_{i< j: |i-j|\leq \ell}\frac{1}{\la_i-\la_j}
\sum_{\boxi}\pi(\boxi)v(\boxi)^2\left(\sum_{\alpha: y_\alpha=i}\psi'_{x_\alpha}(\la_i)
-
\sum_{i:y_\alpha=j}\psi'_{x_\alpha}(\la_j)
\right)\\
\leq 
 \frac{\nu}{N}\sum_{i< j: |i-j|\leq \ell}\frac{1}{\la_i-\la_j}
\sum_{\boxi}\pi(\boxi)v(\boxi)^2\left(\sum_{\alpha: y_\alpha=i}\psi'_{x_\alpha}(\la_i)
-
\sum_{i:y_\alpha=j}\psi'_{x_\alpha}(\la_i) \right)
+
C\, \nu^2\frac{\ell}{N}X_s,\label{eqn:symm}
\end{multline}  
where we just replaced $\psi'_{x_\alpha}(\la_j)$ with $\psi'_{x_\alpha}(\la_i)$, up to an error at most 
$C\, \nu^2\frac{\ell}{N}X_s$, obtained by using $|\psi'_{x_\alpha}(\la_j)-\psi'_{x_\alpha}(\la_i)|/|\la_j-\la_i|\leq \|\psi'_{x_\alpha}\|_\infty\leq\nu$.  
In all the following bounds, we
consider $i$ and 
$j$ as fixed indices. We also introduce the following subsets of configurations with $n$ particles, for any $0\leq q\leq p\leq n$:
$$
\mathcal{A}_{p}=\{\boxi:\xi_i+\xi_j=p\},\
\mathcal{A}_{p,q}=\{\boxi\in\mathcal{A}_p:\xi_i=q\}.
$$
Denote $\bar\boxi$  the configuration exchanging all particles from sites $i$ and $j$, i.e. $\bar\xi_i=\xi_j$, $\bar\xi_j=\xi_i$ and $\bar\xi_k=\xi_k$ if $k\neq i, j$. 
Using  $\pi(\boxi)=\pi(\bar\boxi)$, we can bound  the sum over $\boxi$ in (\ref{eqn:symm}) by 
\begin{multline}
\frac{1}{\la_i-\la_j}\sum_{p=0}^n\sum_{q=0}^p\sum_{\boxi\in\mathcal{A}_{p,q}}\pi(\boxi)v(\boxi)^2
\left( \sum_{\alpha: y_\alpha=i}\psi'_{x_\alpha}(\la_i)
-
\sum_{\alpha:y_\alpha=j}\psi'_{x_\alpha}(\la_i) \right)\\
=
\frac{1}{\la_i-\la_j}\sum_{p=0}^n\sum_{q=0}^{\lfloor p/2\rfloor}c_q
\sum_{\boxi\in\mathcal{A}_{p,q}}\pi(\boxi)  
\Bigg [   v(\boxi)^2  
\left(\sum_{\alpha: y_\alpha=i}\psi'_{x_\alpha}(\la_i)  
-
\sum_{\alpha:y_\alpha=j}\psi'_{x_\alpha}(\la_i) \right)  \\
-
v(\bar\boxi)^2\left(
\sum_{\alpha:  \bar y_\alpha=j}\psi'_{x_\alpha}(\la_i)
-
\sum_{\alpha:  \bar y_\alpha=i}\psi'_{x_\alpha}(\la_i)
\right)
\Bigg],
\label{eqn:interm1}
\end{multline}
\nc
where  the constant $c_q=0$  if $p$ is even and $q=p/2$, and  $c_q=1$ otherwise. 
Remember that for any $a\leq b$, we have $\psi'_{a}\geq \psi'_{b}$. This implies that 
$
\sum_{\alpha: y_\alpha=i}\psi'_{x_\alpha}(\la_i)
\geq 
\sum_{\alpha: \bar y_\alpha=j}\psi'_{x_\alpha}(\la_i)
$ and
$
\sum_{\alpha:  \bar y_\alpha=i}\psi'_{x_\alpha}(\la_i)
\geq \sum_{\alpha:y_\alpha=j}\psi'_{x_\alpha}(\la_i)$
so that
\begin{equation}\label{eqn:interm2}
\sum_{\alpha: y_\alpha=i}\psi'_{x_\alpha}(\la_i)
-
\sum_{\alpha:y_\alpha=j}\psi'_{x_\alpha}(\la_i) 
\ge
\sum_{\alpha:  \bar y_\alpha=j}\psi'_{x_\alpha}(\la_i)
-
\sum_{\alpha:  \bar y_\alpha=i}\psi'_{x_\alpha}(\la_i).
\end{equation}
Equations (\ref{eqn:interm1}) and (\ref{eqn:interm2}) together with $\la_i<\la_j$ give
\begin{multline*}
\frac{1}{\la_i-\la_j}\sum_{p=0}^n\sum_{q=0}^p\sum_{\boxi\in\mathcal{A}_{p,q}}\pi(\boxi)v(\boxi)^2
\left( \sum_{\alpha: y_\alpha=i}\psi'_{x_\alpha}(\la_i)
-
\sum_{\alpha:y_\alpha=j}\psi'_{x_\alpha}(\la_i) \right)
\\
\leq 
\frac{C}{\la_i-\la_j}\sum_{p=0}^n\sum_{q=0}^{\lfloor p/2\rfloor}\sum_{\boxi\in\mathcal{A}_{p,q}}\pi(\boxi)\left(
v(\boxi)^2-v(\bar\boxi)^2 \right)
\left( \sum_{\alpha: y_\alpha=i}\psi'_{x_\alpha}(\la_i)
-
\sum_{\alpha:y_\alpha=j}\psi'_{x_\alpha}(\la_i) \right) \\
\leq
\frac{C}{|\la_i-\la_j|}\sum_{\boxi}  \pi(\boxi) \left|
v(\boxi)^2-v(\bar\boxi)^2 \right|.
\end{multline*}
where we used, in the second inequality, $\|\psi_{x_\alpha}'\|_\infty\leq 1$.
Note that transforming $\boxi$ into $\bar\boxi$ can be achieved by transferring a particle for $i$ to $j$ (or $j$ to $i$) one by one at most $n$ times, so that
\begin{multline*}
\frac{1}{|\la_i-\la_j|}\sum_{\boxi}  \pi(\boxi)  \left|
v(\boxi)^2-v(\bar\boxi)^2 \right|\leq
\frac{C}{|\la_i-\la_j|} \sum_{\boxi}  \pi(\boxi) \left(|v(\boxi)^2-v(\boxi^{ij})^2|+|v(\boxi)^2-v(\boxi^{ji})^2|\right)\\
\leq
CM \sum_{\boxi}   \pi(\boxi)  \frac{(v(\boxi)-v(\boxi^{ij}))^2+(v(\boxi)-v(\boxi^{ji}))^2}{(\la_i-\la_j)^2}+
CM^{-1}\sum_{\boxi}  \pi(\boxi) \left((v(\boxi)+v(\boxi)^{ij})^2+(v(\boxi)+v(\boxi)^{ji})^2\right)
\end{multline*}
for any $M>0$.  
We finally proved that the drift term from (\ref{2ndterm}) is bounded above by 
$$
C\ M \frac{\nu}{N}\sum_{\boxi}  \pi(\boxi) \sum_{|i-j|\leq \ell}\frac{(v(\boxi)-v(\boxi^{ij}))^2}{(\la_i-\la_j)^2}+
C M^{-1}\frac{\nu}{N}\sum_{\boxi}  \pi(\boxi) \sum_{|i-j|\leq \ell}(v(\boxi)+v(\boxi^{ij}))^2+C(\nu\log N+\nu^2\frac{\ell}{N})X_s.
$$
We chose $M=c\nu^{-1}$ with $c$ small enough so that the first sum above can be absorbed into the dissipative term (\ref{good1st}). 
The second sum above is then bounded by $\frac{\nu^2\ell}{N}e^{\nu\frac{\ell}{N}}X_s$.\\

\noindent{\it Fourth step: conclusion.}
All together, the above estimates give
$$
\frac{\rd}{\rd s}\E(X_s)\leq C(\nu\log N+ \frac{\nu^2\ell}{N})e^{\nu\frac{\ell}{N}}\E(X_s),
$$
so for our choice $\nu=N/\ell$ we have $\E(X_s)\leq Ce^{C \frac{N}{\ell}(\log N)s}$. In particular, 
$$
\E(e^{2\frac{N}{\ell}\sum_{\alpha=1}^n\psi_{x_\alpha}(\la_{y_\alpha}(s\wedge \tau))}r_{t\wedge\tau}(\boeta,\boxi)^2)\leq C e^{\frac{N}{\ell}(\log N)t}.
$$
If $d(\boxi,\boeta)\geq N^\e\ell$, then $\sum_{\alpha=1}^n\psi_{x_\alpha}(\la_{y_\alpha}(s\wedge \tau))>\ell\frac{N^\e}{N}$, so that
(remember $\ell\geq  Nt$)
$$
\E\left(r_{t\wedge\tau}(\boeta,\boxi)^2\right)\leq C e^{-c N^\e}.
$$
One concludes using Markov's inequality and $\P(\tau<t)\leq N^{-D}$.\\

\noindent   
The proof of (ii) proceeds in  exactly the same way with  only two differences: 
1.  $g_i(x)=d_i(x)$ for any $x\in\mathbb{R}$ (in particular $\psi_i$ is not made flat near  the edges);
2. $\nu$ is chosen to be  $\nu=(N/\ell)^{2/3}$. Since the full proof for edge case 
is parallel to the bulk case,  we give all details hereafter only for $n=1$.

Let  $\nu=(N/\ell)^{2/3}$. Assume that the initial configuration $\boeta$ consists in one particle at $k_0\in\llbracket 1,N\rrbracket$.
Let $d(x)=|x-\gamma_{k_0}|$ and $\chi$ as in the proof of (i). Define $\psi(x)=\int d(x-y)\nu\chi(\nu y)\rd y$ and
$$
\psi_s(k)=\psi(\la_k(s\wedge\tau)),\ \phi_s(k)=e^{\nu\psi_s(k)},\ v_s(k)=\phi_s(k)r_{s\wedge\tau}(k_0,k),
$$
where $\tau$ is defined by (\ref{eqn:stop}). Then by definition of the dynamics and the It\^o formula, we have  (here we drop  the time parameter $s$ whenever it is obvious)
\begin{align*}
\rd v(k)
=& 2\sum_{ |j-k| \le \ell} c_{jk} \left((v(j) - v(k))+ \left( \frac {  \phi(k)}  {  \phi(j)} - 1 \right)  v(j) \right) \rd (s\wedge\tau)  +  \left(\rd  \phi(k) \right)   r(k_0,k)\\
\frac{\rd \phi(k)}{\phi(k)} 
=& \nu \psi'(\la_{k})\frac{\rd B_{k}(s\wedge\tau)}{\sqrt{N}}+\left(\nu \frac{\psi'(\la_{k})}{N}\sum_{j\neq k}\frac{1}{\la_{k}-\la_j}
+ \frac{\nu}{2N}\psi''(\la_{k})+\frac{\nu^2}{2N}\psi'(\la_{k})^2\right)\rd (s\wedge\tau)
\end{align*}
Thus if we define $X_s=\sum_{k=1}^Nv_s(k)^2$, we obtain
\begin{align}
\rd X_s=&-2\sum_{|j-k|\leq \ell}c_{jk}(v(j)-v(k))^2\rd(s\wedge \tau)\label{new1}\\
&+2\sum_{|j-k|\leq \ell}c_{jk}\left(\frac{\phi(k)}{\phi(j)}+\frac{\phi(j)}{\phi(k)}-2\right)v(j)v(k)\rd(s\wedge \tau)\label{new2}\\
&+\frac{\nu}{N}\sum_k \psi''(\la_k)v(k)^2\rd(s\wedge \tau)\label{new3}\\
&+\frac{\nu^2}{N}\sum_k \psi'(\la_k)^2v(k)^2\rd(s\wedge \tau)\label{new4}\\
&+2\frac{\nu}{N}\sum_{j<k}\frac{\psi'(\la_j)v(j)^2-\psi'(\la_k)v(k)^2}{\la_j-\la_k}\rd(s\wedge\tau)\label{new5}\\
&+2\nu\sum_k\frac{\rd B_k(s\wedge\tau)}{\sqrt{N}}\psi'(\la_k)v(k)^2.\notag
\end{align}
From $\|\phi'\|_\infty\leq 1$, the definition of $\tau$ and $\nu$, we have $\nu|\phi(\la_k)-\phi(\la_j|\leq 
\nu|\la_k-\la_j|\leq C\nu|\gamma_\ell+2|=\OO(1)$ (this is where we critically used that $\nu\leq (N/\ell)^{2/3}$), so that 
$\left|\frac{\phi(k)}{\phi(j)}+\frac{\phi(j)}{\phi(k)}-2\right|\leq C\ \nu^2|\la_k-\la_j|^2$. One concludes easily that 
$(\ref{new2})$ is bounded above by  $C\nu^2\frac{\ell}{N}\rd(s\wedge\tau)X_s$. The terms (\ref{new3}) and (\ref{new4}) are of smaller  order by $\|\psi'\|_\infty\leq 1$ and $\|\psi''\|_\infty\leq \nu$.

Finally, (\ref{new5}) is of order at most
$$
\frac{\nu}{N}\sum_{j<k:|j-k|>\ell}\frac{v(k)^2}{|\la_j-\la_k|}
+
\frac{\nu}{N}\sum_{j<k:|j-k|\leq \ell}|\psi'(\la_j)|\frac{|v(j)^2-v(k)^2|}{|\la_j-\la_k|}
+
\frac{\nu^2}{N}\sum_{|j-k|\leq \ell}\|\psi''\|_\infty v(k)^2
$$
By rigidity, the first sum above has order $\nu(\log N)X_s$. The third sum is at most $\nu^2\ell/N X_s$.
Finally, the second sum is bounded using
$$
2\frac{|v(j)^2-v(k)^2|}{|\la_j-\la_k|}\leq M^{-1}(v(j)+v(k))^2+M\frac{(v(j)-v(k))^2}{(\la_j-\la_k)^2}
$$
Choosing $M=c\nu^{-1}$ for $c$ small enough, this proves that (\ref{new5}) can be absorbed into the dissipative term (\ref{new1}) plus 
an error of order $(\nu^2\ell/N) X_s$.

Using $\nu\leq (N/\ell)^{2/3}$, we  have thus  proved that
$
\frac{\rd}{\rd s}\E(X_s)\leq C\left(\nu\log N+\frac{\nu^2\ell}{N}\right)\E(X_s)\leq C\nu(\log N)\E(X_s).
$
In particular,
$$
\E(e^{2\nu\psi(\la_k(s\wedge\tau))}r_{t\wedge\tau}(k_0,k)^2)\leq e^{C\nu(\log N)t}.
$$
If $|k-k_0|\geq N^{1/3+\e}\ell^{2/3}$, then $\psi(\la_k(s\wedge\tau))\geq N^\e(\ell/N)^{2/3}=N^\e\nu^{-1}$, so we obtained
$$
\E(r_{t\wedge\tau}(k_0,k)^2)\leq e^{C\nu(\log N)t-N^\e},
$$
which is exponentially small: $(N/\ell)^{2/3}(\log N)t=\OO(\log N)$ as $\ell\geq Nt$ and $t\leq 1$.  
By the Markov inequality, we  have thus proved  the  part (ii) of the lemma. 
\end{proof}

\section[Relaxation to equilibrium for $\mathrm{t\hspace{0.05cm}{\gtrsim}\hspace{0.05cm}N^{-1 }}$]{Relaxation to equilibrium for ${t\,{\gtrsim}\,N^{-1 }}$}
\label{sec:relax}

The maximum inequality (\ref{eqn:StQUE}) allowed to prove convergence of the eigenvector moment flow along the whole spectrum, in Section \ref{sec:max}, for $t\gtrsim N^{-1/4}$. Assume that, for some reason, the maximum of this flow is always obtained for configurations supported in the bulk.
Then  we can make the approximation $\Im m(\la_k+\ii\eta)\sim 1$ in (\ref{eqn:StQUE}), and  we obtain
$$
S_t'\leq -\frac{1}{\eta}S_t+\frac{N^\xi}{N^{1/2}\eta^{3/2}}
$$
assuming the optimal isotropic local semicircle law with a tiny error $N^\xi/\sqrt{N\eta}$.
Choosing $\eta=N^{-1+\e}$ for some small $\e>0$ then gives, by Gronwall, a relaxation time of order $\gtrsim N^{-1}$.
The purpose of this section is to make this  argument rigorous by using 
 the  finite speed of propagation for  the eigenvector moment flow, i.e.,   Lemma \ref{FiniteSpeed}.

\subsection{Statement of the result.\ }
The initial matrix is denoted $M_N=M_N(0)$, it satisfies the local semicircle law, and its eigenvalues follow the usual Dyson Brownian motion dynamics.

Let $G^{(s)}_N$ (resp. $G^{(h)}_N$)  be a sequence of $N\times N$ random matrices from the Gaussian orthogonal (resp. unitary) ensemble 
(normalized with limiting spectral measure supported on $(-2,2)$, for example).  Note that in this section $G$ stands for a Gaussian matrix, not its Green function.

\begin{theorem}\label{thm:deterministic}
Let $\e$ be any arbitrarily small positive constant and $t=N^{-1+\e}$.
Assume that,   for   a deterministic sequence of matrices $(M_N)_{N\geq 1}$
and  a 
sequence of  unit vectors $\bq=\bq _N$, 
we have,  for any $\omega>\xi>0, D>0$ and $N$ large enough  (depending on these parameters), 
\begin{equation}\label{technicalGreen}
\mathbb{P}\left( A_1(\bq ,\omega,\xi,N) \mid M_N \right)\geq 1-N^{-D}.
\end{equation}
Here we used the notation (\ref{eqn:condition1}) and $\mathbb{P}(\cdot\mid M_N)$ denotes probability   with respect to the matrix Dyson Brownian motion path 
with  the initial condition fixed by $M_N$.  
Then the 
asymptotic normality of  bulk eigenvectors of 
$M_N+\sqrt{t}\,G^{(s)}_N$ holds. More precisely,
if $\boldu$ is the eigenbasis of $M_N+\sqrt{t}\,G^{(s)}_N$,
for any polynomial $P$ there exists $c>0$ such that
\begin{equation}\label{eqn:detRelax}
\sup_{I\subset\llbracket \al N, (1-\al)N\rrbracket,|I|=m,|\bq|=1}
\left|\E\left(P\left(\left(N|\langle \bq,u_{k}\rangle|^2\right)_{k\in I}\right) \right)-\E\left( P\left((|\mathscr{N}_j|^2)_{j=1}^m\right)\right)\right|\leq N^{-c}.
\end{equation}

If moreover (\ref{technicalGreen}) holds for any given sequence $(\bq_N)_{N\geq 1}$, then
any bulk eigenvector of $M_N+\sqrt{t}\,G^{(s)}_N$ have asymptotically independent normal entries (the analogue of Corollary \ref{cor:Gauss}) and each eigenvector satisfy local quantum unique ergodicity (the analogue of Corollary \ref{cor:QUE}).

Similar results hold for the Hermitian matrices  $M_N+\sqrt{t}\,G^{(h)}_N$.
\end{theorem}

The Green function  in $A_1$ appearing    in    (\ref{technicalGreen}) is with respect to the matrix  $(M_N(s))_{s\geq 0}$ with  $M_N=M_N(0)$  being 
 the initial matrix and $M_N(s)$ the value at time $t$ of a (matrix)  Dyson Brownian Motion.\\

Theorem \ref{thm:deterministic} means that,  the  initial structure of bulk eigenvectors completely disappears with the addition of a small noise, 
provided that the initial matrix satisfies a strong form of semicircle law.
If the initial condition is a generalized Wigner matrix, the matrix Dyson Brownian motion is again a generalized Wigner ensemble 
after rescaling. In this case, the asymptotic normality of the eigenvectors was already proved in Theorem \ref{thm:main}
and therefore the conclusion of Theorem \ref{thm:maxPrincipleLoc} was proved as well. 
The key point of Theorem \ref{thm:deterministic}  
lies in that it holds for deterministic initial matrices, provided that the local  isotropic semicircle law holds.\\

Note that by standard perturbation theory Theorem \ref{thm:deterministic} in general does not hold for  $t\ll N^{-1}$. 
Recall that Dyson's conjecture states that  the  relaxation time to local equilibrium 
for bulk eigenvalues under the DBM is  $t\sim N^{-1}$. 
Thus Theorem~\ref{thm:deterministic} is  the analogue of this conjecture  in the context of bulk eigenvectors.

\begin{remark}
Theorem \ref{thm:deterministic} gives optimal relaxation speed for dynamics of bulk eigenvectors provided that the 
local law holds along the whole spectrum, i.e. condition (\ref{technicalGreen}) holds.
One may be interested in the dynamics relaxation only locally, i.e. proving QUE only for certain eigenvectors with corresponding energy $\lambda_i$ around $E_0=\gamma_{k_0}\in(-2.2)$. Then as an input,  the local law is only needed in a small window around $E$. 

More precisely, let $c>\epsilon>0$ be fixed (remember $t=N^{-1+\e}$). Assume that (\ref{technicalGreen}) holds for any $\omega>\xi>0$   in the smaller domain (replacing the original domain defined in  (\ref{eqn:omega}))
\begin{equation}
\widetilde{\b S} = \widetilde{\b S}(\omega, N)=\hb{z=E+\ii\eta \in\C \col \abs{E-E_0}\leq N^{-1+c}\,,\, N^{-1 + \omega} \leq \eta \leq \omega^{-1}}.
\end{equation}
Then the conclusion (\ref{eqn:detRelax}) holds after restricting the $\sup$ to $I\subset \llbracket 
k_0-N^c/10,
k_0+N^c/10
\rrbracket$.

To summarize,  the  optimal time relaxation result, Theorem \ref{thm:deterministic}, 
 can be made local in the spectrum, because the key input in this result, the finite speed of propagation 
Lemma \ref{FiniteSpeed},  holds locally.  The modifications  needed to prove these local 
versions are obvious and we leave them to interested readers. 
\end{remark}

We will prove Theorem \ref{thm:deterministic}  by using  the maximum principle locally. For this purpose, we will use 
the finite speed of propagation estimate,  Lemma \ref{FiniteSpeed}.  This will be explained in the next subsections.

\subsection{ Flattening of initial condition at the edge.\ }

Let $\al>0$ be a fixed small number. 
We define the following flattening and averaging operators on the space of functions of configurations with $n$ points: any $a\in\llbracket 1,N/2\rrbracket$,
\begin{align*}
&({\rm Flat}_a (f))(\boeta)=f(\boeta)\ {\rm if}\ \boeta\subset\llbracket a,N+1-a\rrbracket,\ 1\ {\rm otherwise},\\
&{\rm Av}(f)=\frac{1}{|\llbracket\al N,2\al N\rrbracket|}\sum_{a\in\llbracket\al N,2\al N\rrbracket}{\rm Flat}_a(f).
\end{align*}

We can write 
\begin{equation}\label{eqn:ak}
{\rm Av}(f)(\boeta)=a_{\boeta} f(\boeta)+(1-a_{\boeta})
\end{equation}
for some coefficient $a_{\boeta}\in[0,1]$ ($a_{\boeta}=0$ if $\boeta\not\subset\llbracket\al N,(1-\al)N\rrbracket$, $1$ if $
\boeta\subset\llbracket 2\al N,(1-2\al)N\rrbracket$). We will only use the elementary property
\begin{equation}\label{eqn:propa}
|a_{\boeta}-a_{{\boxi}}|\leq C\ \frac{d(\boeta,\boxi)}{N}. 
\end{equation}

For a general number of particles $n$, consider now the following modification of the eigenvector moment flow (\ref{ve}).
We only keep the short-range dynamics (depending on a parameter $\ell$) and modify the initial condition to be flat when there is a particle close to the edge:
\begin{align}\label{eqn:modMom}
&\partial_t g_{\bla, t} =  \mathscr{S}(t) g_{\bla, t},\\
&g_{\bla,0}(\boeta)=({\rm Av}f_{\bla,0})(\boeta),\notag
\end{align}
We will abbreviate $g_{\bla,t}(\boeta)$ by $g_t(\boeta)$,
and $f_{\bla,t}(\boeta)$ by $f_t(\boeta)$ (for $n=1$,  we write these functions as  $f_t(k)$ and $g_t(k)$ where $\boeta$ is the configuration with $1$ particle at $k$).  We remind the reader that $f_t(\boeta)$ can be  define either by  \eqref{feq} or by the solution of 
the equation \eqref{ve}. In particular, $f_t(k)$ is the conditional expectation of $|\langle\bq, u_k(t)\rangle|^2$ given 
$\bla$, i.e., 
\be
f_t(k)=  N  \E \left( |\langle\bq, u_k(t)\rangle|^2 \mid \bla \right) 
\ee
where $\bq$ is a fixed unit vector.  In all our application, the initial data $f_{\bla,0})(\boeta)$ is independent of $\bla$
and given by \eqref{feq}  with $t=0$. For $g_{\bla, t}$, we can only understand it as the solution to \eqref{eqn:modMom}. 

For small time $t$, by finite speed of propagation we will prove that $g=1$ (up to exponentially small corrections) close to the edge, so that the maximum principle for the dynamics (\ref{eqn:modMom}) can be localized in the bulk.

We first prove that for these modified dynamics, the isotropic law holds in the following sense. The following result is
deterministic.

\begin{lemma}\label{ILSLFlat} 
Let $\e>0$ be a fixed small number, $t=N^{-1+\e}$ and $\ell=N^{\delta} Nt$ for some $\delta>0$ (here $\ell$
is the short-range dynamics cutoff parameter). 
Then   there exist  (small) positive constants $\omega_0,\xi_0$ such that the following holds.
Assume that for some  $0<\omega<\omega_0$, $0<\xi<\xi_0$, 
$(M_N(s))_{0\leq s\leq 1}$ is in $A_1(\bq,\omega,\xi,N)$. Assume moreover that $(\ref{eqn:finite})$ holds.
Let $z$  satisfy   $-3<\Re(z)<3$ and  $N^{-1+2\omega}<\Im(z)<\min(N^{-1+\delta/2},N^{-3/4})$.
Then we have
\begin{equation}\label{eqn:bound}
 \left|\Im\sum_{k=1}^N\frac{1}{N}\frac{g_t(k)}{z-\la_k}-\Im m(z)\right|\leq  C N^{\xi+\omega}\left(\sqrt{\frac{\Im m(z)}{N\eta}}+\frac{1}{N\eta}\right)+C\frac{ \ell  N^{2\omega}}{N} 
\end{equation}
where  $C$ depends only on $\xi,\omega,\nu,\e$. Moreover,
consider the case of $n$ particles. Let $k_0\in\llbracket 1,N\rrbracket$ and $z=\la_{k_0}+\ii\eta$. Then for any configuration 
$\boeta$ containing at least one particle at $k_0$ we have
\begin{multline}\label{eqn:multibound}
\Im \sum_{k=1}^N\frac{1}{N}\frac{g_t(\boeta^{k_0k})}{z-\la_k}
-
\Im m(z)\left(a_{\boeta}f_t(\boeta\backslash k_0)+(1-a_{\boeta})\right) \leq C \left( N^{\xi+n \omega}\left(\sqrt{\frac{\Im m(z)}{N\eta}}+\frac{1}{N\eta}\right)+\frac{ \ell N^{2\omega}}{N}\nc\right)\\
\nc
\end{multline}
where $\boeta\backslash k_0$ stands for the configuration $\boeta$ with one particle removed from site $k_0$. \nc
\end{lemma}

\begin{proof}
We first show that the difference between  $g_t(k)= (\rU_{\mathscr{S}}(0,t){\rm Av}f_0)(k)$ and $({\rm Av}\rU_{\mathscr{B}}(0,t)f_0)(k)$ is small. 
More precisely, we can bound the left hand side of (\ref{eqn:bound})
by $|{\rm (i)}|+|{\rm (ii)}|+|{\rm (iii)}|$ where
\begin{align*}
&{\rm (i)}=\Im\sum_{k=1}^N\frac{1}{N}\frac{(\rU_{\mathscr{S}}(0,t){\rm Av}f_0)(k)-({\rm Av}\rU_{\mathscr{S}}(0,t)f_0)(k)}{z-\la_k},\\
&{\rm (ii)}=\Im\sum_{k=1}^N\frac{1}{N}\frac{({\rm Av}\rU_{\mathscr{S}}(0,t)f_0)(k)-({\rm Av}\rU_{\mathscr{B}}(0,t)f_0)(k)}{z-\la_k},\\
&{\rm (iii)}=\Im\sum_{k=1}^N\frac{1}{N}\frac{({\rm Av}\rU_{\mathscr{B}}(0,t)f_0)(k)}{z-\la_k}-\Im m(z).
\end{align*}
The term (i) will be controlled by finite speed of propagation;  (ii) will be controlled by  Lemma \ref{cut2}, and (iii) by  the isotropic local semicircle law. 

To bound (i), we write   
\begin{equation}\label{eqn:(i)}
(\rU_{\mathscr{S}}(0,t){\rm Av}f_0)(k)-({\rm Av}\rU_{\mathscr{S}}(0,t)f_0)(k)=
\frac{1}{\al N}\sum_{a\in\llbracket\al N,2\al N\rrbracket}\left(\rU_{\mathscr{S}}(0,t){\rm Flat}_a f_0-{\rm Flat}_a\rU_{\mathscr{S}}(0,t) f_0\right)(k).
\end{equation}
To control the above terms, first assume that $a+ \ell N^\omega < k$
 Denote $(f\1_{\geq a})(x)=f(x)\1_{x\geq a}$, and similarly for $f\1_{< a}$.  Then
\begin{align}
(\rU_{\mathscr{S}}(0,t){\rm Flat}_a f_0)(k)
&=\notag
\left(\rU_{\mathscr{S}}(0,t)(f_0\1_{\geq a})\right)(k)
+
\left(\rU_{\mathscr{S}}(0,t)\1_{<a}\right)(k)\\
&=\notag
\left(\rU_{\mathscr{S}}(0,t) (f_0\1_{\geq a})\right)(k)+\OO(e^{-N^{c}})\\
&=\notag
\left(\rU_{\mathscr{S}}(0,t)(f_0\1_{\geq a})\right)(k)+
\left(\rU_{\mathscr{S}}(0,t)(f_0\1_{<a})\right)(k)+
\OO(e^{-N^{c}})\\
&=\notag
\left(\rU_{\mathscr{S}}(0,t)f_0\right)(k)+
\OO(e^{-N^{c}})\\
&=
({\rm Flat}_a \rU_{\mathscr{S}}(0,t) f_0)(k)+\OO(e^{-N^{c}})\label{eqn:cut1}
\end{align}
In the above lines, we used the finite speed of propagation Lemma \ref{FiniteSpeed} in the second and third equalities, namely $(\rU_{\mathscr{S}}(0,t)\delta_x)(k)=\OO(e^{-N^{c}})$ for any $x\leq a$ and $k\geq a+\ell N^\omega$
(the case $x>a/2$ follows from part (i) of Lemma \ref{FiniteSpeed}, the case $x\leq a/2$ from part (ii)
and $a\in\llbracket\al N,2\al N\rrbracket$).

\noindent
For $k< a-\ell N^\omega$, in the same way we obtain
\begin{equation}
(\rU_{\mathscr{S}}(0,t){\rm Flat}_a f_0)(k)= 1 + \OO(e^{-N^{ c}}) =({\rm Flat}_a\rU_{\mathscr{S}}(0,t) f_0)(k)+\OO(e^{-N^{ c}}).
  \label{eqn:cut2}
\end{equation}

\noindent
For $a- \ell N^\omega \leq k\leq a+  \ell N^\omega $, as $\rU_{\mathscr{S}}$ is a bounded in $L_\infty$ we have
\begin{equation}
|(\rU_{\mathscr{S}}(0,t){\rm Flat}_a f_0)(k)-({\rm Flat}_a\rU_{\mathscr{S}}(0,t) f_0)(k)|\leq 2\sup_k f_0(k)\leq C N^\omega.\label{eqn:cut3}
\end{equation}
Equations (\ref{eqn:cut1}), (\ref{eqn:cut2}), (\ref{eqn:cut3}) together imply that (\ref{eqn:(i)}) and therefore (i) are bounded by 
$C\  \ell N^{2\omega}/N$.

To bound the term (ii), define the reversed dynamics   $\rU_{\mathscr{S}}^*$ by  
\begin{equation}\label{eqn:reversed}
\partial_{\sigma} \rU_{\mathscr{S}}^* (s,\sigma) = \mathscr{S}(t-\sigma)  {\rU}_{\mathscr{S}}^* (s,\sigma),
\end{equation}
and $s$ is always set to be $=0$ ] and similarly for 
$\rU_{\mathscr{B}}^*$. Notice that Lemma \ref{cut2} holds for these time-reversed dynamics, the proof is unchanged. Thus we have 
\begin{multline*}
|({\rm Av}\rU_{\mathscr{S}}(0,t)f_0)(k)-({\rm Av}\rU_{\mathscr{B}}(0,t)f_0)(k)|\leq
|(\rU_{\mathscr{S}}(0,t)f_0)(k)-(\rU_{\mathscr{B}}(0,t)f_0)(k)|\\
=\frac{1}{\pi(k)}|\langle f_0,({\rU}_{{\mathscr{S}}}^*(0,t)-{\rU}_{{\mathscr{B}}}^*(0,t))\delta_k\rangle_\pi|
\leq C\ N^\omega\frac{Nt}{\ell},
\end{multline*}
where the first inequality follows from  (\ref{eqn:ak}), and the second follows from Lemma \ref{cut2}.  
This proves that
$
|{\rm (ii)}|\leq N^{\omega}  \frac { N t} \ell ( \frac { N^\xi}{ N \eta} + \sqrt{\frac{\Im m(z)}{N\eta}}+  \im m(z)),
$ 
where we used the local semicircle law, i.e. our matrix is in $A_1(\bq,\omega,\xi,N)$ from (\ref{eqn:condition1}).
We therefore have $|{\rm (ii)}|\leq N^{\xi+\omega}\sqrt{\frac{\Im m(z)}{N\eta}}$
provided that  $N t/\ell\leq 1/(N\eta)^{1/2}$,  which follows from  our assumptions
on $t, \ell$ and  $\Im(z)\geq N^{-1+2\delta}$.

Concerning the error term (iii), we proceed as follows. Let $m_0$ be the index such that $|\Re(z)-\gamma_{m_0}|=\inf_{1\leq i\leq N}\{|\Re(z)-\gamma_{i}|\}$.
Then 
$$
\Im\sum_{k=1}^N\frac{1}{N}\frac{({\rm Av}\rU_{\mathscr{B}}(0,t)f_0)(k)}{z-\la_k}
=
\Im\sum_{|k-m_0|\leq N\sqrt{\eta}}\frac{1}{N}\frac{({\rm Av}\rU_{\mathscr{B}}(0,t)f_0)(k)}{z-\la_k}+\OO\left(\frac{N^\omega}{N}\sum_{i> N\sqrt{\eta}}\frac{\eta}{\eta^2+(i/N)^2}\right)
$$
where we use that $\|f_0\|_\infty\leq N^\omega$ (which follows from the condition \eqref{eqn:condition3}). For any function $f$
we write ${\rm (Av)} f(k)=a_k f(k)+(1-a_k)$ with the notation from (\ref{eqn:ak}). We obtain
\begin{multline}
\Im\sum_{k=1}^N\frac{1}{N}\frac{({\rm Av}\rU_{\mathscr{B}}(0,t)f_0)(k)}{z-\la_k}
=
\Im\sum_{|k-m_0|\leq N\sqrt{\eta}}\frac{1}{N}\frac{a_k f_t(k)+(1-a_k)}{z-\la_k}+\OO(N^\omega\sqrt{\eta})\\
=
\Im\sum_{|k-m_0|\leq N\sqrt{\eta}}\frac{1}{N}\frac{a_{m_0} f_t(k)+(1-a_{m_0})}{z-\la_k}
+
\Im\sum_{|k-m_0|\leq N\sqrt{\eta}}\frac{1}{N}\frac{(a_k-a_{m_0}) f_t(k)+(a_{m_0}-a_k)}{z-\la_k}
+\OO(N^\omega\sqrt{\eta})\label{eqn:aver}.
\end{multline}
Moreover, the first sum above is equal to
$$
a_{ m_0}
\Im\sum_{k=1}^N\frac{1}{N}\frac{f_t(k)}{z-\la_k}
+(1-a_{ m_0})\Im\sum_{k=1}^N\frac{1}{N}\frac{1}{z-\la_k}
+
\OO(N^\omega\sqrt{\eta})\\
=
\Im m(z)+\OO\left(N^{\xi}\sqrt{\frac{\Im m(z)}{N\eta}}\right)+\OO\left(N^\omega\sqrt{\eta}\right)
$$
where we used $(M_N,\bla)\in A(\bq,\omega,\xi,\nu,N)$. From (\ref{eqn:propa}),  we have 
$|a_k-a_{m_0} | \le \sqrt \eta N^\om$ and  the second sum in (\ref{eqn:aver}) can be bounded by
$
\OO(N^\omega\sqrt{\eta})
$,  which is smaller than $N^\omega/(N\eta)$ for $\eta\leq N^{-3/4}$.
Gathering all estimates, we obtain that (\ref{eqn:bound}) holds.

In the case of general $n$, to prove (\ref{eqn:multibound}), we proceed in the same way. As the term of type (i) is also bounded by finite speed of propagation, we just need to prove that
$$
\Im \sum_{k=1}^N\frac{1}{N}\frac{{\rm Av}\rU_{\mathscr{S}}(0,t)f_0(\boeta^{k_0k})}{z-\la_k}
=
m(z)\left(a_{\boeta}f_t(\boeta\slash k_0)+(1-a_{\boeta})\right)+\OO\left( N^{\xi+n \omega}\left(\sqrt{\frac{\Im m(z)}{N\eta}}+\frac{1}{N\eta}\right)\right).
$$
Thanks to Lemma \ref{cut2} it is sufficient to prove the above estimate replacing $\rU_{\mathscr{S}}$ by $\rU_{\mathscr{B}}$. We also can restrict the summation to $|k-k_0|\leq N\sqrt{\eta}$. Then,
similarly to the $n=1$ case,
we write
\begin{multline*}
{\rm Av}\rU_{\mathscr{B}}(0,t)f_0(\boeta^{k_0k})={\rm Av}f_t(\boeta^{k_0k})=a_{\boeta^{k_0k}}f_t(\boeta^{k_0k})+(1-a_{\boeta^{k_0k}})\\
=\left(a_{\boeta}f_t(\boeta^{k_0k})+(1-a_{\boeta})\right)+\left((a_{\boeta^{k_0k}}-a_{\boeta})f_t(\boeta^{k_0k})+(a_{\boeta}-a_{\boeta^{k_0k}})\right).
\end{multline*}
Using (\ref{eqn:propa}) and $|k-k_0|\leq N\sqrt{\eta}$ to bound the above second term, we are left with proving that
$$
a_{\boeta}\Im \sum_{k=1}^N\frac{1}{N}\frac{f_t(\boeta^{k_0k})}{z-\la_k}+(1-a_{\boeta})\Im \sum_{k=1}^N\frac{1}{N}\frac{1}{z-\la_k}
=
m(z)\left(a_{\boeta}f_t(\boeta\slash k_0)+(1-a_{\boeta})\right)+\OO\left( N^{\xi+n \omega}\left(\sqrt{\frac{\Im m(z)}{N\eta}}+\frac{1}{N\eta}\right)\right).
$$
The second sum above is properly estimated by $m(z)$ because we are in a good set. Concerning the first sum, 
its contribution is not trivial if $a_{\boeta}\neq 0$, in particular $k_0\in\llbracket \al N,(1-\al)N\rrbracket$.
Then $\Im m(z)\sim 1$ and this first sum can
be estimated exactly as in (\ref{eqn:condit1}), (\ref{eqn:condit2}). This concludes the proof.
\end{proof}

\subsection{Localized maximum principle.\ } The following result states that, for a typical initial conditions and a generic eigenvalue path,
the relaxation time of the bulk eigenvectors is of order at most $N^{-1+\e}$ for any small $\e>0$.

\begin{theorem}\label{thm:maxPrincipleLoc}
Let $n\in\NN$, $\al,\e>0$ be arbitrarily small constants and $t=N^{-1+\e}$.
Then there exists a constant $\omega_0$ such that the following holds.

Assume that for some  $0<\xi<\omega<\omega_0$, 
$(M_N(s))_{0\leq s\leq 1}$ is in $A_1(\bq,\omega,\xi,N)$. Assume moreover that $(\ref{eqn:finite})$ holds.
Let $f$ be a solution of the eigenvector moment flow (\ref{ve}) with initial matrix $M_N$ and path $\bla$.
Then there exists $c>0$ such that for large enough $N$ we have
\be\label{fest2}
\sup_{\boeta:\mathcal{N}(\boeta)=n,\boeta\subset\llbracket \al N,(1-\al)N\rrbracket}\left|f_t(\boeta)-1\right|\leq C N^{- c}.\nc
\ee
\end{theorem}

\begin{proof}
As $\al$ is arbitrary we just need to prove the result for $\al$  replaced by $3\al$. Moreover,
we only need to prove (\ref{fest2}) with $f_t(\boeta)$ replaced by $g_t(\boeta)$ solving the cutoff dynamics 
(\ref{eqn:modMom}).  Indeed, we have
$$
f_t(\boeta)-g_t(\boeta)
=
\frac{1}{\pi(\boeta)}\langle f_0,(\rU^*_{\mathscr{B}}(0,t)-\rU^*_{\mathscr{S}}(0,t))\delta_{\boeta} \rangle_\pi
+
\frac{1}{\pi(\boeta)}\langle \rU_{\mathscr{S}}(0,t)(f_0-g_0),\delta_{\boeta} \rangle_\pi.
$$
where we used the notation $(\ref{eqn:reversed})$ for the time-reversed dynamics.
From Lemma \ref{cut2}  (which holds also for the time-reversed dynamics)  and  the bound $\|f_0\|_\infty\leq N^\omega$,   the first term on the right hand side of the equation is bounded by $N^{1+\omega} t/\ell$. By  the finite speed of propagation Lemma \ref{FiniteSpeed}, the second term is exponentially small
(remember that $f_0(\boxi)-g_0(\boxi)=0$ if $\boxi\subset\llbracket 2\al N,(1-2\al)N\rrbracket$ and $\boeta$ is supported in $\llbracket 3\al N,(1-3\al)N\rrbracket$). We therefore just need to show that
$$
\sup_{\boeta:\mathcal{N}(\boeta)=n,\boeta\subset\llbracket 3\al N,(1-3\al)N\rrbracket}\left|g_t(\boeta)-1\right|\leq C N^{-\e}.
$$
We will prove that such an estimate holds for any $\al>0$ by induction on $n$. Assume there is just one particle.
Following the idea from the proof of Theorem \ref{thm:maxPrinciple}, for a given $0\leq s\leq t$ let $k_0$ be an index such that
$g_s(k_0)=\sup_k\{g_s(k)\}$. 
We consider two possible cases: if $g_s(k_0)-1\leq N^{-10}$ then there is nothing to prove. If $g_s(k_0)-1\geq N^{-10}$,
then from the finite speed of propagation  assumption (i.e., we are in the set $\mathcal{A}$),  $k_0$ is in the bulk, i.e.,  $k_0\in\llbracket \frac{\al}{2}N,(1-\frac{\al}{2})N\rrbracket$ (the reason is that if $k_0$ were near the edges, then $g_s(k_0)-1$ is exponentially small).
We then have  
\begin{multline*}
\partial_s g_s(k_0)=(\mathscr{S}(s)g_s)(k_0)=\frac{1}{N}\sum_{j\neq k_0,|j-k_0|\leq \ell}\frac{g_s(j)-g_s(k_0)}{(\la_j-\la_{k_0})^2}
\\
\leq
\frac{1}{\eta}\sum_{j\neq k_0,|j-k_0|\leq \ell}\frac{1}{N}\frac{\eta g_s(j)}{(\la_j-\la_{k_0})^2+\eta^2}
-
\frac{g_s(k_0)}{\eta}\sum_{j\neq k_0,|j-k_0|\leq \ell}\frac{1}{N}\frac{\eta}{(\la_j-\la_{k_0})^2+\eta^2}.
\end{multline*}
If $\ell\gg N\eta$ (which we obviously can assume, as we will chose $\eta=N^{-1+c}$ for some small $c>0$,
extending the above sums to all indices $j$ induces an error $\eta N^{1+\omega}/\ell$ where we have used that $\| g_s\|_\infty \le
\| g_0\|_\infty \le N^\om$. Combining this fact   with   Lemma \ref{ILSLFlat} and the  rigidity of eigenvalues which follows from that the path $\bla$ is 
assumed to be in the set $A_2(\om, N)$ defined in \eqref{eqn:condition2},  we have proved (here $z=\la_{k_0}+\ii\eta$) that
$$
\partial_s (g_s(k_0)-1)\leq -\frac{\Im m(z)}{\eta}(g_s(k_0)-1)+\OO\left(N^{\omega+\xi}\left(\frac{(\Im m(z))^{1/2}}{\eta^{3/2}N^{1/2}}+\frac{1}{N\eta^2}\right)+\frac{ \ell N^{2\omega}\nc}{N\eta}\right)+\OO\left(\frac{N^{1+\omega}}{\ell}\right).
$$
As $k_0\in \llbracket\frac{\al}{2}N,(1-\frac{\al}{2})N\rrbracket$, we have $\Im m(z)\sim 1$. Moreover,  the second error term $\OO\left(\frac{N^{1+\omega}}{\ell}\right)$ is dominated by the first one ( recall that the  cutoff parameter $\ell$ in  Lemma \ref{ILSLFlat} \nc  satisfies $\ell= N^{\delta}Nt$ and $\eta\leq N^{-1+\delta/2}$).
Denote by $S_s=\sup_k (g_s(k)-1)$ and we choose the parameters so that $\eta=N^{-1+\frac{\e}{2}}$ and $\omega_0\leq \e/10$. 
We proved that if $S_s\geq N^{-10}$ then
$$
\partial_s S_s\leq -\frac{c}{\eta}S_s+C\left(\frac{N^{\omega+\xi}}{\eta^{3/2}N^{1/2}}+\frac{\ell N^{2\omega}}{N\eta}\nc\right)
\leq -c N^{1-\frac{\e}{2}}S_s+CN^{1-3\e/4} . 
$$
 By Gronwall's lemma, we obtain $S_t=\OO(N^{-\e/4})$. This concludes the proof for $n=1$.

For general $n$, as in the 1-particle case we can assume that $\sup_{\boeta} g_t(\boeta)$ is achieved for some $\boxi\subset\llbracket\frac{\al}{2}N,(1-\frac{\al}{2})N\rrbracket$. Then the analogue of  (\ref{eqn:diffmult}) holds with $f$ replaced by $g$. The first sum in (\ref{eqn:diffmult}) then can be evaluated using (\ref{eqn:multibound}):
$$
\frac{1}{N\eta}\sum_{j\neq k_r}\frac{\eta g_s(\boxi^{k_r j})}{(\la_{k_r}-\la_j)^2+\eta^2}=\Im m(\la_{k_r}+\ii\eta)
(a_{\boxi}f_s(\boxi\backslash k_r)+(1-a_{\boxi})).
$$
From the result at rank $n-1$ with $\al$ replaced by $\al/10$, we know that for $s\in[t/2,t]$ we have
$$
f_s(\boxi\backslash k_r)=g_s(\boxi\backslash k_r)+\OO(N^{-c})=1+\OO(N^{-c}).
$$
This proves that
$$
\partial_s (g_s(\boxi)-1)\leq -\frac{\Im m(z)}{\eta}(g_s(\boxi)-1)+\OO\left(\frac{(\Im m(z))^{1/2}}{\eta^{3/2}N^{1/2}}\ell  N^{n\omega+\xi}+\frac{\Im m(z)}{\eta}N^{-c}+\frac{\ell N^{2\omega}}{N\eta}\nc\right).
$$
One now can conclude the proof as in the $n=1$ case.
\end{proof}

\noindent 
{\it Proof  of Theorem \ref{thm:deterministic}}
We can assume  that the trajectory $(M_t)_{0\leq t\leq 1}$ is in $A_1(\bq,\omega,\xi,N)\cap A_2(\omega,N)\cap A_3(\omega,N)$. Indeed, as noted in Section 4, $A_1\subset A_2\cap A_3$, and  the complement of $A_1$ has measure at most $N^{-D}$, 
which induces negligible error terms in the universality statements.
For the same reason, thanks to Lemma
\ref{FiniteSpeed}, we can assume that the following finite speed of propagation holds: for any small $c,\al>0$, uniformly $\boeta$ supported in the $\llbracket \al N,(1-\al)N\rrbracket$ and $d(\boeta,\boxi)>N^c\ell$, for large enough $N$ we have 
\begin{equation}\label{eqn:finite}
r_s(\boeta,\boxi)<e^{-N^{c/2}}.
\end{equation}
Under the above two assumptions, we apply Theorem \ref{thm:maxPrincipleLoc}, which 
proves the first statement of Theorem \ref{thm:deterministic}.
 The last two statements of Theorem \ref{thm:deterministic} easily follow by the  arguments used in Sections
\ref{secInd} and \ref{secQUE}.

\nc
\setcounter{equation}{0}
\setcounter{theorem}{0}
\renewcommand{\theequation}{A.\arabic{equation}}
\renewcommand{\thetheorem}{A.\arabic{theorem}}
\appendix
\setcounter{secnumdepth}{0}
\section[Appendix A\ \ \ Continuity estimate for $\mathrm{t\hspace{0.05cm}{\lesssim}\hspace{0.05cm}N^{-1/2}}$]
{Appendix A\ \ \ Continuity estimate for ${t\, {\lesssim}\,N^{-1/2}}$}

The main result in Section \ref{sec:relax}, Theorem \ref{thm:maxPrincipleLoc}, asserts  the asymptotic normality 
of eigenvector components for Gaussian divisible ensembles for $t\gtrsim N^{-1}$. To prove Theorem \ref{thm:main} for bulk eigenvectors,  in this appendix
we remove the small Gaussian components of the matrix elements. 
As we saw in Section \ref{sec:Proof}, one way to proceed consists in a Green function comparison theorem.
Here, we proceed in a different way: the Dyson Brownian motion preserves the local structure of generalized Wigner matrices up to time $N^{-1/2}$ (see the lemma hereafter).
This approach is much more direct and there is no need to construct moment matching matrices. 
It provides a completely dynamical proof of Theorem  \ref{thm:main} 
for bulk eigenvectors.

We remark that although this proof is very simple, the fact that the Dyson Brownian motion preserves the detailed behaviour of eigenvalues and eigenvectors is surprising and even contradictory. Consider for example the eigenvalue flow. It was proved that this spectral dynamics take very general initial data to local equilibrium for any time $t\gtrsim N^{-1}$. So how can we prove that the changes of the eigenvalues up to time $N^{-1/2}$ is less than the accuracy $N^{-1}$? The answer is that we only prove the preservation of the Dyson Brownian motion for matrix models. In other words, the matrix structure gives this preservation of the local structure.

We start with the following matrix stochastic differential equation which  is an Ornstein-Uhlenbeck version of the Dyson Brownian motion.
Let $H_t=(h_{ij}(t))$ be a symmetric $N\times N$ matrix. The 
dynamics of the matrix entries are given by the 
stochastic differential equations 
\be\label{eqn:generalDBM}
  \rd h_{ij} (t)  = \frac {\rd B_{ij}  (t) } { \sqrt N}  - \frac{1}{2 N s_{ij}} h_{ij} (t) \rd t,
\ee
where  $B$ is symmetric with  $(B_{ij})_{i\leq j}$   a family of independent Brownian motions. 
The parameter   $s_{ij} > 0$ can take any positive values, but in this paper, 
we choose $s_{ij}$ to be the variance of  $h_{ij}(0)$.
Clearly,  for any $t\geq 0$ we have $\E(h_{ij}(t)^2)=s_{ij}$ and thus the variance of the matrix element is preserved
in this flow. 
We will call this system of stochastic differential equations \eqref{eqn:generalDBM} a generalized Dyson Brownian motion. 
For this flow, the following continuity estimate  holds.

\begin{lemma}\label{DBMcont}
Suppose that  we have $c/N  \le s_{ij} \le  C /N$ 
for some fixed constants $c$ and $C$, uniformly in $i$ and $j$.  Denote  $ \partial_{ij} =  \partial_{h_{ij}}$.  Suppose that $F$ is a smooth function of the matrix elements $(h_{ij})_{i\leq j}$ satisfying
\be\label{AI}
 \sup_{0 \le s \le t, i\leq j, \b  \theta } \E \left((N^{3/2} |h_{ij}(s)^3| + \sqrt N  |h_{ij}(s)|) \big | \partial_{{ij}}^3 F  ( \b \theta H_s  )\big |\right) \le M,
\ee
where $(\b \theta H)_{ij} = \theta_{ij} h_{ij}$, $\theta_{k\ell} = 1$ unless $\{ k, \ell\} = \{ i, j \}$ and $0\le \theta_{ij} \le 1$.
Then
$$
 \E F\left(H_t\right)  -  \E F\left(H_0\right)
 = \OO( t {N^{1/2}})M.
$$ 
\end{lemma} 

\begin{proof} 
By It{\^o}'s formula, we have 
$$
\partial_t \E F  \left( H_t\right) =  -\frac {1 }{2 N}  \sum_{i\leq j}\left(\frac{1}{s_{ij}}  \E \left(h_{ij}(t)   \partial_{{ij}} F(H_t)\right)  
-\E \left(\partial_{{ij}}^2  F(H_t)\right)\right).
$$
A Taylor expansion yields
\begin{align*}
\E  \left(h_{ij}(t) \partial_{{ij}} F (H_t)\right) &= \E  h_{ij}(t) \partial_{{ij}} F_{h_{ij}(t) = 0}
 +\E  \left(h_{ij}(t)^2 \partial_{ij}^2 F_{h_{ij}(t) = 0}\right) +  \OO\left(\sup_{\b \theta}\E   \left(|h_{ij}(t)^3  \partial_{{ij}}^3 F  ( \b \theta H_t)|\right) \right)\\
 & 
 =s_{ij} \E\left(\partial_{{ij}}^2 F_{h_{ij}(t) = 0}\right) +\OO\left(\sup_{\b \theta}\E   \left(|h_{ij}(t)^3  \partial_{{ij}}^3 F  ( \b \theta H_t)|\right)\right), \\
 \E\left(\partial_{{ij}}^2 F(H_t)\right)&=\E \left(\partial_{{ij}}^2 F_{h_{ij}(t)= 0}  \right)
  + \OO\left(\sup_{\b \theta}\E\left(|h_{ij}(t)(\partial_{{ij}}^3 F  (\b \theta H )|)\right)\right).
\end{align*} 
Together with the condition $c/N\leq s_{ij}\leq C/N$, we have
$$
\partial_t \E F(H_t)=  N^{1/2}\OO\left(\sup_{i\leq j,\b\theta} \E  ( N^{3/2} |h_{ij}(t)^3| + N^{1/2}  |h_{ij}(t)|)
|\partial_{{ij}}^3 F  ( \b \theta H_t)|\right).
$$
Integration over time finishes the proof. 
\end{proof}

The previous lemma implies the following eigenvalues and eigenvectors continuity estimate for the dynamics (\ref{eqn:generalDBM}).

\begin{corollary}\label{cor:continuity}
Let $\alpha>0$ be arbitrarily small, $\delta\in(0,1/2)$ and $t=N^{-1+\delta}$. 
Denote by  $H_t$ the solution of (\ref{eqn:generalDBM}) 
with   a symmetric generalized Wigner matrix $H_0$  as the initial condition. 
Let $\mu_t$ be the law of $H_t$. 
Let $m$ be 
any positive integer and $\Theta:\RR^{2m}\to\RR$ be a smooth function satisfying
\begin{equation}\label{eqn:Theta}
\sup_{k\in\llbracket 0,5\rrbracket, x\in\mathbb{R}}|\Theta^{(k)}(x)|(1+|x|)^{-C}<\infty
\end{equation}
for some $C>0$. 
Denote by $(u_1(t),\dots,u_N(t))$ the eigenvectors of $H_t$ associated with the eigenvalues $\la_1(t)\leq \dots\leq \la_N(t)$.
Then there exists $\e>0$ (depending only on $\Theta,\delta$ and $\alpha$) such that, for large enough $N$,
$$
\sup_{I\subset\llbracket \alpha N,(1-\alpha) N\rrbracket, |I|=m,|\bq|=1}\left|
(\E^{\mu_t}-\E^{\mu_0})\Theta\left(
(N(\la_k-\gamma_k),N \langle\bq,u_k\rangle^2)_{k\in I}
\right)
\right|\leq N^{-\e}.
$$ 
\end{corollary}

\begin{proof}
One may try to apply Lemma \ref{DBMcont} directly for $F(H)=(\bla,\bu)$, but the third derivative of this function seems hard to bound.  Instead, we can
prove the continuity estimate when $F$ is a product of Green functions of $H$, which in turn  implies the continuity estimate for eigenvalues and eigenvectors. In the following, the fact that (i) and (ii) imply (\ref{eqn:result}) relies on classical techniques \cite{KnoYin2011}. The crucial condition is (i), i.e., comparison of Green functions up to some scale smaller than microscopic, $\eta=N^{-1-\e}$. In Section \ref{sec:Proof} such a comparison was shown by moment matching. Hereafter, Lemma \ref{DBMcont} allows to prove this Green function comparison by a dynamic approach.

Let $\bv$ and $\bw$ refer to two generalized Wigner ensembles. 
Consider the following statements.
\begin{enumerate}[(i)]
\item Green functions comparison up to a very small scale. 
For any $\kappa>0$ there exists $\xi,\e>0$ such that
for any $N^{-1-\xi}<\eta<1$ and  any smooth function $F$ with polynomial growth,  we have
$$
\sup_{|\bq|=1,E_1,\dots,E_m\in(-2+\kappa,2-\kappa)^m}\left|
(\E^{\bv}-\E^{\bw})F\left((\langle \bq,G(z_k)\bq\rangle)_{k=1}^m\right)
\right|\leq C N^{-\e}\left(\frac{1}{N\eta}+\frac{1}{\sqrt{N\eta}}\right),
$$ 
for some $C=C(\kappa,F)>0$. Here $z_k=E_k+\ii\eta$.

\item Level repulsion estimate. For both ensembles $\bv$ and $\bw$ and for any $\kappa>0$ the following holds. There exists $\xi_0>0$ such that for any $0<\xi<\xi_0$ there exists $\delta>0$ satisfying
$$
\P\left(|\{\la_i\in[E-N^{-1-\xi},E+N^{-1-\xi}]\}|\geq 2\right)\leq N^{-\xi-\delta},
$$
for any $E\in(-2+\kappa,2-\kappa)$. Here the probability measure can be either  the ensemble $\bv$ or $\bw$.
\end{enumerate}

From Section 5 in \cite{KnoYin2011}, if (i) and (ii) hold then
for any $\al>0$ and $\Theta$ satisffying (\ref{eqn:Theta})
there exists $\e>0$ such that for large enough $N$ we have
\begin{equation}\label{eqn:result}
\sup_{I\subset\llbracket \alpha N,(1-\alpha) N\rrbracket, |I|=m,|\bq|=1}\left|
(\E^{\bv}-\E^{\bw})\Theta\left(
(N(\la_k-\gamma_k),N \langle\bq,u_k\rangle^2)_{k\in I}
\right)
\right|\leq N^{-\e}.
\end{equation}
The level repulsion condition condition (ii) was proved in the generalized Wigner context \cite[equation (5.32)]{ErdYau2012singlegap}. We therefore only need to check the main assumption (i), which is a consequence of Lemma \ref{DBMcont} and the isotropic local semicircle law, Theorem \ref{thm:ILSC}.
Indeed, we need to find a good bound $M$ in (\ref{AI}) for a function $F$ of type given in (i).
For simplicity we only consider the case
$$
F(H)=\langle\bq,G(z)\bq\rangle,  
$$
where $z=E+\ii\eta$ with $N^{-1-\xi}<\eta<1$ and $-2+\kappa<E<2-\kappa$.
The general case
$$
F(H)=\langle\bq_1,G(z_1)\bq_1\rangle \ldots \langle\bq_k ,G(z_k)\bq_k\rangle
$$
is  analogous.   We have
$$
\partial_{ij}^3 \langle \bq, G(z)  \bq\rangle =-\sum_{a, b} \sum_{\b \alpha, \b \beta} q_a G(z)_{a \alpha_1} G(z)_{\beta_1, \alpha_2} G(z)_{\beta_2, \alpha_3} G(z)_{\beta_3, b}  \, q_b
$$
where $\{\alpha_k, \beta_k\} = \{i, j\}$ or $\{j, i\}$. From the isotropic local semicircle law (\ref{ISSC_estimate}) the following four expressions 
$$
\sum_{a}  q_a \, G(z)_{a \alpha_1}, \; G(z)_{\beta_1, \alpha_2},  \; 
G(z)_{\beta_2, \alpha_3},  \;   \sum_b G(z)_{\beta_3, b} \, q_b
$$
are  bounded by $N^{2\xi}((N\eta)^{-1}+(N\eta)^{-1/2})$ with very high probability provided that $N^{-1+ \xi}\leq \eta\leq 1$.  
Moreover, by a dyadic argument explained in \cite{ErdYauYin2012Univ} Section 8, we have for any $y\leq \eta$
$$
|\langle\bq, G(E+\ii y)\bq\rangle|\leq C\log N\frac{\eta}{y}\Im\langle\bq, G(E+\ii\eta)\bq\rangle.
$$
Consequently, we proved that uniformly in $E\in(-2+\kappa,2-\kappa)$, $N^{-1-\xi}\leq \eta\leq 1$, we have
$$
\partial_{ij}^3 \langle \bq, G(E+\ii\eta)  \bq\rangle=\OO(N^{5\xi}(N\eta)^{-1}+(N\eta)^{-1/2}))
$$
with very high probability. 
The hypothesis (\ref{AI}) therefore holds with $M=C(\e) N^{5\xi}((N\eta)^{-1}+(N\eta)^{-1/2})$. As $\xi$ is arbitrarily small, Lemma \ref{DBMcont}  proves that for any $\delta\in(0,1/2)$ and $t=N^{-1+\delta}$ there exists some $\e>0$ 
with 
$$
|\E F(H_t)-\E F(H_0)|\leq N^{-\e}((N\eta)^{-1}+(N\eta)^{-1/2}).
$$
Thus  assumption (i) holds and the Corollary is proved.
\end{proof}

To complete the proof of Theorem \ref{thm:main} for bulk eigenvectors by a dynamical approach, we proceed as follows.
Let $H_0$ be a generalized Wigner matrix.  For $\delta\in(0,1/2)$ and $t=N^{-1+\delta}$, let
$H_t$ be the solution of (\ref{eqn:generalDBM}) at time $t$. On the one hand, from Corollary \ref{cor:continuity} we have
$$
\sup_{I\subset\llbracket \al N,(1-\al)N\rrbracket,|I|=m, |\bq|=1}
\left|\E \left(P\left((N\langle \bq,u_k(t) \rangle^2)_{k\in I}\right)  \right)-\E \left(P\left((N\langle \bq,u_k\rangle^2)_{k\in I}\right)  \right)\right|\leq N^{-\e}.
$$
On the other hand, the entry $h_{ij}(t)$ of $H_t$ is distributed as
\begin{equation}\label{eqn:hij}
e^{-\frac{t}{2Ns_{ij}}}h_{ij}(0)+\left(s_{ij}\left(1-e^{-\frac{t}{N s_{ij}}}\right)\right)^{1/2}\mathscr{N}^{(ij)}
\end{equation}
where $(\mathscr{N}^{(ij)})_{i\leq j}$ are independent standard Gaussian random variables.
 For any $\nu<\frac{1}{2}\inf_{i,j}s_{ij}\left(1-e^{-\frac{t}{N s_{ij}}}\right)$,
let $W_0$ be a random matrix with entry $(W_0)_{ij}$ distributed as
\begin{align*}
&e^{-\frac{t}{2Ns_{ij}}}h_{ij}(0)+\left(s_{ij}\left(1-e^{-\frac{t}{N s_{ij}}}\right)-\nu\right)^{1/2}\mathscr{N}^{(ij)}_1\ \ \ \mbox{if $i\neq j$},\\
&e^{-\frac{t}{2Ns_{ij}}}h_{ij}(0)+\left(s_{ij}\left(1-e^{-\frac{t}{N s_{ij}}}\right)-2\nu\right)^{1/2}\mathscr{N}^{(ij)}_1\ \ \ \mbox{if $i=j$},
\end{align*}
where $(\mathscr{N}_1^{(ij)})_{i\leq j}$ are independent standard Gaussian random variables, independent from $H_0$. Then 
$W_0$ is a generalized Wigner matrix modulo scaling: for any $i$ we have $\sum_j\var(W_0)_{ij}=1-(N+1)\nu$. Moreover
from (\ref{eqn:hij}) $h_{ij}(t)$ is distributed as 
\begin{align*}
&(W_0)_{ij}+
\nu^{1/2}\mathscr{N}^{(ij)}_2\ \ \ \mbox{if $i\neq j$},\\
&(W_0)_{ij}+
(2\nu)^{1/2}\mathscr{N}^{(ij)}_2\ \ \ \mbox{if $i=j$},
\end{align*}
where $(\mathscr{N}_2^{(ij)})_{i\leq j}$ are independent standard Gaussian random variables, independent of $W_0$. 
This proves that $H_t$ is distributed as $W_{t'}$, where $(W_s)_{s\geq 0}$ satisfies ({\ref{eqn:SymSDE}) and $t'=N\nu$.
We choose $\nu=N^{-2+\xi}$ for some $\xi\in(0,1)$ 
and apply Theorem \ref{thm:maxPrincipleLoc}  to $W_{t'}$: this yields
$$
\sup_{I\subset\llbracket \al N,(1-\al)N\rrbracket,|I|=m, |\bq|=1}
\left|\E \left(P\left((N\langle \bq,u_k (t) \rangle^2)_{k\in I}\right)  \right)-\E P\left((\mathscr{N}_j^2)_{j=1}^m\right)\right|\leq N^{-\e}.
$$
We have thus proved Theorem \ref{thm:main} by a dynamic approach,  in the bulk case.

\nc
\setcounter{equation}{0}
\setcounter{theorem}{0}
\renewcommand{\theequation}{B.\arabic{equation}}
\renewcommand{\thetheorem}{B.\arabic{theorem}}
\appendix
\setcounter{secnumdepth}{0}
\section{Appendix B\ \ \ Generator of the Dyson vector flow}

\subsection{B.1\ \ \ Proof of Theorem \ref{thm:PCE}.\ }
We first consider the symmetric case.

$(a)$ For any $\e>0$, let $\tau_\e=\inf\{t\geq 0\mid |\lambda_i-\lambda_j|=\e\ \mbox{for some } i\neq j\ or |\lambda_i|=\e^{-1}\  \mbox{for some } i\}$ and $\phi_\e$
be a sooth function on $\RR$ such that $\phi_\e(x)=x^{-1}$ if $x\geq \e$. Then, as all of the following coefficients
are Lipschitz, pathwise existence and uniqueness holds for the system of stochastic differential equations
\begin{align*}
\rd\la_k&=\frac{\rd B^{(s)}_{kk}}{\sqrt{N}}+\frac{1}{N}\sum_{\ell\neq k}\phi_\e(\la_k-\la_\ell)\rd t,\\
\rd u_k&=\frac{1}{\sqrt{N}}\sum_{\ell\neq k}(\rd B^{(s)}_{k\ell})\phi_\e(\lambda_k-\lambda_\ell)u_\ell
-\frac{1}{2N}\sum_{\ell\neq k}\phi_\e(\la_k-\la_\ell)^2u_k\rd t.
\end{align*}
Consequently, if one can prove that $\tau_\e\to\infty$ almost surely as $\e\to0$ , then existence and strong uniqueness 
for the system (\ref{eqn:eigenvaluesSymmetric}), (\ref{eqn:eigenvectorsSymmetric})
easily follow.
This non-explosion nor collision result follows from Proposition 1 in \cite{RogShi1993}.
It immediately yields $\bla_t\in\Sigma_N$ for ant $t\geq 0$.

To prove that $\boldu_t\in \OO(N)$ for any $t\geq 0$, we consider the stochastc differential equations satisfied by $u_i\cdot u_j$, $1\leq i\leq j\leq N$. It\^o's formula yields
\begin{align*}
\rd(u_i\cdot u_j)&=\frac{1}{\sqrt{N}}\sum_{k\not\in\{i,j\}}\left(\frac{\rd B^{(s)}_{jk}}{\la_j-\la_k}u_i\cdot u_k+\frac{\rd B^{(s)}_{ik}}{\la_i-\la_k}u_j\cdot u_k\right)+\frac{1}{\sqrt{N}}\frac{\rd B^{(s)}_{ji}}{\la_j-\la_i}(|u_i|^2-|u_j|^2)\\
& \ \ -\frac{1}{2N}\left(\sum_{k\neq j}\frac{1}{(\la_j-\la_k)^2}+\sum_{k\neq i}\frac{1}{(\la_i-\la_k)^2}+
\frac{1}{(\la_i-\la_j)^2}\right)u_i\cdot u_j\rd t,\ i\neq j,\\
\rd(|u_i|^2)&=\frac{2}{\sqrt{N}}\sum_{k\neq i}\frac{\rd B^{(s)}_{ik}}{\la_i-\la_k}u_i\cdot u_k+\frac{1}{N}
\sum_{k\neq i}\frac{|u_k|^2-|u_i|^2}{(\lambda_i-\lambda_k)^2}.
\end{align*}
For the same reason as previously, existence and strong uniqueness hold for the above system, 
and $u_i\cdot u_j=0$ ($i\neq j$), $|u_i|^2=1$ is an obvious solution (remember that $\boldu_0\in \OO(N)$), which completes the proof.

(b) Let $\tilde H^{(s)}_t=\boldu_t\bla_t\boldu_t^*$. On the one hand, It\^o's formula gives
\begin{equation}\label{eqn:evol}
\rd \tilde H^{(s)}_{km}=(\boldu\bla(\rd\boldu)^*+\boldu(\rd \bla)\boldu^*+(\rd\boldu)\bla\boldu^*)_{km}+
\sum_{\ell\neq s}\frac{1}{N}\frac{\la_\ell}{(\la_s-\la_\ell)^2}u_s(k)u_s(m)\rd t.
\end{equation}
On the other hand, the evolution equations for $\bla$ and $\boldu$ is
\begin{align*}
\rd\bla&=\rd M_{\bla}+\rd D_{\bla}&(\rd M_{\bla})_{ij}=\frac{\rd B^{(s)_{ii}}}{\sqrt{N}}\mathds{1}_{i=j},\ 
(\rd D_{\bla})_{ij}=\left(\frac{1}{2}\sum_{\ell\neq i}\frac{1}{\la_i-\la_\ell}\right)\rd t\1_{i=j},\\
\rd \boldu&=\boldu(\rd M_{\boldu}+\rd D_{\boldu})&(\rd M_{\boldu})_{ij}=\frac{1}{\sqrt{N}}\frac{\rd B^{(s)}_{ij}}{\lambda_i-\lambda_j}\1_{i\neq j},\ 
(\rd D_{\boldu})_{ij}=-\frac{1}{2N}\sum_{\ell\neq i}\frac{\rd t}{(\la_i-\la_j)^2}\1_{i\neq j}.
\end{align*}
Consequently, after defining the diagonal matrix process $D$ by
$$
(\rd D)_{ij}=\frac{1}{N}\sum_{\ell\neq i}\frac{\la_\ell}{(\la_\ell-\la_i)^2}\rd t\1_{i=j},
$$
the equation (\ref{eqn:evol}) can be written
$$
\rd \tilde S=\boldu(\bla(\rd M_{\boldu})^*+(\rd M_{\boldu})\bla+\rd M_{\bla})\boldu^*+
\boldu(\bla(\rd D_{\boldu})^*+(\rd D_{\boldu})\bla+\rd D_{\bla}+\rd D)\boldu^*.
$$
We have $\bla(\rd M_{\boldu})^*+(\rd M_{\boldu})\bla+\rd M_{\bla}=\frac{1}{\sqrt{N}}\rd B^{(s)}$ and
$\bla(\rd D_{\boldu})^*+(\rd D_{\boldu})\bla+\rd D_{\bla}+\rd D=0$, so
$$
\rd \tilde S=\frac{1}{\sqrt{N}}\boldu(\rd B^{(s)})\boldu^*.
$$
As $\boldu_t\in\OO(N)$ almost surely for any $t\geq 0$, by L\'evy's criterion, the process $M$ defined by $M_0=0$ and $\rd M_t=\boldu(\rd B^{(s)})\boldu^*$ is 
a symmetric Dyson Brownian motion. This concludes the proof: $(\tilde H^{(s)}_t)_{t\geq 0}$ and  
$(H^{(s)}_t)_{t\geq 0}$ have the same law, as they are both solution of the same stochastic differential equation, for which weak uniqueness holds.

(c) Existence and strong uniqueness for (\ref{eqn:eigenvaluesSymmetric}) has a proof strictly identical to (a).
For a given continuous trajectory $(\bla_t)_{t\geq 0}\subset\Sigma_N$, existence and strong uniqueness for (\ref{eqn:eigenvectorsSymmetric}) is elementary, because $\sup_{t\in[0,T],i\neq j}|\lambda_i-\lambda_j|^{-1}<\infty$ and the coefficients are Lipschitz for any given $t\in[0,T]$.

Let $\bla'$ be the solution of (\ref{eqn:eigenvaluesSymmetric}), and
$(\boldu^{(\bla')}_t)_{t\geq 0}$ be the solution of (\ref{eqn:eigenvectorsSymmetric}) for given $\bla'$.
If the initial conditions match, we have
\begin{equation}\label{eqn:trivial}
\P((\bla'_t,\boldu^{(\bla')}_t)=(\bla_t,\boldu_t)\ \mbox{for all}\ t\geq 0)=1,
\end{equation}
because $(\bla'_t,\boldu^{(\bla')}_t)$ is a solution of the system of stochastic differential equations
 (\ref{eqn:eigenvaluesSymmetric},\ref{eqn:eigenvectorsSymmetric}) for which strong uniqueness holds. 
Equations (\ref{eqn:trivial}) together with $(b)$ yields
\begin{equation}\label{eqn:test}
\E(F((H^{(s)}_t)_{0\leq t\leq T}))=\E(F((\boldu^{(\bla')}_t\bla'_t(\boldu^{(\bla')}_t)^*)_{0\leq t\leq T})).
\end{equation}
As strong uniqueness holds, $(\bla'_t)_{0\leq t\leq T}$ is a measurable function (called $f$) of 
$((B^{(s)}_{ii})_{0\leq t\leq T})_{i=1}^N$, and $(\boldu^{\bla'}_t)_{0\leq t\leq T}$ is a measurable function of
$((B^{(s)}_{ij})_{0\leq t\leq T})_{i<j}$ and $(\bla'_t)_{0\leq t\leq T}$ (called $g$). We therefore have (for some Wiener measures $\omega_1,\omega_2$) for any bounded continuous function $G$
\begin{multline*}
\E(G((\bla'_t)_{0\leq t\leq T},(\boldu^{(\bla')}_t)_{0\leq t\leq T}))=\iint\rd\omega_1(B_1)\rd\omega_2(B_2)G(f(B_1),g(f(B_1),B_2))
\\=\iint\rd \nu_T(\bla)\rd\omega_2(B_2)G(\bla,g(\bla,B_2))
=\iint\rd \nu_T(\bla)\rd\mu_T(\boldu^{(\bla)})G(\bla,\boldu^{(\bla)}).
\end{multline*}
Together with (\ref{eqn:test}), this concludes the proof. We used the independence of the diagonal of $B^{(s)}$ with the other entries in the first equality above.

\subsection{B.2\ \ \ Proof of Lemma \ref{lem:generator}.\ }
We consider the Hermitian setting, the symmetric one being slightly easier. 
Let $f$ be a smooth function of the matrix entries, $u_{k}(\alpha)=x_{k\alpha}+\ii y_{k\alpha}$, $1\leq k,\alpha\leq N$. We denote 
$\langle\cdot,\cdot\rangle'=(\rd/\rd t)\langle\cdot,\cdot\rangle$.
It\^o's formula yields
\begin{align*}
\frac{\rd}{\rd t} \E(f)&=\E(\mbox{(I)}+\mbox{(II)}+\mbox{(III)}),\\
\mbox{(I)}&=\sum_{k,\al}(-\frac{1}{2}\sum_{\ell\neq k}c_{k\ell})(x_{k\al}\partial_{x_{k\al}}+y_{k\al}\partial_{y_{k\al}})f,\\
\mbox{(II)}&=\frac{1}{2}\sum_{k,\al,\beta}\left(\langle x_{k\al},x_{k\beta}\rangle'\partial_{x_{k\al}x_{k\beta}}
+\langle y_{k\al},y_{k\beta}\rangle'\partial_{y_{k\al}y_{k\beta}}
+\langle x_{k\al},y_{k\beta}\rangle'\partial_{x_{k\al}y_{k\beta}}
+\langle y_{k\al},x_{k\beta}\rangle'\partial_{y_{k\al}x_{k\beta}}\right)f,\\
\mbox{(III)}&=
\sum_{k<\ell,\al,\beta}\left(\langle x_{k\al},x_{\ell\beta}\rangle'\partial_{x_{k\al}x_{\ell\beta}}
+\langle y_{k\al},y_{\ell\beta}\rangle'\partial_{y_{k\al}y_{\ell\beta}}
+\langle x_{k\al},y_{\ell\beta}\rangle'\partial_{x_{k\al}y_{\ell\beta}}
+\langle y_{k\al},x_{\ell\beta}\rangle'\partial_{y_{k\al}x_{\ell\beta}}\right)f.
\end{align*}
Substituting $\partial_x=\partial_u+\partial_{\overline{u}}$ and $\partial_y=\ii(\partial_u-\partial_{\overline{u}})$ gives
\begin{align*}
\mbox{(I)}&=-\frac{1}{2}\sum_{k<\ell,\al}c_{k\ell}\left(u_k(\al)\partial_{u_k(\al)}+\overline{u}_k(\al)\partial_{\overline{u}_k(\al)}+
u_\ell(\al)\partial_{u_\ell(\al)}+\overline{u}_\ell(\al)\partial_{\overline{u}_\ell(\al)}\right)f\\
&=-\frac{1}{2}\sum_{k<\ell}c_{k\ell}\left(u_k\partial_{u_k}+\overline{u}_k\partial_{\overline{u}_k}+
u_\ell\partial_{u_\ell}+\overline{u}_\ell\partial_{\overline{u}_\ell}\right)f.
\end{align*}
Moreover, from the stochastic differential equation (\ref{eqn:eigenvectorsHermitian}), we obtain
$$
\langle x_{k\al},x_{k\beta}\rangle'=\langle y_{k\al},y_{k\beta}\rangle'=\frac{1}{2}\sum_{\ell\neq k}c_{k\ell}\Re(u_\ell(\al)\overline{u}_\ell(\beta)),\langle x_{k\al},y_{k\beta}\rangle'=-\langle y_{k\al},x_{k\beta}\rangle'=
\frac{1}{2}\sum_{\ell\neq k}c_{k\ell}\Im(\overline{u}_\ell(\al)u_\ell(\beta)).
$$
It implies that
\begin{align*}
\mbox{(II)}&=\frac{1}{2}\sum_{k<\ell,\al,\beta}c_{k\ell}\left(
u_\ell(\al)\overline{u}_\ell(\beta)\partial_{u_k(\al)\overline{u}_k(\beta)}+
\overline{u}_\ell(\al)u_\ell(\beta)\partial_{\overline{u}_k(\al)u_k(\beta)}+
u_k(\al)\overline{u}_k(\beta)\partial_{u_\ell(\al)\overline{u}_\ell(\beta)}+
\overline{u}_k(\al)u_k(\beta)\partial_{\overline{u}_\ell(\al)u_\ell(\beta)}\right)f\\
&=
\frac{1}{2}\sum_{k<\ell}c_{k\ell}\left(
u_\ell\partial_{u_k}\overline{u}_\ell\partial_{\overline{u}_k}+
\overline{u}_\ell\partial_{\overline{u}_k}u_\ell\partial_{u_k}+
\overline{u}_k\partial_{\overline{u}_\ell}u_k\partial_{u_\ell}+
u_k\partial_{u_\ell}\overline{u}_k\partial_{\overline{u}_\ell}
\right)f.
\end{align*}
Finally, concerning the term (III), a calculation yields, for $k\neq \ell$,
$$
\langle x_{k\al},x_{\ell\beta}\rangle'=-\langle y_{k\al},y_{\ell\beta}\rangle'=-\frac{1}{2}c_{k\ell}\Re(u_{\ell}(\al)u_{k}(\beta)),
\langle x_{k\al},y_{\ell\beta}\rangle'=\langle x_{\ell\beta},y_{k\alpha}\rangle'=-\frac{1}{2}c_{k\ell}\Im(u_{\ell}(\al)u_{k}(\beta)).
$$
We therefore get
\begin{align*}
\mbox{(III)}&=-\frac{1}{2}\sum_{k<\ell,\al,\beta}c_{k\ell}\left(
u_\ell(\al)u_k(\beta)\partial_{u_k(\al)u_\ell(\beta)}+
\overline{u}_\ell(\al)\overline{u}_k(\beta)\partial_{\overline{u}_k(\al)\overline{u}_\ell(\beta)}+
u_k(\al)u_\ell(\beta)\partial_{u_\ell(\al)u_k(\beta)}+
\overline{u}_k(\al)\overline{u}_\ell(\beta)\partial_{\overline{u}_\ell(\al)\overline{u}_k(\beta)}
\right)f\\
&=
-\frac{1}{2}\sum_{k<\ell}c_{k\ell}
\left(
u_\ell\partial_{u_k}u_k\partial_{u_\ell}
-u_\ell\partial_{u_\ell}
+
\overline{u}_\ell\partial_{\overline{u}_k}\overline{u}_k\partial_{\overline{u}_\ell}
-\overline{u}_\ell\partial_{\overline{u}_\ell}+
u_k\partial_{u_\ell}u_\ell\partial_{u_k}
-u_k\partial_{u_k}
+
\overline{u}_k\partial_{\overline{u}_\ell}\overline{u}_\ell\partial_{\overline{u}_k}
-\overline{u}_k\partial_{\overline{u}_k}
\right)f.
\end{align*}
Gathering our estimates for (I), (II) and (III) yields
$$
\frac{\rd}{\rd t}\E(f)=
\frac{1}{2}\sum_{k<\ell}c_{k\ell}\E\left(\left(\left(u_k\partial_{u_\ell}-\overline{u}_\ell\partial_{\overline{u}_k}\right)
\left(\overline{u}_k\partial_{\overline{u}_\ell}-u_\ell\partial_{u_k}\right)+
\left(\overline{u}_k\partial_{\overline{u}_\ell}-u_\ell\partial_{u_k}\right)
\left(u_k\partial_{u_\ell}-\overline{u}_\ell\partial_{\overline{u}_k}\right)
\right)f\right).
$$
This completes the proof.

\setcounter{equation}{0}
\setcounter{theorem}{0}
\renewcommand{\theequation}{C.\arabic{equation}}
\renewcommand{\thetheorem}{C.\arabic{theorem}}
\appendix
\setcounter{secnumdepth}{0}
\section{Appendix C\ \ \ Covariance matrices}

Because of motivations in statistics, we will only define the eigenvector moment flow for  real-valued covariance matrices. 
The eigenvector dynamics were already considered in   \cite{Bru1989}. The normalization
constants follow our convention and are different  from \cite{Bru1989}.

Let $B$ be a $M\times N$ real matrix Brownian motion: $B_{ij} (1\leq i\leq N, 1\leq j\leq M)$ are independent standard Brownian motions. We define the $M\times N$ matrix $M$ by
$$
M_{t}=M_0+\frac{1}{\sqrt{N}} B_t,
$$
Then the real Wishart process is $X$ is defined by $X_t=M_t^* M_t$. In the following, we will assume for simplicity
that $M\geq N$, to avoid trivial eigenvalues of $X$ (the case $M\leq N$ admits similar results, up to trivial adjustements).
The eigenvalues and eigenvectors dynamics were given in \cite{Bru1989}, i.e. the direct analogue of definitions (\ref{eqn:eigenvaluesSymmetric}), (\ref{eqn:eigenvectorsSymmetric})
and Theorem \ref{thm:PCE} hold for the following stochastic differential equations:
\begin{align}
\rd\la_k&=2\sqrt{\la_k}\frac{\rd B^{(s)}_{kk}}{\sqrt{N}}+\left(\frac{1}{N}\sum_{\ell\neq k}\frac{\la_k+\la_\ell}{\la_k-\la_\ell}+\frac{M}{N}\right)\rd t,\notag\\
\rd u_k&=\frac{1}{\sqrt{N}}\sum_{\ell\neq k}\frac{\sqrt{\lambda_k+\lambda_\ell}}{\lambda_k-\lambda_\ell}(\rd B^{(s)}_{k\ell})u_\ell
-\frac{1}{2N}\sum_{\ell\neq k}\frac{\la_k+\la_\ell}{(\la_k-\la_\ell)^2}(\rd t)u_k.\label{dynamWishart}
\end{align}
where $B^{(s)}$ is a (symmetric) $N\times N$ Dyson Brownian motion.

After conditioning on the eigenvalues trajectory, in the same way as Lemma \ref{lem:generator}, the generator for the above eigenvector dynamics can be shown to be
$$\LL_t=\sum_{1\leq k<\ell\leq N}d_{k\ell}(t)(X_{k\ell}^{(s)})^2$$
where we used the notations (\ref{eqn:Xkl}) and
$$
d_{k\ell}(t)=\frac{\lambda_k+\lambda_\ell}{N(\la_k(t)-\la_\ell(t))^2}
$$
The definition and utility of the eigenvector moment flow for covariance matrices are then summarized as follows.

\begin{theorem}[Eigenvector moment flow for covariance matrices] Let $\bq\in\mathbb{R}^N$, $z_k=\sqrt{N}\langle\bq, u_k (t) \rangle$.
Suppose that $\boldu$ is the solution to the stochastic differential equation (\ref{dynamWishart})
and $ f_{\bla, t} (\boeta)$ is defined analogously to \eqref{feq}, where   
$\boeta$ denote  the  configuration  $ \{(i_1,  j_1),  \dots, (i_m , j_m ) \}$. Then 
$f_{\bla, t}$
satisfies the equation  
\begin{align*}
&\partial_t f_{\bla, t} =  \mathscr{B}^{(s)}(t)  f_{\bla, t},\\
&\mathscr{B}^{(s)}(t)  f(\boeta) = \sum_{i \neq j} d_{ij}(t) 2 \eta_i (1+ 2 \eta_j) \left(f(\boeta^{i, j})-f(\boeta)\right).
\end{align*}
\end{theorem}

As in the case of symmetric matrices, the above eigenvector moment flow is reversible with respect to the measure $\pi^{(s)}$ defined in (\ref{eqn:weight}). Thus analogues of Theorems \ref{thm:main}, Corollary \ref{cor:Gauss}, \ref{cor:QUE} and Theorem \ref{thm:deterministic} for covariance matrices can be proved with arguments parallel to those used in  Sections 4, 5 and 6.

\end{document}